\documentclass[preprint,11pt]{elsarticle} 
\makeatletter
\def\ps@pprintTitle{%
	\let\@oddhead\@empty
	\let\@evenhead\@empty
	\def\@oddfoot{}%
	\let\@evenfoot\@oddfoot}
\makeatother
\usepackage[left=15mm,right = 15mm,top = 15mm,bottom = 15mm]{geometry}
\usepackage[titletoc]{appendix}
\usepackage[table]{xcolor}
\usepackage{float}
\usepackage{subcaption}
\usepackage{amssymb}
\usepackage{amsmath}
\usepackage{amsthm}
\usepackage{graphics}
\usepackage{psfrag}
\usepackage{graphicx}
\usepackage{color}
\usepackage{fancyhdr}
\usepackage{algorithmic}
\usepackage[color=yellow]{todonotes}
\usepackage{appendix}
\usepackage{verbatim}
\usepackage{blindtext}
\usepackage{hyperref}
\usepackage{cleveref}
\usepackage{stmaryrd}
\usepackage{soul}
\usepackage{sidecap}
\usepackage{tikz}
\usepackage{rotating}
\usepackage{booktabs}
\usepackage{comment}
\usepackage{multirow}
\usepackage{changepage}
\usepackage{wrapfig}
\usepackage{pgfplots}
\usepackage{framed}
\usepackage{float}

\usepgfplotslibrary{fillbetween}
\pgfplotsset{compat=1.18}
\crefformat{appendix}{#2#1#3}

\definecolor{lightblue}{rgb}{0.32,0.45,0.90}
\definecolor{lightgreen}{rgb}{0.42,0.7,0.40}
\numberwithin{equation}{section}

\hypersetup{colorlinks=true,linkcolor = blue,citecolor=blue}
\numberwithin{figure}{section}

\makeatletter

\def\b{\boldsymbol}

\makeatother\def\b{\boldsymbol}

\colorlet{cgray}{gray!20!white}

\theoremstyle{definition}
\newtheorem{definition}{Definition}[section]

\newtheorem{method}{Method}[section]

\newtheorem{theorem}{Theorem}[section]
\theoremstyle{remark}
\newtheorem*{remark}{Remark}

\newtheorem{lemma}{Lemma}[section]

\begin{document}
\begin{frontmatter}
\title{Easy-to-Implement One-Step Schemes for Stochastic Integration}
\author[inst4]{J. Woodfield}
\author[inst4]{A. Lobbe}

\affiliation[inst4]{organization={Department of Mathematics, Imperial College London},
addressline={South Kensington Campus}, 
city={London},
postcode={SW7 2AZ}, 
country={United Kingdom}}

\begin{abstract}
Convenient, easy to implement stochastic integration methods are developed on the basis of abstract one-step deterministic order $p$ integration techniques. The abstraction as an arbitrary one step map allows the inspection of easy to implement stochastic exponential time differencing Runge-Kutta (SETDRK), stochastic integrating factor Runge-Kutta (SIFRK) and stochastic RK (SRK) schemes. Such schemes require minimal modifications to existing deterministic schemes and converge to the Stratonovich SDE. 

These schemes capture all symmetric terms in the Stratonovich-Taylor expansion, are order $p$ in the limit of vanishing noise, can attain at least strong order $p/2$ or $p/2-1/2$ (parity dependent) for drift commutative noise, strong order $1$ for commutative noise, and strong order $1/2$ for multidimensional non-commutative noise. Numerical convergence is demonstrated using different bases of noise for 2nd, 3rd and 4th order SETDRK, SIFRK and SRK schemes. 

\end{abstract}
\end{frontmatter}
{\footnotesize
\tableofcontents}

\section{Introduction}
\subsection{Motivation}
Stochastic differential equations (SDEs) and stochastic partial differential equations (SPDEs) are central to ensemble forecasting, uncertainty quantification, and data assimilation. Developing dedicated solvers for SDEs and SPDEs is often impractical in large-scale applications. This work studies convergence properties of deterministic one-step integrators, which converge to an SDE under an easily implementable change of variables. Abstracting to arbitrary one-step schemes, allows convergence properties for Runge-Kutta (RK), Integrating Factor Runge-Kutta (IFRK) and Exponential Time Differencing Runge-Kutta (ETDRK) methods to be established in the stochastic context. 

Whilst Integrating Factor Runge-Kutta (IFRK) methods and Runge-Kutta (RK) methods can be naturally derived for Stratonovich SDEs analogously to the deterministic case, Stochastic Exponential Time Differencing Runge-Kutta (SETDRK) schemes do not derive analogously. Here, abstract one-step map convergence theory is particularly useful for constructing SETDRK methods from existing deterministic schemes.

\subsection{Deterministic Integration Schemes}

Runge–Kutta (RK) methods are widely used one-step integrators for ODEs, offering arbitrary order accuracy and, in certain formulations, desirable properties such as strong stability preservation or symplecticity \cite{gottlieb2011strong,hairer2006geometric,butcher2016numerical}. Integrating Factor Runge–Kutta (IFRK) methods extend RK schemes by removing stiff linear terms via an integrating factor allowing larger timesteps. Exponential Time Differencing Runge–Kutta (ETDRK) methods achieve the same goal using exponential propagators acting on the linear stiff term and have shown to be efficient for solving stiff coupled systems of ordinary differential equations (ODEs) arising in diverse contexts ranging from American options under stochastic volatility to shallow water models of the atmosphere \cite{sun2023family,fu2024higher,wang2025novel,yousuf2013efficient,bhatt2016compact,zhangcharacterizing,ahmat2021compact,niegemann2007higher,beylkin1998new,clancy2013use}. For background on deterministic ETDRK methods see Cox and Matthews (2002), Kassam and Trefethen (2005) and Hochbruck and Ostermann (2010) \cite{cox2002exponential,kassam2005fourth,hochbruck2010exponential}.

\subsection{Related work and literature }
Whilst our work focuses on minimal-change Stratonovich-targeted deterministic schemes there exists a range of related and extended approaches in the SDE literature, which we briefly review below.

The Wong-Zakai theorem
\cite{eugene1965relation,wong1965convergence,ikeda1977comparison,lord2014introduction,shmatkov2006rate}
states that an ODE driven by a sufficiently smooth or piecewise-linear approximation of Brownian motion converges to the corresponding Stratonovich SDE. This idea can be leveraged to apply ODE solvers to SDEs, as investigated numerically in \cite{sengul2021performance,londono2016numerical}. However, as demonstrated by Rüemelin \cite{ruemelin1982numerical}, discrete numerical approximations of this ODE do not necessarily converge to the same Stratonovich SDE. In the Runge-Kutta setting, first-order methods may instead converge to an Itô SDE, and the general case requires error analysis.

Rüemelin in 1982 proposed a class of explicit Stochastic Runge-Kutta (SRK) schemes dependent on random variables approximating the first Stratonovich integral $J_i$, and gave algebraic conditions on the Double-Butcher-Tableau sufficient for the SRK scheme to converge with strong order $1/2$ and commutative order $1$. It was later shown in \cite{burrage1998general}, that for arbitrary dimensional SDEs with a non-commuting basis of noise, the maximum strong order of a SRK method dependent on random variables representing only the first Stratonovich integral cannot exceed $1/2$, due to the missing Lévy area correction.

In 1996 Burrage and Burrage \cite{burrage1996high} proposed a more general class of SRK schemes by allowing dependence on random variables approximating nested integration of Brownian motion. In the one-dimensional case $(d=1)$, including approximations of higher order nested integration such as $J_{0i}$ further increased the possible order of such SRK methods to $1.5$ see \cite{burrage1996high,burrage1998general}. However, as shown by the same authors in \cite{burrage1998general}, in the multidimensional-noncommutative case the inclusion of higher order nested integration still remained maximum strong order $1/2$ due to the particular structure of the SRK integrator.

In 2019 Erdogan and Lord \cite{erdougan2019new} propose stochastic exponential time differencing methods that converge to an Itô system with strong order 1/2, commutative order 1. In \cref{sec:SETDRK methods for Ito systems} we give additional literature review relating to stochastic exponential time differencing schemes for Itô-SDE's which are of a different derivation to those in this paper, some use Itô's isometry for the drift term whilst others can be derived as integrating factor Runge-Kutta methods. In \cite{komori2017weak} linear A-stability is established for explicit exponential Runge-Kutta methods for weak solutions of Itô SDE's and have been successfully implemented for an Itô driven non-linear stochastic Schrödinger equation in \cite{anton2018exponential}. 

In \cite{becker2016exponential} an exponential Wagner-Platen scheme is proposed for Itô-SPDE's these schemes are fundamentally different to SRK, SIFRK, and SETDRK schemes in this paper, and rely on iterative substitutions similar in spirit to a Wagner-Platen expansion \cite{kloeden2011numerical}, rather than a RK approach. In \cite{von2024exponential}, an exponential stochastic Runge-Kutta scheme based on the exponential Wagner-Platen type scheme is proposed for SPDE's which, under multiple operator-valued commutator relationships, could achieve order 3/2. 


In \cite{nie2006efficient} deterministic ETD schemes are introduced which treat the nonlinear term semi-implicitly, and in \cite{ta2015integration}
some of these semi-implicit schemes are extended to the stochastic context with strong order 1/2, commutative order 1. Integrating factor Runge-Kutta methods (sometimes referred to as Runge-Kutta Lawson methods after the seminal work by Lawson \cite{lawson1967generalized}), were extended to the stochastic setting in \cite{debrabant2021runge} where a general class of stochastic Runge-Kutta-Lawson schemes are presented for SDE's with multiple linear commutative drift terms.

In \cite{bhatt2017structure,besse2017high,hochbruck2010exponential} deterministic IFRK and ETDRK schemes are considered as generalised deterministic ERK methods, and this is used as a basis for even more generalised stochastic integration in \cite{yang2022new}. In \cite{yang2022new}, a general class of s-stage stochastic exponential integrators are considered following the approach \cite{burrage1996high}, where preserving conformal quadratic invariants, the generalisation to multiple linear terms, and explicit B-series expansions are discussed.

This paper considers abstract one-step integrators including SETDRK, SIFRK, and SRK methods, but only those which are related to deterministic integrators under a change of variables and those dependent on random variables representing the first Stratonovich integral, for ease of implementation, and because of the theoretical order barriers in \cite{burrage1998general}.

\subsection{Contributions}

\Cref{thm:onestep}, presents the main theoretical result. 
For any convenient stochastic one-step method generated from a deterministic one-step method of order $p$, we show that an easily implementable change of variables yields a scheme converging to a Stratonovich SDE with the following properties: strong order 1/2 (\cref{item1}); exact representation of all symmetric terms in the Stratonovich–Taylor expansion up to order $p/2$ (\cref{item2}); drift-commutative order $\operatorname{int}(p/2)$  (\cref{item3,item4}); commutative order $1$ (\cref{item6}); deterministic order $P$ \cref{item5}. 

Whilst results (\cref{item1,item6}) are known for SRK SIFRK SETDRK methods, \cite{ruemelin1982numerical,yang2022new} the proof differs. Convergence follows directly from deterministic Lie-series expansions rather than the algebraic structure of the numerical scheme. Other properties (\cref{item2,item3,item4}) are specific to the particular class of methods we consider.

The main numerical contribution is numerical comparison of SRK, SIFRK, and SETDRK methods of different orders, with different basis of noise, where the theoretical properties in \cref{thm:onestep} are observed in practice. We investigate a convenient class of stochastic exponential time differencing methods (based on the deterministic integration schemes in \cite{cox2002exponential,kassam2005fourth}) for instance SETDRK$3$, SETDRK$4$, and show they permit larger timesteps than standard SRK or SIFRK approaches, while preserving higher-order deterministic accuracy in the small-noise/large timestep regime. We translate some successful IFRK schemes into the stochastic context, for instance the eSSPIFRK$^{+}(3,3)$ from \cite{isherwood2018strong}.

\subsection{Paper Organization}

The paper is organised as follows. 
\Cref{sec:setup} defines the stochastically modified one-step map and the associated SDE, \cref{sec:convergence} provides the convergence framework. \Cref{sec:SRK methods,sec:SIFRK,sec:SETDRK} introduces stochastic Runge–Kutta (SRK) methods, stochastic integrating factor Runge–Kutta (SIFRK) methods, and stochastic exponential time differencing Runge–Kutta (SETDRK) methods. \Cref{sec:Numerical implementation details} then outlines numerical implementation details for ETDRK methods.

\Cref{sec:Numerical_results} presents numerical results. \Cref{sec:ex0} verifies convergence theory. \Cref{sec:ex1} investigates multidimensional non-commutative noise, in terms of convergence and efficiency. \Cref{sec:ex2} investigates commutative noise. \Cref{sec:ex3,sec:ex3a} investigates drift commutative noise. \Cref{sec:ex4} explores decreasing noise magnitude and errors at large timesteps, whilst \cref{sec:ex5} applies ETDRK4 to the two-dimensional incompressible Navier–Stokes equation.

The paper concludes in \cref{sec:conclusion}. The appendices provide supporting material: \cref{sec:sym_J} derives the symmetric part of nested Stratonovich integrals, \cref{app:Transfer of symmetrisation} verifies a transferring a symmetrisation operation lemma, \cref{sec: stochastic travelling wave solutions to the KdV equation} derives stochastic travelling wave solutions to the KdV equation, \cref{sec:SETDRK methods for Ito systems} gives background and literature on SETDRK methods for Itô systems. \cref{sec:CPU-Cost for a given tolerance} presents a informal CPU cost given a fixed tolerance experiment.

\begin{figure}[H]
\centering
\begin{subfigure}{.245\textwidth}
    \centering
\includegraphics[trim={70pt 10pt 40pt 45pt}, clip, width=.95\linewidth]{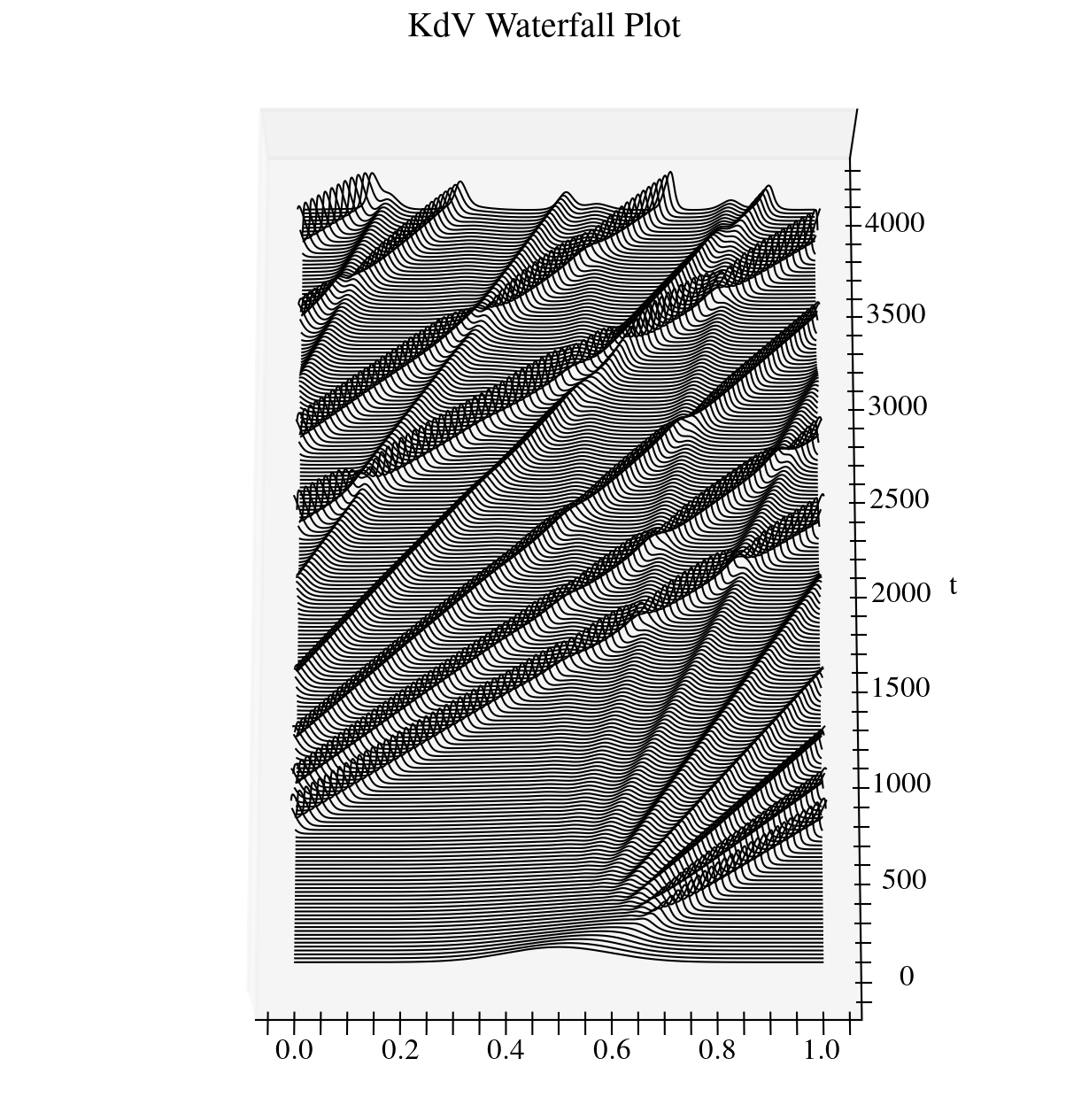}
\caption{KdV}
\label{fig:KdV_Waterfall}
\end{subfigure}
\begin{subfigure}{.245\textwidth}
\centering
\includegraphics[trim={70pt 10pt 40pt 45pt}, clip,width=.95\linewidth]{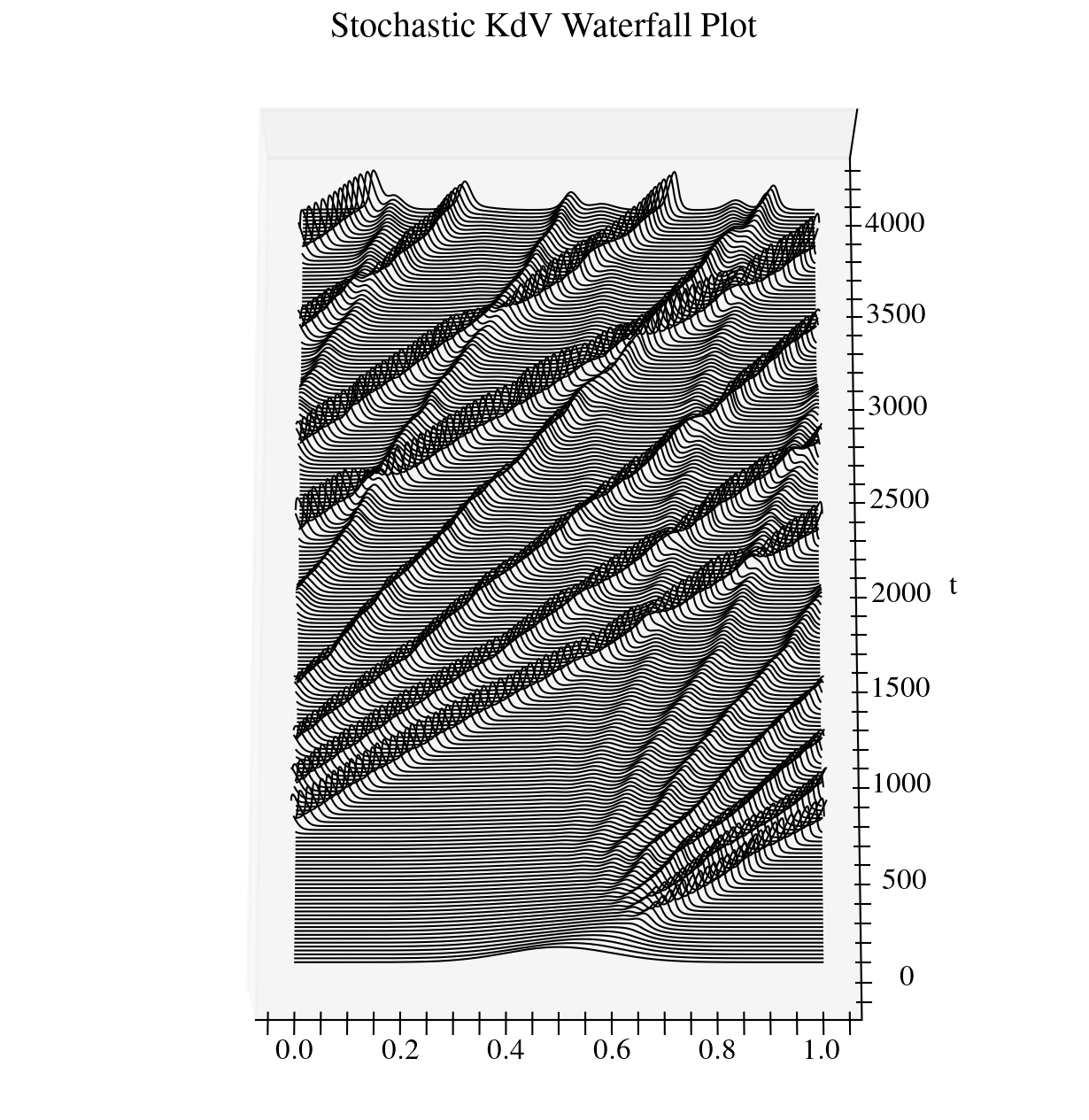}
\caption{KdV-Stochastic}
\label{fig:KdV_Waterfall_stochastic}
\end{subfigure}
\begin{subfigure}{.245\textwidth}
    \centering
\includegraphics[trim={70pt 10pt 40pt 45pt}, clip,width=.95\linewidth]{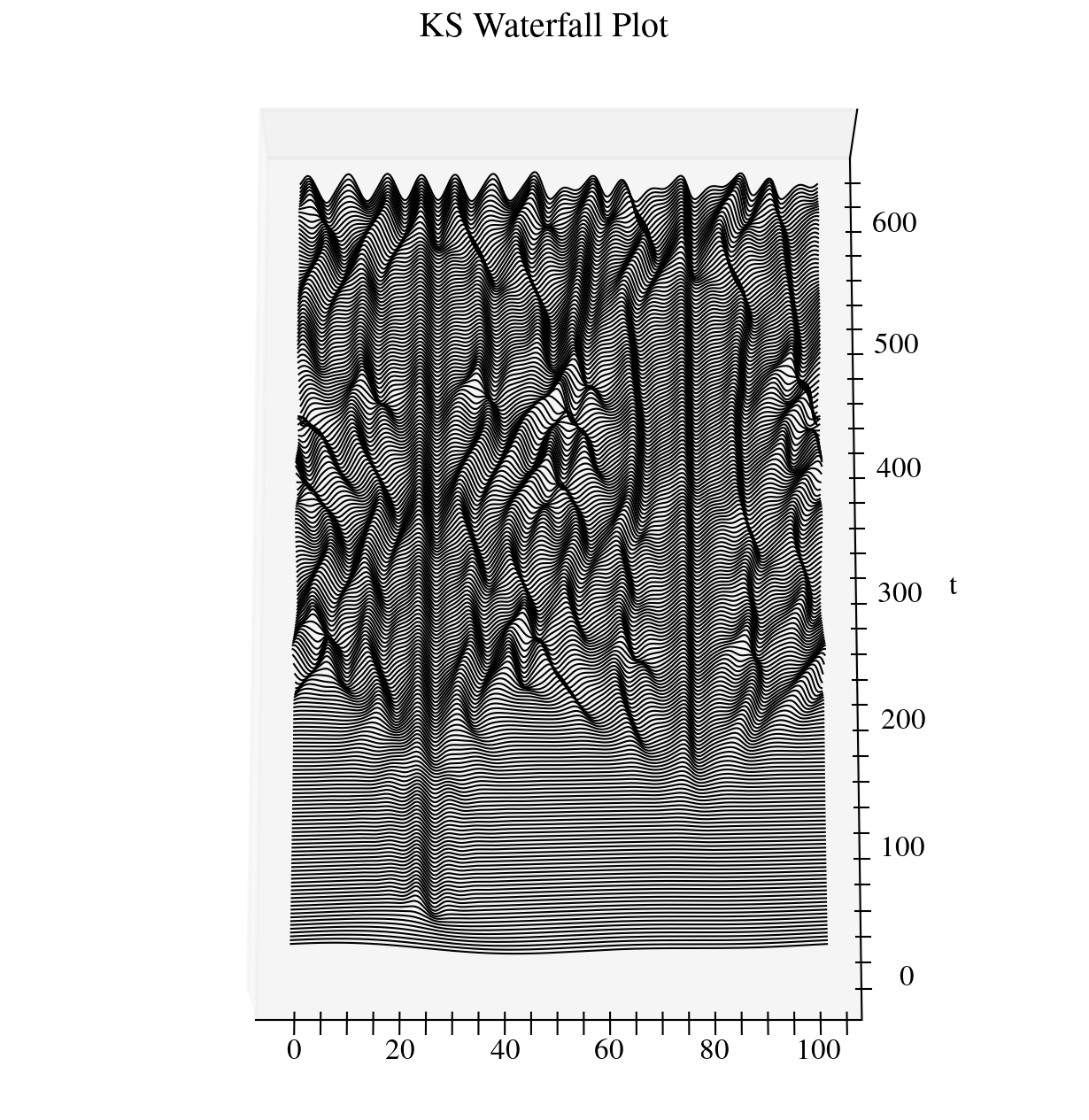}
\caption{KS}
\label{fig:KS_Waterfall}
\end{subfigure}
\begin{subfigure}{.245\textwidth}
\centering
\includegraphics[trim={70pt 10pt 40pt 45pt}, clip,width=.95\linewidth]{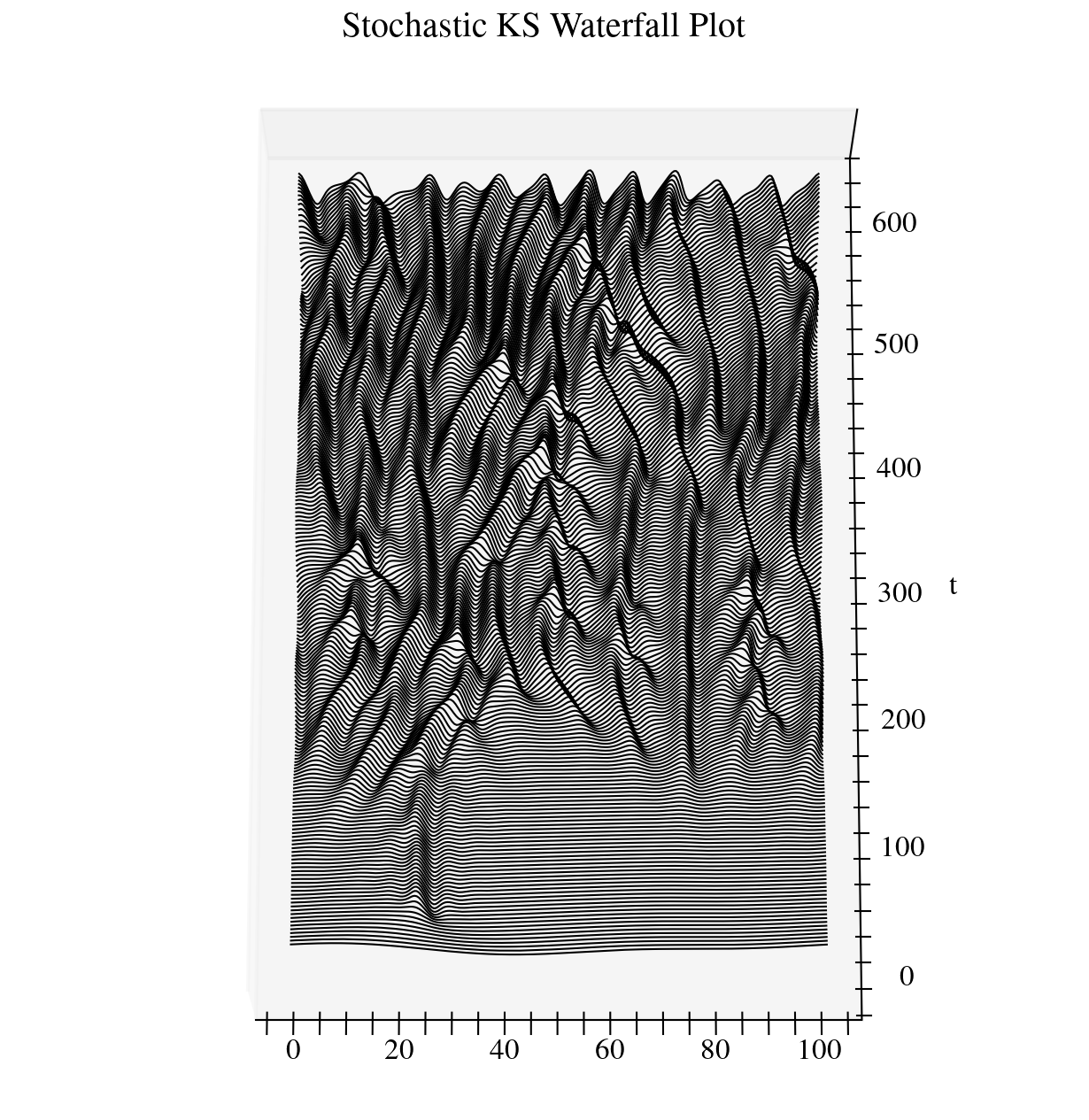}
\caption{KS-Stochastic}
\label{fig:KS_Waterfall_stochastic}
\end{subfigure}
\caption{Waterfall Plots of deterministic and stochastic KdV and KS equations under nonlinear transport noise by \cref{method:SETDRK4}.}
\label{fig:waterfall_KDV_KdV}
\end{figure}


\section{Convenient Stochastic Integration}\label{sec:con}
\subsection{Setup}\label{sec:setup}


In this work, we study the numerical integration of stochastic differential equations (SDEs) using modified deterministic solvers. We describe the setting below. Let $(\Omega, \mathcal{F}, \mathbb{P})$ be a complete probability space with filtration $\{ \mathcal{F}_t \}_{t \geq 0}$ satisfying the usual conditions. We consider the $d$-dimensional Stratonovich SDE
\begin{align}
    u_t = u_0 + \int_{0}^{t} f(s, u_s) \, ds + \sum_{m=1}^{M} \int_{0}^{t} g_m(s, u_s) \circ dW^m_s, \label{eq:Stratonovich sde}
\end{align}
where $u_0 \colon \Omega \rightarrow \mathbb{R}^d$ denotes the initial condition, assumed {$\mathcal{F}_0$-}measurable and with finite {second} moment $\mathbb{E}||u_0||_2^2< \infty$. The drift and diffusion coefficients $f$ and $\{g_m\}_{m=1}^M$ are $d$-dimensional continuous functions satisfying global Lipschitz condition and linear growth bounds on the Itô-Stratonovich corrected drift and diffusion, and $\{W^i\}_{m=1}^{M}$ are independent standard Brownian motions adapted to $\{\mathcal{F}_t\}_{t\geq 0}$ \cite{lord2014introduction,burrage2000order}.\newline

Given a deterministic one-step integration scheme
\begin{align}
    y^{n+1} = \Psi(y^n, \Delta t, f), \quad n = 0, 1, \dots, N, \label{eq:ODE_solver}
\end{align}
designed to approximate solutions of the $d$-dimensional ordinary differential equation (ODE)
\begin{align}
    \frac{dy}{dt} = f(t, y), \quad y(0) = y_0.\label{eq:ode1}
\end{align}
We define a convenient stochastic integration scheme as an abstract one-step stochastic method:
\begin{align}
    u^{n+1} = \widehat{\Psi}(u^n) := \Psi\left(u^n, \Delta t, f + \sum_{m=1}^{M} g_m \frac{\Delta W^m}{\Delta t} \right), \quad \text{where}\quad \Delta W^m := W^m(t^{n+1}) - W^m(t^n) \sim \sqrt{\Delta t}\,\mathcal{N}(0, 1),\label{eq:stochastic_abstract_method}
\end{align}
for which the deterministic local order-P expansion of $\Psi$ has stochastic remainder $
\mathcal{O}\left(\Delta t^{(P+1)/2}\right)$.  In what follows, we prove convergence results for the method \eqref{eq:stochastic_abstract_method} to solutions of the SDE \eqref{eq:Stratonovich sde}. 

\subsection{Convergence}\label{sec:convergence}
 We shall study convergence in the following mean-square sense:
\begin{definition}[Mean-square convergence] Let $u^N$ be the numerical approximation to $u(t^N)$ after $N$ steps with constant stepsize
$\Delta t$; then $u_N$ is said to converge strongly (mean square) to $u$ with strong global order $p$ if $\exists C>0$
(independent of $\Delta t$) and $\tau>0$ such that
\begin{align}
\mathbb{E}\left[||u_{N}-u(t_{N}) ||_2^2\right]^{1/2}\leq C \Delta t^{p},
\quad \forall \Delta t \leq \tau.\label{eq:rmse}
\end{align}
\end{definition}

 In order to prove convergence properties about one-step approximations, we use Milstein and Tretyakov's Fundamental Theorem of Mean Square Convergence (FTMSC) \cite{milstein2004stochastic} stated below:


\begin{theorem}[FTMSC \cite{milstein2004stochastic}] \label{thm:FTMSC}
Suppose the one-step approximation $\bar{u}_{t, x}(t+\Delta t) := \Psi(x,\Delta t,f+\frac{1}{\Delta t}\sum_{m=1}^{M}g_m \Delta W^m)$ starting from $x = u(t)$ and the exact solution $ u(t+\Delta t)$ satisfies the local error estimates:
\begin{align}
\left\lVert \mathbb{E}\left[\bar{u}_{t, x}(t+\Delta t)-u(t+\Delta t)\right]\right\rVert_2 & \leq K\left(1+||x||_2^2\right)^{1 / 2} (\Delta t)^{p_1} ,\\
{\left[\mathbb{E}\left\lVert \bar{u}_{t, x}(t+\Delta t)-u(t+\Delta t)\right\rVert^2_2\right]^{1 / 2} } & \leq K\left(1+||x||_2^2\right)^{1 / 2} (\Delta t)^{p_2},
\end{align}
for
\begin{align}
p_2 \geq \frac{1}{2},\quad p_1 \geq p_2+\frac{1}{2}, \quad t_0 \leq t \leq T-\Delta t, \quad x \in \mathbb{R}^d.
\end{align}
Then $\bar{u}_{t, x}(t+\Delta t)$ is global mean square (strong) order $p=p_2-1/2$, i.e. the global approximation $\bar{u}_{t_0,u_0}(t_k)$, obtained by iterating the one-step map $k$ times from initial condition $u_0$ satisfies:
\begin{align}
\left[\mathbb{E}\left\lVert \bar{u}_{t_0, u_0}\left(t_k\right)-u\left(t_k\right)\right\rVert_2^2\right]^{1 / 2} \leq K\left(1+\mathbb{E}\left\lVert u_0\right\rVert_2^2\right)^{1 / 2} (\Delta t)^{p_2-1 / 2},\quad k=0,1, \ldots, N, \quad \forall N
\end{align}
\end{theorem}

\begin{definition}[$L^0$, $L^i$ operators] Given the functions $f$, $\lbrace g_i \rbrace_{i=1}^M : [0,T]\times \mathbb{R}^{d} \rightarrow \mathbb{R}^{d}$, we shall define the corresponding operators $L^0,L^i$, as follows
\begin{align}
    L^0 := \frac{\partial}{\partial t}  + f^{k} \frac{\partial}{\partial u^{k}}, \quad L^i  :=g^k_i \frac{\partial}{\partial u^{k}},\quad i=1,...,M,
\end{align}
where the repeated index $k$ is summed over its range $1,...,d$. 
\end{definition}

\begin{definition}[Bracket]
We shall define the commutator of d-dimensional vector-valued functions $f,\lbrace g_i\rbrace_{i=1}^{M}$, through the action of $L^i$ in the following manner
\begin{align}
[f,g_i] := L^{0}g_i -L^if, \quad [g_i,g_j]:=L^i g_j - L^j g_i. \label{def:bracket}
\end{align}
The $l$-th component of this bracket agrees componentwise with the vector field generated by the commutator of the vector field L-operators, namely:
\begin{align}
[L^i,L^j] &:= L^i L^j -L^j L^i = g^k_{i}\frac{\partial g^l_{j}}{\partial u^k} \frac{\partial}{\partial u^l} - g^k_{j}\frac{\partial g^l_{i}}{\partial u^k} \frac{\partial}{\partial u^l} = [g_i,g_j]^{l}\frac{\partial}{\partial u^l}, \quad i,j \in \lbrace 1,...,M\rbrace.\\
[L^0,L^i] &:= L^0 L^i-L^i L^0 = \left(\frac{\partial g^l_i}{\partial t}  + f^k \frac{\partial g^l_i}{\partial u^{k}} \right)\frac{\partial }{\partial u^{l}} - g^k_i \frac{\partial f^l}{\partial u^{k}} \frac{\partial }{\partial u^{l}} = [f,g_i]^l \frac{\partial}{\partial u^l} , \quad i\in \lbrace 1,...,M\rbrace
\end{align} 
Therefore, the commutativity of $f,g_i$ under the definition of commutator defined in \cref{def:bracket} is equivalent to having commutativity of the corresponding Lie operators, since componentwise:
\begin{align}
[f,g_i]^k := [L^0,L^i]^{k}, \quad [g_i,g_j]^k := [L^i,L^j]^{k}. \label{eq:commutator_componentwise_equivalence}
\end{align}
\end{definition}

\begin{definition}[Commutative]\label{def:commutative} We shall say the basis of noise is commutative if $[g_i,g_j]=0$, or equivalently if $[L^i,L^j]=0$, in summation notation:
\begin{align}
    [g_i,g_j] := g^k_{i}(t,u) \frac{\partial g_j}{\partial u^k} - g^k_{j}(t,u) \frac{\partial g_i}{\partial u^k} = 0, \quad i,j \in \lbrace 1,...,M\rbrace.
\end{align}
\end{definition}
\begin{definition}[Drift Commutative] We shall say the basis of noise is drift commutative if in addition to satisfying \cref{def:commutative} we have $[f,g_i]=0$, or equivalently if $[L^0,L^i]=0$, in summation notation:
\begin{align}
[f,g_i]:= L^{0}g_i - L^{i}f = \frac{\partial}{\partial t}g_i + f^{k}\frac{\partial g_i}{\partial u^k} - g^k_{i}\frac{\partial f}{\partial u^k} = 0, \quad i\in \lbrace 1,..., M \rbrace.
\end{align}
\end{definition}

This can be used for the following theorem:

\begin{theorem}\label{thm:onestep} Let $
f, g_{i} \in C_b^{\infty}\left( [0,T]\times \mathbb{R}^d; \mathbb{R}^d\right)$, for $i = 1,...,M$,
such that global Lipschitz and linear growth bounds hold in the Euclidean norm and sufficient derivatives exist and are bounded. 
Let $\Psi$ (\cref{eq:ODE_solver}) denote a one-step deterministic order $P$ numerical integrator of \cref{eq:ode1} with fixed timestep $\Delta t = t^{n+1}- t^n$ for all $n$.
Let $u(t)$ denote the strong solution to the Stratonovich SDE \cref{eq:Stratonovich sde}.
Let $\widehat{\Psi}$ denote a convenient stochastic one-step integrator \cref{eq:stochastic_abstract_method}, satisfying its stochastic remainder assumption. 
Then the following properties hold:
\begin{enumerate}
    \item \label{item1} If $P \geq 2$, then the sequence $\lbrace u_n \rbrace_{n=0}^{N}$ generated by \cref{eq:stochastic_abstract_method} converges to the solution $u(t)$ of \cref{eq:Stratonovich sde} 
with global (mean square) strong order $1/2$. Meaning, the following estimate holds
\begin{align}
\big(\mathbb{E}\|u_n - u(t_n)\|^2_2\big)^{1/2} \leq C (\Delta t)^{1/2}, \quad \forall n = 0,1,...,N,\quad\forall N.
\end{align}
\item \label{item2} Method \cref{eq:stochastic_abstract_method} captures all symmetric terms in the Stratonovich-Taylor expansion up to (strong) order $P/2$, independent of whether $P$ is even or odd. 
    \item \label{item3}  If $P \geq 2$, and $P$ is even and the vector fields drift commute $[f,g_{i}]=L^0g_i- L^i f = 0$, $[g_i, g_j] = L^ig_j- L^jg_i = 0, \quad \forall i,j \in \lbrace 1,...,M\rbrace,$ then $u_{n+1}=\widehat{\Psi}( u_n)$ generated from the onestep-map \cref{eq:stochastic_abstract_method} from the initial condition $u_n$ satisfies the following \emph{local} error estimates 
    \begin{align}
    \mathbb{E}[u_{n+1} - u(t_{n+1})] &= \mathcal{O}(\Delta t^{P/2+1}),\\
    \left(\mathbb{E}[||u_{n+1} - u(t_{n+1})||_2^2]\right)^{1/2} &= \mathcal{O}(\Delta t^{P/2+1/2}),
    \end{align}
   sufficient by the FTMSC \cref{thm:FTMSC}-\cite{milstein2004stochastic} to prove the sequence $\lbrace u_n \rbrace_{n=0}^{N}$ generated by \cref{eq:stochastic_abstract_method} converges to the solution $u(t)$ of the Stratonovich SDE \cref{eq:Stratonovich sde} with strong global (mean square) order $P/2$, with the strong global error bound 
\begin{align}
    \big(\mathbb{E}[\|u_n - u(t_n)\|^2_2]\big)^{1/2} \leq C (\Delta t)^{P/2}, \quad \forall n = 0,1,...,N, \quad \forall N.
\end{align}
\item \label{item4} If $P \geq 2$, and $P$ is odd and the vector fields drift commute $[f,g_i]=0$, $[g_i, g_j] = L^ig_j- L^jg_i = 0, \quad \forall i,j \in \lbrace 1,...,M\rbrace$, then $u_{n+1}$ generated by \cref{eq:stochastic_abstract_method} from initial condition $u_n$ satisfies the following local error estimates 
    \begin{align}
    \mathbb{E}[u_{n+1} - u(t_{n+1})] &= \mathcal{O}(\Delta t^{P/2+1/2})\\
    \left(\mathbb{E}[||u_{n+1} - u(t_{n+1})||_2^2]\right)^{1/2} &= \mathcal{O}(\Delta t^{P/2+1/2})
    \end{align}
    sufficient, since the mean-square estimate may be weakened to $\mathcal{O}(\Delta t^{P/2})$, to apply the FTMSC with $p_1=P/2+1/2$, $p_2=P/2$, \cite{milstein2004stochastic} to converge to the solution $u(t)$ of the Stratonovich SDE \cref{eq:Stratonovich sde} with strong order $P/2-1/2$, with the strong global error bound
\begin{align}
    \big(\mathbb{E}[\|u_n - u(t_n)\|^2_2]\big)^{1/2} \leq C (\Delta t)^{P/2-1/2}.
\end{align}
\item \label{item5} If $g_i = 0,$ for $i=1,...,M$ we recover deterministic order $P$ convergence.
\item \label{item6} If $P \geq 2$, and $[g_i, g_j] = L^ig_j- L^jg_i = 0, \quad \forall i,j \in \lbrace 1,...,M\rbrace,$ then $u_{n+1}=\widehat{\Psi}( u_n)$ generated from the onestep-map \cref{eq:stochastic_abstract_method} from the initial condition $u_n$ satisfies the following \emph{local} error estimates 
    \begin{align}
    \mathbb{E}[u_{n+1} - u(t_{n+1})] &= \mathcal{O}(\Delta t^{2}),\\
    \left(\mathbb{E}[||u_{n+1} - u(t_{n+1})||_2^2]\right)^{1/2} &= \mathcal{O}(\Delta t^{3/2}),\label{eq:ruemelin equality}
    \end{align}
   sufficient by the FTMSC \cite{milstein2004stochastic} to prove the sequence $\lbrace u_n \rbrace_{n=0}^{N}$ generated by \cref{eq:stochastic_abstract_method} converges to the solution $u(t)$ of the Stratonovich SDE \cref{eq:Stratonovich sde} with strong global (mean square) order $1$, with the strong global error bound 
\begin{align}
    \big(\mathbb{E}[\|u_n - u(t_n)\|^2_2]\big)^{1/2} \leq C (\Delta t)^{1}, \quad \forall n = 0,1,...,N, \quad \forall N.
\end{align}
\end{enumerate}
\end{theorem}

\begin{proof}
To simplify notation 
let $g_0=f$, $t = W^0$, and for a multi-index $\b j = (j_1,...,j_n) \in \lbrace 0,1,...,M\rbrace^n$, denote nested Stratonovich integration of Brownian motion  in $[t_n, t_{n+1})$ as 
\begin{align}
J_{j_1, \ldots, j_n}:=\int_{t_n}^{t_{n+1}} \left[\int_{t_n}^{s_n} ... \left[\int_{t_n}^{s_{2}} \circ  d W^{j_1}\left(s_1\right)\right]... \circ d W^{j_{n-1}} \left(s_{n-1}\right) \right] \circ d W^{j_n}\left(s_n\right).
\end{align}

Then the $k$-th component of a Stratonovich-Taylor expansion of the analytic solution $u(t+\Delta t)$ of \cref{eq:Stratonovich sde} about local initial condition $u_n$ at time $t_n$ with timestep $\Delta t$ up to strong order $(P+1)/2$ can be written compactly as
\begin{align}
u^k(t_{n+1})&=u_n^k+\sum_{j=0}^M g^{k}_{j} J_j+\sum_{0 \leq j_1, j_2 \leq M} L^{j_1} g^{k}_{j_2} J_{j_1, j_2} +\sum_{0 \leq j_1, j_2,j_3 \leq M} L^{j_1} L^{j_2} g^{k}_{j_3} J_{j_1, j_2, j_3} +\ldots\\
&+\sum_{0 \leq j_1, j_2,... j_P \leq M} L^{j_1} L^{j_2}... L^{j_{P-1}} g^{k}_{j_P} J_{j_1,j_2,...j_P}+O(\Delta t^{\frac{P+1}{2}}).\label{eq:strat-taylor}
\end{align}
Whereas the $k$-th component of the Taylor expansion of the deterministic flow map up to order $P$ from local initial condition $u_n$ at time $t_n$ with timestep $\Delta t$ with the subsequent substitution $\Delta t f\mapsto \Delta t f+\sum_{i=1}^{M} g_i\Delta W^i$, takes the following local form
\begin{align}
\hat{u}^k_{n+1} &= u_{n}^k + \sum_{j=0}^M g^{k}_{j}J_{j} +\frac{1}{2} \sum_{0 \leq j_1, j_2 \leq M} L^{j_1} g^{k}_{j_2} J_{j_1}J_{j_2}  +\frac{1}{3!}\sum_{0 \leq j_1, j_2,j_3 \leq M} L^{j_1} L^{j_2} g^{k}_{j_3} J_{j_1}J_{j_2} J_{j_3}+ ...\\
&+\frac{1}{P!}\sum_{0 \leq j_1, j_2,... j_P \leq M} L^{j_1} L^{j_2}... L^{j_{P-1}} g^{k}_{j_P} J_{j_1}J_{j_2}... J_{j_P}+ O(\Delta t^{\frac{P+1}{2}}). \label{eq:taylor-sub}
\end{align}
Where all terms are evaluated at $(t^n,u_n)$ unless otherwise specified. Let $\mathrm{\b j}$ denote a multi-index length $n$, then the symmetric part of $J_{\mathrm{\b j}}$ is defined by summation over all permutations in the symmetric group $S_n$, then the following identity 
\begin{align}
\operatorname{Sym}\left(J_{\mathrm{\b j}}\right) = \frac{1}{n!} \sum_{\sigma \in S_n} J_{j_{\sigma(1)}, \ldots, j_{\sigma(n)}} = \frac{1}{n!}J_{j_1}J_{j_2}...J_{j_n}, \quad n =  1,2..., \label{eq:Symetric equals product}
\end{align}
(proven inductively using the shuffle property of Stratonovich integrals \cref{sec:sym_J}),
clarifies all symmetric terms, up to strong order $P/2$ in the stochastic expansion \cref{eq:strat-taylor} are captured by the stochastically modified deterministic scheme in \cref{eq:taylor-sub} proving \cref{item2}. 

For all other terms, the difference between the $p$-th order term and its symmetrisation, can be written as a sum where every term contains at least one commutator. 
This fact
will be used to prove \cref{item3,item4}, and is demonstrated below.


To begin with we will define some simplifying notation. Let $\pi^k(u):=u^k$, denote the $k$-th coordinate projection, such that $L^j \pi^k(u) =\sum_{l=1}^{d} g^l_{j}(u) \frac{\partial \pi^k}{\partial u^l} = g^k_j(u)$, and $L^{j_1}...L^{j_{p-1}}g^k_{j_{p}}(u)=L^{j_1}...L^{j_{p}}\pi^k(u)$. The $p$-th order term in \cref{eq:strat-taylor} can then be written as 
\begin{align}
    \mathcal{T}_p^k:=\sum_{0\leq j_{1},j_{2},...,j_{p}\leq M} L^{j_1}...L^{j_{p}}\pi^k(u) J_{j_{1}...j_{p}},
\end{align}
and its symmetrisation written as
\begin{align}
    S_p^k:=\operatorname{Sym}(\mathcal{T}_p^k) = \sum_{0\leq j_{1},j_{2},...,j_{p}\leq M} L^{j_1}...L^{j_{p}}\pi^k(u) \operatorname{Sym}(J_{j_{1}...j_{p}}) = \sum_{0\leq j_{1},j_{2},...,j_{p}\leq M} \operatorname{Sym}(L^{j_1}...L^{j_{p}})\pi^k(u) J_{j_{1}...j_{p}}.
\end{align}
Where the symmetrisation operator is transferred from the nested integration of Brownian motion to the sequentially applied Lie operators, this property is proven in \cref{app:Transfer of symmetrisation}. 
We can then write the difference as
\begin{align}
    \mathcal{T}_p^k - \operatorname{Sym}(\mathcal{T}_p^k) = \sum_{0\leq j_{1},j_{2},...,j_{p}\leq M} \left(L^{j_1}...L^{j_{p}} -  \operatorname{Sym}(L^{j_1}...L^{j_{p}})\right)\pi^k(u) J_{j_{1}...j_{p}}\label{eq:step1}.
\end{align}

Within the summand of \cref{eq:step1} expand the definition of $\operatorname{Sym}$ and bring the other term inside the sum to give
\begin{align}
   L^{j_1}...L^{j_{p}} -  \operatorname{Sym}(L^{j_1}...L^{j_{p}})=  \frac{1}{p!}\sum_{\sigma\in S_p} \left(L^{j_1}...L^{j_{p}} -  \left(L^{j_{\sigma(1)}}...L^{j_{\sigma(p)}}\right)\right) \label{eq:symmetrised_operator_difference}.
\end{align}

Let $\mathcal{W}_{\b j} = L^{j_1}...L^{j_{p}}$, be a multi-index $\b j = (j_1,...j_{p})\in \lbrace 0,...,M\rbrace^p$ operator word, so that we can denote the summand in
\cref{eq:symmetrised_operator_difference} as follows $\left(L^{j_1}...L^{j_{p}} -  \left(L^{j_{\sigma(1)}}...L^{j_{\sigma(p)}}\right)\right) = \mathcal{W}_{\b j}-\sigma\mathcal{W}_{\b j}$.

The remainder of the proof uses two facts to show \cref{eq:symmetrised_operator_difference} can be written as a sum where each term contains at least one commutator. The first fact is that the difference attained by a single adjacent transpose leads to a commutator, the second fact is that a permutation can be written in terms of adjacent  transpositions.

Consider an adjacent transposition operator, $s_r\in S_p$, defined to act on a operator word by interchanging the entries/indices in positions $r,r+1$, such that 
\begin{align}
    s_r \mathcal{W}_{\b j} = \underbrace{L^{j_1}...L^{j_{r-1}}}_{A_r} L^{j_{r+1}} L^{j_{r}} \underbrace{L^{j_{r+2}}...L^{j_{p}}}_{B_r}.
\end{align}
By direct calculation the difference attained by any single adjacent transpose leads to a commutator 
\begin{align}
\mathcal{W}_{\b j} - s_r \mathcal{W}_{\b j} = A_r[ L^{j_{r}}, L^{j_{r+1}}]B_r\label{eq:adjacent_transpose_gives_commmutator}.
\end{align}

We shall now equip each permutation in the symmetric group of order $p$ a unique index $l$ and explicitly refer to it as $\sigma^l$, so that each term we sum over in \cref{eq:symmetrised_operator_difference} has a unique index in the range from $1$ to $p!$, for the purpose of book keeping.

Any permutation in the symmetric group $\sigma^l\in S_p$, can be written as a finite combination of adjacent transposes. 
Meaning there exists a sequence of length $N_{l}$ denoted $s^{l}_{i_{N_{l}}}...s^{l}_{i_1}$ such that $\sigma^{l} \mathcal{W}_{\b j}= s^{l}_{i_{N_{l}}}...s^{l}_{i_{1}}\mathcal{W}_{\b j}$. We shall then per $\sigma^{l}, $ define the intermediary operator words 
\begin{align}
    \mathcal{W}^0_{\b j} :=  \mathcal{W}_{\b j},\quad \mathcal{W}^1_{\b j} := s^l_{i_{1}}\mathcal{W}^0_{\b j},\quad 
    \mathcal{W}^2_{\b j} := s^l_{i_{2}}\mathcal{W}^1_{\b j},\quad
    ...\quad, \mathcal{W}^{N_{l}}_{\b j} := s^l_{i_{N_{l}}}\mathcal{W}^{N_{l}-1}_{\b j} = \sigma^{l} \mathcal{W}_{\b j},
\end{align}
and telescope the expression of adjacent transpositions
\begin{align}
    \mathcal{W}_{\b j} - \sigma^{l} \mathcal{W}_{\b j} &= \mathcal{W}^0_{\b j} - \mathcal{W}^1_{\b j} + \mathcal{W}^1_{\b j}- ... + \mathcal{W}^{N_{l}-1}_{\b j} - \mathcal{W}^{N_{l}}_{\b j}\\
    &= (\mathcal{W}^0_{\b j} - s_{i_1}^{l} \mathcal{W}^0_{\b j}) + ... + (\mathcal{W}_{\b j}^{N_{l}-1} - s^{l}_{i_{N_{l}}}\mathcal{W}_{\b j}^{N_{l}-1})\\
    & = \sum_{\ell = 0}^{N_{l}-1} \left( \mathcal{W}^{\ell}_{\b j} - s^{l}_{i_{\ell+1}}\mathcal{W}^{\ell}_{\b j} \right) \label{eq:telescoping_words_over_ajacent_transposes}.
\end{align}
$\mathcal{W}^{\ell}_{\b j}$, denotes the intermediary operator word obtained after applying the first $\ell$ adjacent transpositions in a chosen decomposition of a permutation $\sigma^{l}$. One can then define the intermediary index tuple $\b j^{\ell} = (j^{\ell}_1,...,j^{\ell}_p)$, where $\b j^{0} = \b j = (j_1,...,j_p)$ and recursively $\b j^{\ell} = s_{i_{\ell}}\b j^{\ell-1}$ for $\ell = 1,...,N_l$. It then follows that $\mathcal{W}^{\ell}_{\b j} = \mathcal{W}_{\b j^{\ell}}$. This allows \cref{eq:telescoping_words_over_ajacent_transposes} to be written in a way we can apply \cref{eq:adjacent_transpose_gives_commmutator} as follows
\begin{align}
    \mathcal{W}_{\b j} - \sigma^{l} \mathcal{W}_{\b j}  &= \sum_{\ell = 0}^{N_{l}-1} \left( \mathcal{W}_{\b j^{\ell}} - s^{l}_{i_{\ell+1}}\mathcal{W}_{\b j^{\ell}} \right) =\sum_{\ell = 0}^{N_{l}-1} A_{\ell}[ L^{j^{\ell}_{i_{\ell+1}}}, L^{j^{\ell}_{i_{\ell+1}+1}}]B_{\ell},
\end{align}
where 
$A_{\ell} = L^{j^{\ell}_1}...L^{j^{\ell}_{i_{\ell+1}-1}}$, $B_{\ell} = L^{j^{\ell}_{i_{\ell+1}+2}}...L^{j^{\ell}_{p}}$. In summary, the difference at order $p$ can be written as
\begin{align}
\mathcal{T}_p^k - \operatorname{Sym}(\mathcal{T}_p^k) &=   \sum_{0\leq j_{1},j_{2},...j_{p}\leq M} \left(L^{j_1}...L^{j_{p}} -  \operatorname{Sym}(L^{j_1}...L^{j_{p}})\right)\pi^k(u) J_{j_{1}...j_{p}}\\
&=\sum_{0\leq j_{1},j_{2},...j_{p}\leq M}\frac{1}{p!}\sum_{\sigma\in S_p} \left(L^{j_1}...L^{j_{p}} - (L^{j_{\sigma(1)}}...L^{j_{\sigma(p)}})\right) \pi^k(u) J_{j_{1}...j_{p}}\\
&=\frac{1}{p!}\sum_{0\leq j_{1},j_{2},...j_{p}\leq M}\sum_{\sigma^{l}\in S_p} \sum_{\ell = 0}^{N_{l}-1} \left( W^{\ell}_{\b j} - s^l_{i_{\ell+1}}W^{\ell}_{\b j} \right)\pi^k(u) J_{j_{1}...j_{p}}\\
&=\frac{1}{p!}\sum_{0\leq j_{1},j_{2},...j_{p}\leq M}\sum_{\sigma^{l}\in S_p} \sum_{\ell = 0}^{N_{l}-1} \left(A_{\ell}[ L^{j^{\ell}_{i_{\ell+1}}}, L^{j^{\ell}_{i_{\ell+1}+1}}]B_{\ell} \right) \pi^k(u) J_{j_{1}...j_{p}}.\label{eq:commutator_decomposition}
\end{align}
Where $l$ uniquely indexes all permutations in the symmetric group $S_p$. In the case the vector fields commute and drift commute $[g_i,g_j]=0$, $\forall (i,j)\in \lbrace 0,...,M \rbrace^2$, the Lie operators commute $[L^{i},L^{j}]=0$, $\forall (i,j)\in \lbrace 0,...,M \rbrace^2$ (see \cref{eq:commutator_componentwise_equivalence}) and the entire \cref{eq:commutator_decomposition} term vanishes.

In order to prove \cref{item3}, suppose that $P\geq 2$ is even, and $[g_i,g_j]=0$, $\forall (i,j)\in \lbrace 0,...,M \rbrace^2$, then invoking that \cref{eq:commutator_decomposition} is zero for $p=2,...,P$
the leading error term is
\begin{align}
u^k(t_{n+1}) - \hat{u}^k_{n+1} &= \sum_{0 \leq j_1, j_2,... j_{P+1} \leq M} L^{j_1} L^{j_2}... L^{j_{P}} g^{k}_{j_{P+1}} J_{j_1,j_2,...j_{P+1}} + \mathcal{O}((\Delta t)^{P/2+1})\label{eq:assumption}\\
& =  \frac{1}{(P+1)!}\sum_{0 \leq j_1, j_2,... j_{P+1} \leq M} L^{j_1} L^{j_2}... L^{j_{P}} g^{k}_{j_{P+1}} J_{j_1}J_{j_2}...J_{j_{P+1}} + \mathcal{O}((\Delta t)^{P/2+1}).
\end{align}
Taking expectation, the lowest order term is a product of an odd number of symmetric random variables and gives zero. Non-vanishing contributions must contain at least one deterministic factor $J_0=\Delta t$, which raises the order beyond the existing $\mathcal{O}((\Delta t)^{P/2+1})$ term. Giving the following local error estimate of the mean error and mean square error respectively
\begin{align}
\mathbb{E}[ u(t_{n+1}) - \hat{ u}_{n+1}] &= \mathcal{O}((\Delta t)^{P/2+1}), \quad \left( \mathbb{E}[|| u(t_{n+1}) - \hat{ u}_{n+1}||_2^2] \right)^{1/2} = \mathcal{O}((\Delta t)^{P/2+1/2}).
\end{align}
Therefore, by the FTMSC \cite{milstein2004stochastic} one establishes the following global error estimate 
\begin{align}
\left( \mathbb{E}[||u(t_{n+1}) - \hat{ u}_{n+1}||_2^2] \right)^{1/2} &= \mathcal{O}((\Delta t)^{P/2}), \quad \forall n=0,1,...,N,
\end{align}
implying global mean square order $P/2$, when $P$ is even.\newline

For \cref{item4}, if instead $P$ is odd, one instead attains $\mathbb{E}[J_{j_1}...J_{j_{P+1}}] = \mathcal{O}((\Delta t)^{P/2+1/2})$ sufficient for the following local error estimates
\begin{align}
\mathbb{E}[u(t_{n+1}) - \hat{ u}_{n+1}] &= \mathcal{O}((\Delta t)^{P/2+1/2}),\quad
\left( \mathbb{E}[||u(t_{n+1}) - \hat{u}_{n+1}||_2^2] \right)^{1/2} = \mathcal{O}((\Delta t)^{P/2+1/2}).
\end{align}
Therefore, using the FTMSC \cite{milstein2004stochastic} with the previous local estimates is only sufficient to prove (strong) mean square convergence order $P/2-1/2$ with the following global error estimate 
\begin{align}
\left( \mathbb{E}[||u(t_{n+1}) - \hat{ u}_{n+1}||_2^2] \right)^{1/2} &= \mathcal{O}((\Delta t)^{P/2-1/2}). 
\end{align}
\newline

For the non-commuting case \cref{item1} consider the decomposition at weight order 2, 
\begin{align}
\sum_{0 \leq j_1, j_2 \leq M} L^{j_1} g^{k}_{j_2} J_{j_1, j_2} &= \sum_{0 \leq j_1, j_2 \leq M} L^{j_1} g^{k}_{j_2} \operatorname{Sym}(J_{j_1, j_2}) + \sum_{0 \leq j_1, j_2 \leq M} L^{j_1} g^{k}_{j_2} \operatorname{Alt}(J_{j_1, j_2})\\
&= \frac{1}{2}\sum_{0 \leq j_1, j_2 \leq M} L^{j_1} g^{k}_{j_2} J_{j_1}J_{ j_2} + \frac{1}{2} \sum_{0 \leq j_1, j_2 \leq M}\left(L^{j_1} g^{k}_{j_2}-L^{j_2} g^{k}_{j_1}\right) J_{j_1, j_2}.\label{eq:second_order_weighted_terms}
\end{align}
Where $\operatorname{Alt}\left(J_{\mathrm{\b j}}\right) := \frac{1}{n!} \sum_{\sigma \in S_n} \operatorname{sign}(\sigma)J_{j_{\sigma(1)}, \ldots, j_{\sigma(n)}}$ denotes the alternating part of nested Brownian motion.\newline

In order to prove \cref{item1}, 
suppose that $P\geq 2$, and $f,g_i$ do not commute or satisfy any other higher order commutative identities, consider $u(t+\Delta t) - \hat{u}_{n+1}$, using \cref{eq:strat-taylor} and \cref{eq:taylor-sub}, one attains that the leading order term \cref{eq:second_order_weighted_terms} is $\mathcal{O}(\Delta t)$, and given by the Lévy area correction 
\begin{align}
u^k(t+\Delta t) - \hat{u}^k_{n+1} =  \sum_{1\leq j_1,j_2\leq M} L^{j_1}g^k_{j_2}\operatorname{Alt}(J_{j_{1},j_{2}}) + \mathcal{O}(\Delta t^{3/2}).
\end{align}

The expectation of the Lévy area is 0 and the expectation of the square of the Lévy area using Itô's isometry is $\mathcal{O}(\Delta t^2)$, and is sufficient to give local error estimates of at least
\begin{align}
\mathbb{E}[u(t_{n+1}) - \hat{ u}_{n+1}] &= \mathcal{O}((\Delta t)^{3/2}), \quad \left( \mathbb{E}[||u(t_{n+1}) - \hat{ u}_{n+1}||_2^2] \right)^{1/2} = \mathcal{O}((\Delta t)^{1}).
\end{align}
Which is sufficient by FTMSC \cite{milstein2004stochastic} for global mean square order 1/2, with the error estimate
\begin{align}
\left( \mathbb{E}[||u(t_{n+1}) - \hat{u}_{n+1}||_2^2] \right)^{1/2} &= \mathcal{O}((\Delta t)^{1/2}), \quad \forall n=0,1,...,N.
\end{align}\newline

In order to prove \cref{item6}, suppose $[g_i,g_{j}] = 0$ for $i,j=1,...,M$, but $[g_i,g_{0}] \neq 0$ for some $i$. Combinations of $g_{i},g_j,g_k$ where $i,j,k>0$ must vanish since they must have a commutator in, and any expansion with $g_{0}$ in must be at least $\mathcal{O}(\Delta t^{2})$, consequently the leading order $\mathcal{O}(\Delta t^{3/2})$ term arises solely from the following expression
\begin{align}
\frac{1}{2} \sum_{0 \leq j_1, j_2 \leq M}\left(L^{j_1} g^{k}_{j_2}-L^{j_2} g^{k}_{j_1}\right) J_{j_1, j_2}. 
\end{align}
Which upon re-indexing, using the commutativity condition, the shuffle property $J_{0,j_1} = -J_{j_1,0} + J_{0}J_{j_1}$ gives
\begin{align}
u^k(t+\Delta t) - \hat{u}^k_{n+1}  =  \frac{1}{2} \sum_{1 \leq j_1 \leq M}\left(L^{0} g^{k}_{ j_1}-L^{j_1} g^{k}_{ 0}\right) J_{0}J_{j_1} + \sum_{1 \leq j_1 \leq M}\left(L^{j_1} g^{k}_{0}-L^{0} g^{k}_{ j_1}\right) J_{j_1, 0} + \mathcal{O}(\Delta t^{2}).\label{eq:order 3 over 2}
\end{align}
In expectation, the leading order term in \cref{eq:order 3 over 2} vanishes, but the square is $\mathcal{O}(\Delta t^3)$, such that 
\begin{align}
\mathbb{E}[ u(t_{n+1}) - \hat{ u}_{n+1}] &= \mathcal{O}((\Delta t)^{2}), \quad \left( \mathbb{E}[|| u(t_{n+1}) - \hat{ u}_{n+1}||_2^2] \right)^{1/2} = \mathcal{O}((\Delta t)^{3/2}).\label{eq:Rumelin commutative estimate}
\end{align}
Sufficient by the FTMSC \cref{thm:FTMSC} to establish the following global error estimate 
\begin{align}
\left( \mathbb{E}[||u(t_{n+1}) - \hat{ u}_{n+1}||_2^2] \right)^{1/2} &= \mathcal{O}((\Delta t)^{1}), \quad \forall n=0,1,...,N,
\end{align}
implying global mean square order $1$ for a commutative basis of noise, proving \cref{item6}.\newline
\Cref{item5} is true by construction.
\end{proof}

\begin{remark}
In the above theorem: 
\begin{enumerate}
    \item We do not assume any structure of the one step map beyond beyond the local expansion requirement in the definition of \cref{eq:stochastic_abstract_method}. 
    \item When using a odd order deterministic integrator the expectation of the deviation (mean of the error), was not sufficient to prove order $P/2$ by the FTMSC under drift commuting noise, and lowered the order of provable convergence by 1/2. 
    \end{enumerate}
\end{remark}
\begin{definition}
We shall refer to a scheme with quadruple order $(P_{d},P_{dc},P_{c},P_s)$.  $P_{d}$ denotes the deterministic global order of convergence. $P_{dc}$ denotes the mean square strong stochastic global order convergence for a drift commutative basis of noise. $P_{c}$ denotes mean square strong stochastic global order convergence for a commutative basis of noise. $P_{s}$ refers to the strong stochastic order for general multidimensional noise. 
\end{definition}
\begin{remark} Informally, 
\cref{thm:onestep} states any deterministic one-step map of order $p\geq 2$ modified to solve a SDE under the transform
\cref{eq:stochastic_abstract_method} has deterministic, drift commutative, commutative, stochastic order of at least
\begin{align}
(P_{d},P_{dc},P_{c},P_s)  = (p,\operatorname{int}(p/2),1,1/2),
\end{align}
understood in the global mean square sense. 
\end{remark}   

\section{Methods}\label{sec:methods}
\subsection{Stochastic Runge-Kutta (SRK) methods}\label{sec:SRK methods}

The simplest example of a convenient integration technique is to take any deterministic order $p$ Runge-Kutta scheme uniquely specified by the Butcher-Tableau $(A,b,c)$ under the following mapping 
\begin{align}
f \mapsto f + \sum_{m=1}^{M}g_m \frac{\Delta W_m}{\Delta t},
\end{align}
to give the following SRK scheme:
\begin{method}[SRK]\label{method:SRK}
\begin{align}
u^{(i)} &= u^n +  \Delta t\sum_{j=1}^s a_{ij} \left(  f(t^n + c_j \Delta t, u^{(j)}) + g_m\left(t^n + c_j \Delta t, u^{(j)}\right) \frac{\Delta W^m}{\Delta t} \right) , \quad i = 1,...,s,\label{eq:SRK1}\\
u^{n+1} &= u^{n} + \Delta t \sum_{i=1}^s b_i  \left( f(t^n + c_i \Delta t, u^{(i)}) + g_m\left(t^n + c_i \Delta t, u^{(i)}\right) \frac{\Delta W^m}{\Delta t} \right). \label{eq:SRK2}
\end{align} 
where summation over repeated index $m$ is implied.
\end{method}
\Cref{thm:onestep} applies to \cref{method:SRK}, so that order $P\geq 2$ Butcher-Tableau conditions are sufficient for convergence properties: $(P_d,P_{dc},P_{c},P_{s}) = (p,[p/2],1,1/2)$, capturing all symmetric terms in the Stratonovich-Taylor expansion. Second order Butcher-tableau conditions for $P_{c},P_{s} = 1,1/2$ are consistent with the convergence results derived by Rüemelin (compare \cref{eq:Rumelin commutative estimate} to Theorem 4 in \cite{ruemelin1982numerical}). 


We shall now define three practical SRK schemes tested in this paper:

\begin{method}[SSP22]
\begin{align}
u^{(1)} &= u^n + \Delta t \left( f(t^n,u^{n})+g_m(t^n,u^{n})\frac{\Delta W^m}{\Delta t}  \right)     \\
u^{n+1} &= \frac{1}{2} u^{n} + \frac{1}{2} \left( u^{(1)} + \Delta t \left( f(t^{n}+\Delta t,u^{(1)}) + g_m(t^n+\Delta t,u^{(1)})\frac{\Delta W^m}{\Delta t} \right) \right) 
\end{align}
\Cref{thm:onestep} implies convergence properties $(P_{d},P_{dc},P_{c},P_s)=(2,1,1,1/2)$. The naming convention SSP22 refers to being strong-stability preserving two-stage and second order accurate. However, it should be noted this refers to the deterministic scheme, strong stability in the stochastic context is discussed in \cite{woodfield2024strong}.
\end{method}

\begin{remark}[Notation]
We shall present schemes in the time independent case where $f(t,u)=f(u)$, $g_m(t,u) = g_m(u)$, and adopt the shorthand $G(u)\Delta S = g_m(u)\Delta W^m$ for presentation purposes, this is wlog and the time dependent scheme  
can be inferred from \cref{method:SRK}-\cref{eq:SRK1,eq:SRK2}. More specifically: recover $c$ from the consistency condition $c=Ae$ and evaluate temporally dependent functions at $t^n+c_i\Delta t$ in stage $i$.
\end{remark}

\begin{method}[SSP33-\cite{shu1988efficient}] 
\begin{align}
u^{(1)} &= u^n + \Delta t \left( f(u^{n})+G(u^{n})\frac{\Delta S}{\Delta t}  \right)     \\
u^{(2)} &= \frac{3}{4} u^{n} + \frac{1}{4} \left( u^{(1)} + \Delta t \left( f(u^{(1)}) + G(u^{(1)})\frac{\Delta S}{\Delta t} \right) \right) \\
u^{(3)} &= \frac{1}{3} u^{n} + \frac{2}{3} \left( u^{(2)} + \Delta t \left( f(u^{(2)}) + G(u^{(2)})\frac{\Delta S}{\Delta t} \right) \right) 
\end{align}
\Cref{thm:onestep} implies convergence properties $(P_{d},P_{dc},P_{c},P_s)=(3,1,1,1/2)$ and symmetric terms of order $3/2$ are captured. For the temporally dependent version use $c = (0,1,1/2)$, \cref{method:SRK}. 
\end{method}

\begin{method}[SRK4]\label{method:SRK4} The SRK4 method is given by  
\begin{align}
u^{(1)} 
& = u_n\\
u^{(2)} 
& = u^{n}+\frac{\Delta t}{2} \left( f\left(u^{(1)}\right) + G\left(u^{(1)}\right)\frac{\Delta S}{\Delta t} \right), \\
u^{(3)}
&= u^n+\frac{\Delta t}{2} \left( f\left(u^{(2)} \right) + G\left(u^{(2)} \right)\frac{\Delta S}{ \Delta t}\right), \\
u^{(4)} 
& = u^n + \Delta t \left( f\left(u^{(3)}\right) +  G\left(u^{(3)}\right) \frac{\Delta S}{\Delta t} \right) , \\
u^{n+1} 
& = u^n+\Delta t\left(\frac{1}{6}  \mathcal{F}\left(u^{(1)}\right)+\frac{1}{3}  \mathcal{F}\left(u^{(2)}\right)+\frac{1}{3}  \mathcal{F}\left(u^{(3)}\right)+\frac{1}{6} \mathcal{F}\left(u^{(4)}\right)\right).\\
\mathcal{F}(a) &:= f(a) +\frac{ G(a) \Delta S}{\Delta t}.
\end{align}
For the temporally varying version $c = (0,1/2,1/2,1)$. 
\Cref{thm:onestep} implies convergence properties $(P_{d},P_{dc},P_{c},P_s)=(4,2,1,1/2)$ and symmetric terms up to order $2$ are captured.
\end{method}

\begin{remark}[Implementation convenience]
A single optimised function $\mathcal{F}$, can be called at each stage in the convenient SRK approach (see \cref{method:SRK4}), and a EM scheme can be reused in the SSP methods.
\end{remark}

\subsection{Stochastic Integrating Factor Runge Kutta (SIFRK) methods}\label{sec:SIFRK}
We shall recall the derivation of SIFRK methods, before subsequently showing they can be alternatively derived as deterministic IFRK methods under a change of variables. Consider the Stratonovich SDE \cref{eq:Stratonovich sde}, admitting the following decomposition of drift $f$ into linear $L$ and nonlinear $N$ terms:
\begin{align}
d_t u  = \underbrace{\left[ L u + N(t,u)\right]}_{f} dt + g_m(t,u) \circ dW^m(t). \label{eq:sde}
\end{align}
Where the vector-valued functions $g_m(t,u)$ are integrated against the increments of the $m$-th i.i.d Brownian motion in a Stratonovich sense, and summation over the repeated index $m=1,...,M$ is assumed. 
It is assumed $L$ is a linear matrix, and $N$ is a nonlinear vector valued function. SIFRK methods can be derived, by first multiplying \cref{eq:sde} by the integrating factor $e^{-Lt}$ to rearrange into the SDE
\begin{align}
d_t(e^{-Lt}u) =  e^{ -Lt} N(t,u) dt + e^{-Lt} g_m(t,u) \circ dW^m.\label{eq:SDE_IFRRK}
\end{align}

By transforming variables, one can write \cref{eq:SDE_IFRRK} as the non-autonomous SDE
\begin{align}
d_t w = H(t,w)dt + K_m(t,w) \circ dW^m,\label{eq:SDE}
\end{align}
where $w = e^{-Lt}u$, $H(t,w):= e^{-Lt}N(t,e^{Lt}w)$, $K_m(t,w) := e^{-Lt}g_m(t,e^{Lt}w)$. A one-step RK scheme (specified by the Butcher tableau $(A, b)$ and the temporal consistency conditions $c = A e$) applied to \cref{eq:SDE} gives
\begin{align}
k^{(i)}&=w_n+\Delta t \sum_{j=1}^s a_{i j} H\left( t_n+c_j \Delta t,k^{(j)}\right) + \sum_{j=1}^s {a}_{i j} K_m\left( t_n+{c}_j \Delta t,k^{(j)}\right)\Delta W^m, \quad i=1, \ldots, s,\\
w_{n+1}&=w_n+\Delta t \sum_{i=1}^s b_i H\left( t_n+c_i \Delta t,k^{(i)} \right)+\sum_{i=1}^s {b}_i K_m\left( t_n+{c}_i \Delta t,k^{(i)}\right)\Delta W^m.
\end{align}
Transforming back into the variable $u$ 
and multiplying through by $e^{L(t^n+c_i\Delta t)}$ gives the SIFRK method:

\begin{method}[SIFRK]\label{method:SIFRK}
\begin{align}
u^{(i)} &= e^{L c_i\Delta t}u^n +  \Delta t\sum_{j=1}^s a_{ij} e^{L(c_i - c_j)\Delta t} \left(  N(t^n + c_j \Delta t, u^{(j)}) + g_m(t^n + c_j \Delta t, u^{(j)}) \frac{\Delta W^m}{\Delta t} \right) , \quad i = 1,...,s,\label{eq:SIFRK1}\\
u^{n+1} &= e^{L \Delta t}u^{n} + \Delta t \sum_{i=1}^s b_i e^{L (1 - c_i) \Delta t} \left( N(t^n + c_i \Delta t, u^{(i)}) + g_m(t^n + c_i \Delta t, u^{(i)}) \frac{\Delta W^m}{\Delta t} \right). \label{eq:SIFRK2}
\end{align} 
\end{method}

It is apparent from \cref{eq:SIFRK1,eq:SIFRK2} that this class of SIFRK schemes are deterministic IFRK method under the substitution 
\begin{align}
    N\mapsto N  + \sum_{m=1}^M g_m \frac{\Delta W_m}{\Delta t},\label{eq:mapping1} 
\end{align}
and convergence follows directly from \cref{thm:onestep} since \cref{eq:mapping1} is equivalent to \begin{align}\Delta t f\mapsto \Delta t f+\sum_{m=1}^{M}g_m\Delta W^m.\end{align} 
The Butcher tableau corresponding to the RK4 scheme 
\begin{align}
(A,b) = 
\begin{array}{c|cccc}
0 & & & & \\
1 / 2 & 1 / 2 & & & \\
1 / 2 & 0 & 1 / 2 & & \\
1 & 0 & 0 & 1 & \\
\hline & 1 / 6 & 1 / 3 & 1 / 3 & 1 / 6
\end{array},
\end{align}
gives rise to the following Stochastic integrating factor Runge-Kutta scheme:
\begin{method}[IFSRK4]\label{method:IFSRK4} The IFSRK4 method is given by 
\begin{align}
u^{(1)} 
& = u_n\\
u^{(2)} 
& =\mathrm{e}^{\frac{\Delta t}{2} L} u^{n}+\frac{\Delta t}{2} \mathrm{e}^{\frac{\Delta t}{2} L} N\left(u^{(1)}\right) +\frac{1}{2} \mathrm{e}^{\frac{\Delta t}{2} L} G\left(u^{(1)}\right)\Delta S, \\
u^{(3)}
& =\mathrm{e}^{\frac{\Delta t}{2} L} u^n+\frac{\Delta t}{2} N\left(u^{(2)} \right) +\frac{1}{2} G\left(u^{(2)} \right)\Delta S, \\
u^{(4)} 
& =\mathrm{e}^{\Delta t L} u^n+\Delta t \mathrm{e}^{\frac{\Delta t}{2} L} N\left(u^{(3)}\right) + \mathrm{e}^{\frac{\Delta t}{2} L} G\left(u^{(3)}\right) \Delta S, \\
u^{n+1} 
& =\mathrm{e}^{\Delta t L} u^n+\Delta t\left(\frac{1}{6} \mathrm{e}^{\Delta t L} \mathcal{N}\left(u^{(1)}\right)+\frac{1}{3} \mathrm{e}^{\frac{\Delta t}{2} L} \mathcal{N}\left(u^{(2)}\right)+\frac{1}{3} \mathrm{e}^{\frac{\Delta t}{2} L} \mathcal{N}\left(u^{(3)}\right)+\frac{1}{6} \mathcal{N}\left(u^{(4)}\right)\right).
\end{align}
Where 
\begin{align}
\mathcal{N}(a) = N(a) +\frac{ G(a) \Delta S}{\Delta t}.
\end{align}
\end{method}
Again we present schemes in the time independent case where $N(t,u)=N(u)$, $G(u)\Delta S = g_m(u)\Delta W^m$ for presentation purposes as the time dependent case can be inferred from \cref{eq:SIFRK1,eq:SIFRK2}). Below we define two (stochastic Stratonovich modified schemes) from \cite{isherwood2018strong} below.

\begin{method}[eSSPIFSRK$^{+}(2,2)$]\label{method:eSSPIFSRK22} 
The eSSPIFSRK$^{+}(2,2)$ method is given by: \begin{align}
k^1 &= e^{L\Delta t} \left( u^n + \Delta t N(u^n) + G(u^n) \Delta S \right)\\
u^{n+1} &= \frac{1}{2} e^{L\Delta t} u^{n} + \frac{1}{2} e^{L\Delta t} \left( k^1 + \Delta t N(k^1) + G(k^1) \Delta S \right)
\end{align}
\end{method}

\begin{method}[eSSPIFSRK$^{+}(3,3)$]\label{method:eSSPIFSRK33}
The eSSPIFSRK$^{+}(3,3)$ method is given by:
\begin{align}
k^{1}= & \frac{1}{2} e^{\frac{2}{3} \Delta t L} u^n+\frac{1}{2} e^{\frac{2}{3} \Delta t L}\left(u^n+\frac{4}{3} \Delta t N(u^n)
+\frac{4}{3} G(u^n) \Delta S \right) \\
k^{2}= & \frac{2}{3} e^{\frac{2}{3} \Delta t L} u^n+\frac{1}{3}\left(k^{1}+\frac{4}{3} \Delta t N(k^{1}) + \frac{4}{3}G(k^{1}) \Delta S\right) \\
u^{n+1}= & \frac{59}{128} e^{\Delta t L} u^n+\frac{15}{128} e^{\Delta t L}\left(u^n+\frac{4}{3} \Delta t N\left(u^n\right)+\frac{4}{3}G(u^n)\Delta S\right) +\frac{27}{64} e^{\frac{1}{3} \Delta t L}\left(k^{2}+\frac{4}{3} \Delta t N\left(k^{2}\right)  + \frac{4}{3} G(k^2)\Delta S \right) .
\end{align}
\end{method}

\subsection{Stochastic Exponential Time Differencing Runge Kutta (SETDRK) methods}\label{sec:SETDRK}

Consider the following SDE, 
\begin{align}
d_t u  = \left[ N(t,u) + L u \right] dt +  \left[ g_m(t,u)\right] \circ dW^m_t,
\end{align}

Using an integrating factor technique, one can write this in integral form between $t^n$ and $t^{n+1}$ as follows
\begin{align}
 u^{n+1}  = e^{L \Delta t} u^{n}  + e^{L \Delta t}\int_{0}^{\Delta t} e^{-L s}  N(t_n + s,u(t_n + s)) ds +  e^{L (\Delta t+t^{n})}\int_{t^{n}}^{t^{n+1}} e^{-L s}  g_m(s,u(s)) \circ dW^m_s. \label{eq:integral_Stratonovich}
\end{align}
Unlike SRK or SIFRK methods, SETDRK, do not derive analogously to the deterministic ETDRK methods, and
to approximate \cref{eq:integral_Stratonovich}, we explicitly require the use of the transform in \cref{eq:mapping1} and \cref{thm:onestep} to use any deterministic order 2 or higher ETDRK integration method. Following \cite{hochbruck2010exponential} the general case with the substitution \cref{eq:mapping1} is given by:
\begin{method}[CSETDRK]
\begin{align}
k_{i} & =\chi_i\left(\Delta t L\right) u_n+\Delta t \sum_{j=1}^s a_{i j}\left(\Delta t L\right) \left( N\left(t_n+c_j \Delta t, k_{j}\right) + \frac{g_{m}\left(t_n+c_j \Delta t, k_{j}\right)\Delta W^{m}}{\Delta t}\right),  i=1,...,s\\
u_{n+1} & =\chi\left(\Delta t L\right) u_n+\Delta t \sum_{i=1}^s b_i\left(\Delta t L\right) \left( N\left(t_n+c_i \Delta t, k_{i}\right) + \frac{g_m(t_n+c_i\Delta t,k_i)\Delta W^m}{\Delta t}\right)
\end{align}
Here, $\chi, \chi_i, a_{i j}$, and $b_i$ are not just coefficients, but are constructed from rational expressions involving the exponential evaluated at $\left(\Delta t L \right)$. For consistency it is assumed $\chi(0)=\chi_i(0)=1$, and the underlying Runge-Kutta method is denoted by the coefficients $(a_{i,j}(0),b_j(0))$, typically assumed to satisfy $c_i = A(0)e$, $b^T(0)e = 1$. It is normally considered $\chi(z)=\mathrm{e}^z$ and $\chi_i(z)=\mathrm{e}^{c_i z}, 1 \leq i \leq s$. We refer to \cite{hochbruck2010exponential} for the general construction of deterministic order 2 and higher ETDRK methods, which require specific construction of $\chi$, $\chi_i$, $a_{i j}$, and $b_i$. 
\end{method}

\begin{remark}
    Fortunately, in light of \cref{thm:onestep}, one need not know the general construction of ETDRK methods, but simply adopt any deterministic order $p\geq 2$ method, under the change of variables:
\end{remark}
\begin{align}
\mathcal{N}(t^n,u^n) := N\left(t_n,u_n\right) + g_m\left( t_n,u_n\right) \frac{\Delta W^m}{\Delta t}.
\end{align}

One can inspect a convenient stochastic ETDRK2 (modified from \cite{cox2002exponential}) as follows.
\begin{method}[SETDRK2]\label{method:SETDRK2} The Stochastic-ETDRK2 scheme is given by:
\begin{align}
k_1 &=  e^{L \Delta t} u^{n}  + A_1 \mathcal{N}(t^n,u^n),\\
u_{n+1} &= k_1+A_2 \left(\mathcal{N}(t_n+\Delta t, k_1) - \mathcal{N}(t_n, u_n)\right)
\end{align}
where 
\begin{align}
A_1 =  \frac{ e^{L\Delta t}-1}{L}, \quad
A_2 = \frac{\left(e^{L \Delta t} -1 -\Delta t L \right) }{ (\Delta t) L^2}.
\end{align}
\Cref{thm:onestep} implies this scheme has deterministic, drift commutative, commutative, and general mean square order $(P_d,P_{dc},P_c,P_s) = (2,1,1,0.5)$. The approximation of $A_{i}$ is discussed in \cref{sec:Numerical implementation details}
\end{method}

Similarly, one can inspect a convenient stochastic ETDRK3 (modified from \cite{iserles2009first,cox2002exponential}) 
\begin{method}[SETDRK3]\label{method:SETDRK3} The Stochastic-ETDRK3 scheme is given by:
\begin{align}
k^1 & = u_n e^{L \Delta t / 2}+B_1 \mathcal{N}\left(t_n,u_n\right) ,  \\
k^2 &= u_n e^{L \Delta t}+B_2 \left(2 \mathcal{N}\left( t_n+\Delta t / 2,k^1\right)   -\mathcal{N}\left(t_n,u_n\right)\right),  \\
u_{n+1} &=  u_n e^{L \Delta t} + B_3\mathcal{N}\left(t_n,u_n\right) +B_4 \mathcal{N}\left(t_n+\Delta t / 2, k^1\right)  +B_5 \mathcal{N}\left(t_n+\Delta t,k^2\right) .
\end{align}

Where
\begin{align}
B_1 &= \frac{\left(e^{L \Delta t / 2}-1\right)}{L}, \quad B_2 = \frac{\left(e^{L \Delta t }-1\right)}{L},\quad 
B_3 = \frac{-4-\Delta t L+e^{L \Delta t}\left(4-3 \Delta t L+\Delta t^2 L^2\right)}{\Delta t^2 L^3}, \\
B_4 &= 4 \frac{2+\Delta t L+e^{L \Delta t}(-2+\Delta t L)}{\Delta t^2 L^3}, \quad
B_5 = \frac{-4-3 \Delta t L-\Delta t^2 L^2+e^{L \Delta t}(4-\Delta t L)}{\Delta t^2 L^3}.
\end{align}
\Cref{thm:onestep} gives deterministic, drift commutative, commutative, and general mean square order $(P_d,P_{dc},P_c,P_s) = (3,1,1,0.5)$. The approximation of $B_{i}$ is discussed in \cref{sec:Numerical implementation details}
\end{method}

Similarly one can inspect, the convenient Stratonovich extension of Cox-Mathews fourth order integration scheme originally derived in the deterministic setting by symbolic computation in \cite{cox2002exponential}. 
\begin{method}[SETDRK4]\label{method:SETDRK4} The Stochastic ETDRK4 scheme is given by:
\begin{align} a_n &= e^{L \Delta t / 2} u_n+E_0  \mathcal{N}\left( t_n,u_n\right)  , \\ 
b_n &= e^{L \Delta t / 2} u_n+E_0 \left[ \mathcal{N}\left(t_n+\Delta t / 2, a_n\right)  \right], \\ 
c_n&= e^{L \Delta t / 2} a_n+ E_0 \left[2 \mathcal{N}\left(t_n+\Delta t / 2, b_n\right) - \mathcal{N}\left(t_n,u_n\right)\right], \\ 
u_{n+1}&=  e^{L \Delta t} u_n + E_1 \mathcal{N}\left(t_n,u_n\right) +  E_2 \left[\mathcal{N}\left(t_n+\Delta t / 2,a_n\right)+\mathcal{N}\left(t_n+\Delta t / 2,b_n\right)\right]  + E_3 \mathcal{N}\left(t_n+\Delta t,c_n \right),
\end{align}
where
\begin{align}
E_0 &:= L^{-1}\left(e^{L \Delta t / 2}-I\right), \\
E_1 &:= \Delta t^{-2} L^{-3}\left[-4-L \Delta t +e^{L \Delta t}\left(4-3 L \Delta t+(L \Delta t)^2\right)\right], \\
E_2 &:= \Delta t^{-2} L^{-3} \left[4+2L \Delta t+e^{L \Delta t}(-4+2L \Delta t)\right], \\
E_3 &:= \Delta t^{-2}L^{-3} \left[ -4-3 L \Delta t-(L \Delta t)^2+e^{L \Delta t}(4-L \Delta t)\right].
\end{align}
\end{method}
\Cref{thm:onestep} gives deterministic, drift commutative, commutative, and general mean square order $(P_d,P_{dc},P_c,P_s) = (4,2,1,0.5)$. The approximation of $E_{i}$ is discussed in \cref{sec:Numerical implementation details}

\subsubsection{Numerical implementation details}\label{sec:Numerical implementation details}

In the SDE case, the rational expressions involving the exponential $\lbrace A_i\rbrace_{i=1,2}$, $\lbrace B_i\rbrace_{i=1,2,3,4,5}$, $\lbrace E_i\rbrace_{i=1,2,3,4}$, in \cref{method:SETDRK2}, \cref{method:SETDRK3}, \cref{method:SETDRK4} can be evaluated exactly for large $\Delta t$, and replaced with a truncated Taylor expansion when numerical instability occurs, at small values of $\Delta t$. In the SPDE's case other approaches can be taken,
consider the following motivating SPDE, as a prototypical nonlinear SPDE with quadratic nonlinearity, linear stiff term and nonlinear noise, 
\begin{align}
d_t u + \left( \underbrace{\frac{c_1}{2} (u^2)_x}_{Nonlinear} + \underbrace{c_0 u_x  +c_2  u_{xx} + c_3 u_{xxx} + c_4 u_{xxxx}}_{Linear} \right)dt + \underbrace{\sum_{m=1}^{M}(\xi_m(x)u)_x\circ dW^m}_{Nonlinear} = 0. \label{eq:SPDE}
\end{align}
Subcases include the Heat equation, Linear advection equation, Burger's equation, KdV equation, KS equation, under a nonlinear flux form transport noise interpreted in the Stratonovich sense. Taking the Fourier transform gives, 
\begin{align}
d \hat{u} + \underbrace{\frac{ik c_1}{2} (\widehat{u^2})}_{-N(u)} + \underbrace{(c_0 i k \hat{u} - k^2 c_2 \hat{u} - c_3 i k^3 \hat{u} + c_4 k^4 \hat{u})}_{-Lu} + \sum_{m=1}^{M}\underbrace{i k \widehat{(\xi_m u)}}_{g_m(u)}\circ dW^m = 0 \label{eq:fourier space}
\end{align}
where $k$ denotes the wave-number (scaled by $2\pi/L_x$ where $L_x$ is the length of the domain). By taking a finite dimensional approximation of \cref{eq:fourier space} one has a discrete analogue to \cref{eq:fourier space}, a finite-dimensional system of SDE's of the following form
\begin{align}
d_t  u_h  = \left[  N(\b u_h,t) + L  u_h \right] dt +    \sum_{m=1}^{M}  g_m(u_h,t) \circ dW^m_t. \label{eq:multidimensional case of interest}
\end{align}
In our work we use a pseudospectral discretisation
where the nonlinear terms $ N, g_m$ of quadratic type are computed by multiplication in real space, then Fourier transformed, multiplied by the wave numbers and dealiased through a $2/3$-rds rule before inverse Fourier transformed.

Due to numerical instability highlighted in \cite{higham2002accuracy} the rational expressions involving the exponential $\lbrace A_i\rbrace_{i=1,2}$, $\lbrace B_i\rbrace_{i=1,2,3,4,5}$, $\lbrace E_i\rbrace_{i=1,2,3,4}$, in \cref{method:SETDRK2}, \cref{method:SETDRK3}, \cref{method:SETDRK4} respectively, have removable singularities and require approximation, and in this work are computed before runtime following Kassam and Trefethen's, numerical quadrature of contour integration in the complex plane \cite{kassam2005fourth}. As pointed out in \cite{kassam2005fourth} the practical computation of $E_{i}$, using the following contour integration
\begin{align}
f(\b L) = \frac{1}{2\pi i}\oint_{\Gamma} f(t)(t\b I-\b L)^{-1}dt, 
\end{align}
has significant simplification when the operator diagonalises and the trapezium rule can be employed around roots of a circle in the complex plane centred at specific points relating to eigenvalues of $\b L$ \cite{kassam2005fourth}. The trapezium rule (in this complex plane context) converges exponentially fast, and is done before runtime. Eigenvalues of $\b L$ depend on the spatial discretisation, and number of spatial points used, for even diffusive operators such as the KS equation $-k^2+k^4$ eigenvalues are near the
negative real axis, and for dispersive odd-order operators like the KdV equation eigenvalues are close
to the imaginary axis $ -i k^3$, such differences can be accounted for when defining the contour integration for increased performance \cite{kassam2005fourth}.
The contour integration method applies more generally and includes operators that do not diagonalise \cite{kassam2005fourth}. 
This is not the only approach, numerical computation of matrix exponentials and rational expressions involving exponential operators can be attained through Chebyshev approximation, diagonalisation, Padé approximation \cite{dai2023exponential}, Leja approximation, contour integration, Krylov subspace methods \cite{geiger2012exponential}, scaling and squaring \cite{chen2025parallelization}, see \cite{hochbruck2010exponential} for an overview.

\section{Numerical results}\label{sec:Numerical_results}

In this section we show \cref{thm:onestep} has practical consequences for numerical integration.

\begin{enumerate}
    \item For general non-commutative noise we show mean square order $\mathcal{O}(\Delta t ^{1/2})$, \cref{thm:onestep}-\cref{item1}.
    \item For commutative noise we show mean square order $\mathcal{O}(\Delta t^{1})$ \cref{thm:onestep}-\cref{item6}. 
    \item For drift commutative noise we show order $\mathcal{O}(\Delta t^{\operatorname{int}(p/2)})$ \cref{thm:onestep}-\cref{item3}+\cref{item4}.
    \item In the limit of small noise, large timestep $\Delta t$ and no noise we show $\mathcal{O}(\Delta t^{p})$ \cref{thm:onestep}-\cref{item5}.
\end{enumerate}

We show each of these convergence behaviours in the context of SIFRK, SETDRK, and SRK methods of orders $p=2,3,4$, and facilitate relative comparisons between all $9$ methods of different types and orders for different examples.

\begin{remark}
In Experiment 0 (\cref{sec:ex0}) the strong global root-mean-square error is estimated by 
\begin{align}
e_{\mathrm{MS}}(\Delta t)=
\left(
\frac{1}{R}\sum_{r=1}^{R}
||u_{ref}(\omega_r)-\hat u_{\Delta t}(\omega_r)||_2^2
\right)^{1/2}, \label{eq:ms_estimate_err}
\end{align}
where $(\omega_r)$, $(r=1,\ldots,R)$, are independent Brownian realizations. This is a Monte Carlo estimate of the root mean-square error in \cref{eq:rmse}. All other numerical convergence plots report the samplewise relative L2 error \begin{align}
e_{\mathrm{rel}}(\Delta t)=
\frac{||u_{ref}(\omega)-\hat u_{\Delta t}(\omega)||_2}
{||u_{ref}(\omega)||_2}.\label{eq:samplewise_rel_l2_err}
\end{align}
Errors at successive resolutions are computed using nested refinements of the same Brownian realization \cite{lord2014introduction}.
\end{remark}

\subsection{Experiment 0}\label{sec:ex0}

As a one-dimensional example, inspired by equation (4.52) in \cite{kloeden2011numerical}, p.~125, we consider the SDE
\begin{align}\label{eq:1d-sde}
dq = f(q)dt + g_1(q)\circ dW^1+ g_2(q)\circ dW^2 , \quad f(q)=\lambda q - q^3,
\end{align}
under the following choices of diffusion vector fields
\begin{enumerate}
\item deterministic: $
    g_1(q) = g_2(q) = 0$, 
\item commutative noise: $g_1(q)= \kappa_1 q,\; g_2(q) = \kappa_2 q$, so that $[g_1,g_2]=0$,
\item drift and diffusion commutative: $g_1(q)=\varepsilon_1 f(q),\; g_2(q)=\varepsilon_2 f(q)$, so that $[g_1,g_2]=[f,g_2]=[f,g_1]=0$,
\item fully non-commutative: $g_1(q)=\sigma$, $g_2(q)=\kappa q$.
\end{enumerate}
In \eqref{eq:1d-sde}, $W^1$ and $W^2$ are independent one-dimensional Brownian motions and $\lambda$, $\kappa_1$, $\kappa_2$, $\varepsilon_1$, $\varepsilon_2$, $\sigma$, and $\kappa$ are positive real constants. For the following numerical example, we have chosen $\lambda=2$, $\kappa_1=0.3$, $\kappa_2=0.2$, $\varepsilon_1=0.3$, $\varepsilon_2=0.2$, $\sigma=0.3$, $\kappa=0.2$, and initial value $q_0=0.5$.

Equation \eqref{eq:1d-sde} admits analytic solutions for scenarios 1 and 3 above. In the deterministic case, the substitution $w=q^{-2}$ gives $\mathrm{d}{w}=(-2\lambda w+2 )\,\mathrm{d}t$, which yields the flow map
\begin{equation}\label{eq:ex0_flow}
  \Phi(s)=\Big[\tfrac{1}{\lambda}
    +\big(q_{0}^{-2}-\tfrac{1}{\lambda}\big)e^{-2\lambda s}\Big]^{-1/2}.
\end{equation}

The deterministic solution is $q(t)=\Phi(t)$. In the drift and diffusion commutative case, first-order Stratonovich calculus yields the analytic solution 
\begin{align}
    q(t)= \Phi(t+\varepsilon_1 W^1(t)+\varepsilon_2 W^2(t)),
\end{align}
provided that $0<q_0<\sqrt{\lambda}$.
In the cases with no closed form solution, the reference solution $q_{ref}$ is a SETDRK4 approximation
computed on the same Brownian path with $\Delta t=6.10\times10^{-6}$, otherwise the analytic solution is used. We use $R=2000$ independent Brownian realizations and evaluate the root-mean-square error defined in \cref{eq:ms_estimate_err} at the final time. 

We plot the convergence behaviour in \cref{fig:ex0_GL_allatty}, where for all 9 methods, we observe strong orders: $(P_{d},P_{dc},P_{c},P_s)  = (p,\operatorname{int}(p/2),1,1/2),$ for $p=2,3,4$ consistent with \cref{thm:onestep}. These numerical results indicate that the meansquare convergence orders in \cref{thm:onestep} are sharp.

\begin{figure}
    \centering
    \includegraphics[width=0.95\linewidth]{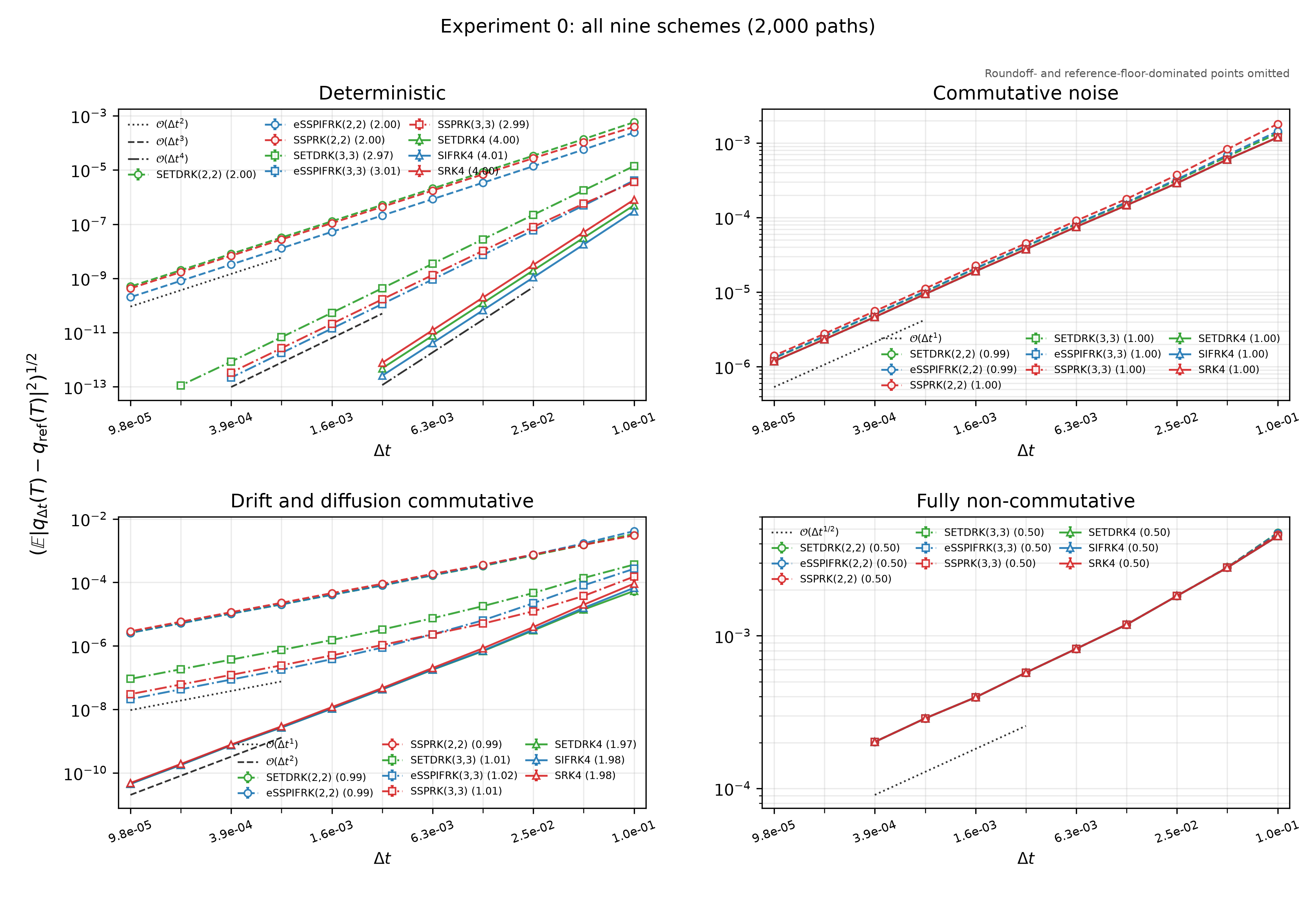}
    \caption{Temporal convergence of all nine schemes for
$\mathrm{d}q=f(q)\,\mathrm{d}t+g_1(q)\circ\mathrm{d}W^1+g_2(q)\circ\mathrm{d}W^2$
with $f(q)=\lambda q-q^{3}$, $\lambda=2$, under the four commutator regimes.
Measured strong orders: ($p$, $\operatorname{int}(p/2)$, $1$, $\frac{1}{2}$) for $p=2,3,4$.}
    \label{fig:ex0_GL_allatty}
\end{figure}

\subsection{Experiment 1: General Multidimensional noise, convergence and efficiency.}\label{sec:ex1}

\Cref{thm:onestep}, \cref{item1} indicates all numerical schemes presented in this paper converge to the Stratonovich system with strong order $1/2$ (for general multidimensional non-commutative noise). However, \cref{thm:onestep} \cref{item2} also indicates that all higher order symmetric terms up to $p/2$ are captured  when using a deterministic scheme of order $p$. This experiment determines to what extent both \cref{item1,item2} are observed in practice for general multidimensional non-commutative noise, by considering the following initial value problem 
\begin{align}
d_t u  +uu_x + u_{xxx} + \sum_{m=1}^{M}(\xi_m u)_x \circ dW_t= 0,\quad \xi_m(x) = \frac{\sin(2\pi x m )}{100 m},\quad u(x,0) = \exp(-50(x - 0.5)^{2}).\label{eq:KdV multipliciative noise}
\end{align}
This setup does not admit an analytic pathwise solution and the noise (plotted in \cref{fig:non-commutative basis}) is non-commutative in the infinite dimensional Fréchet operator sense $[g_i,g_j]=\left(\xi_i (\xi_{j})_{xx} - \xi_j (\xi_{i})_{xx} \right)u + ((\xi_{i})_x \xi_{j} - (\xi_{j})_x \xi_{i}) u_x\neq 0$, nor drift commutes $[f,g_i]\neq 0$. A high resolution solution $u_e$ is used as the reference solution and plotted in
\cref{fig:KdV_Waterfall_stochastic}. 
\begin{figure}[H]
    \centering
    \includegraphics[width=0.5\linewidth]{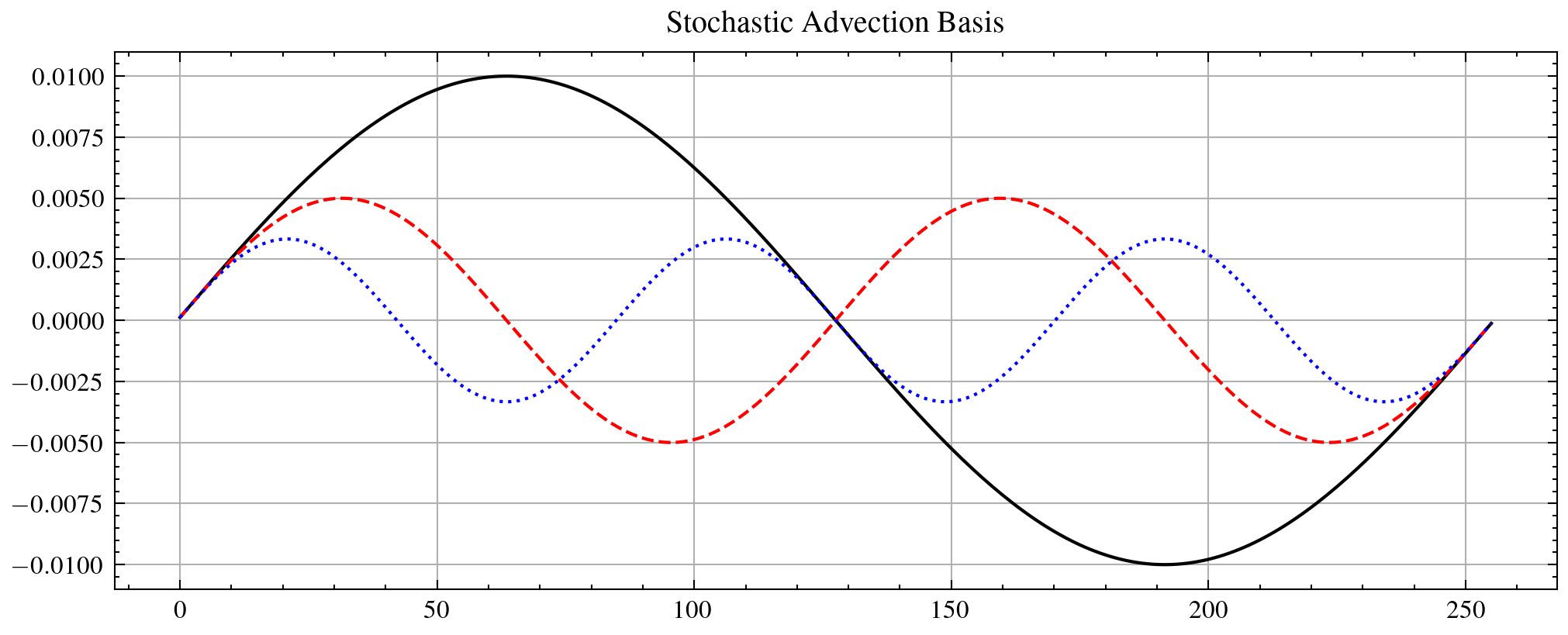}
    \caption{Non-commutative basis $\lbrace\xi_{m}\rbrace_{m=1,2,3}$}
    \label{fig:non-commutative basis}
\end{figure}

In \cref{fig:convergence non-commutative} we plot the results, in \cref{fig:noncommutative second order}, \cref{fig:noncommutative third order}, \cref{fig:noncommutative fourth order}, we respectively plot the relative L2 errors $||u_{ref}-u_{\Delta t}||_2/||u_{ref}||_{2}$ of the second order, third order, and fourth order schemes employed in this paper as compared with the high resolution reference. 
This is to facilitate a comparison between SRK, SIFRK and SETDRK schemes. In \cref{fig:all order non-commutative basis} all convergence results are overlaid on the same graph as to facilitate comparison between methods of different orders. If a figure omits a data point, the numerical method failed to converge at that temporal resolution, and this is adopted for all figures. 


\begin{figure}[H]
\centering
\begin{subfigure}{.495\textwidth}
    \centering
\includegraphics[width=.95\linewidth]{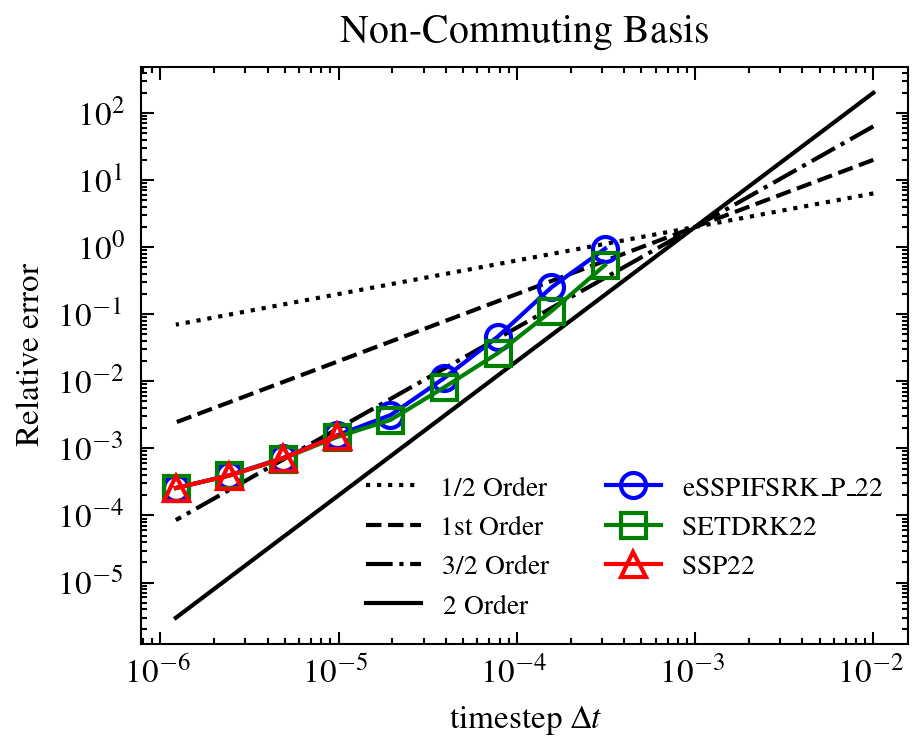}
\caption{Comparison of 2nd order methods non-commuting basis.}
\label{fig:noncommutative second order}
\end{subfigure}
\begin{subfigure}{.495\textwidth}
\centering
\includegraphics[width=.95\linewidth]{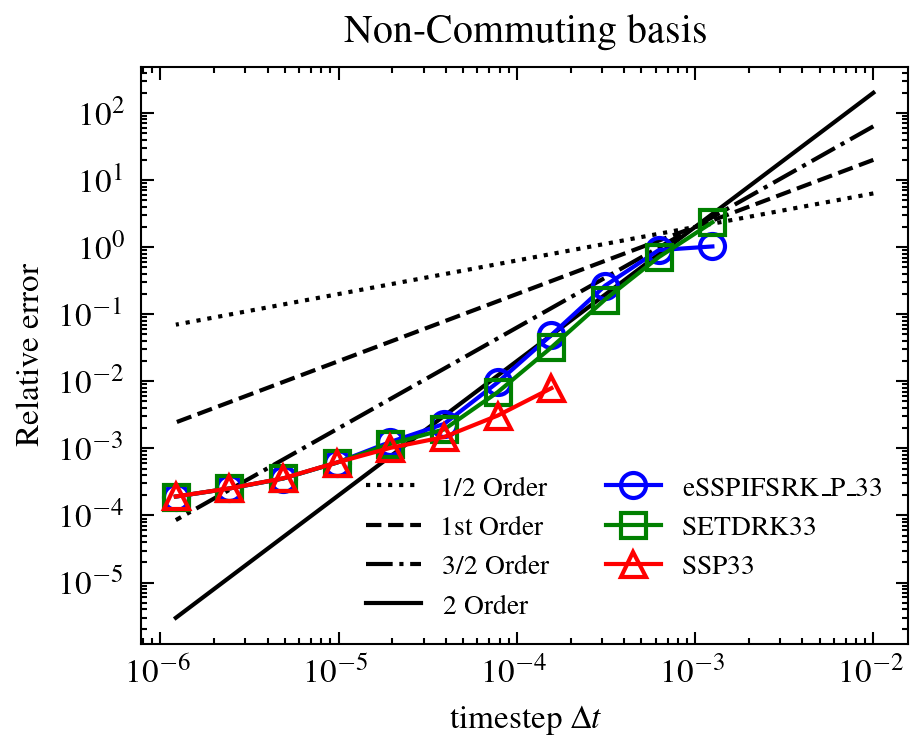}\caption{Comparison of 3-rd order methods non-commuting basis.}
\label{fig:noncommutative third order}
\end{subfigure}\\
\begin{subfigure}{.495\textwidth}
    \centering
\includegraphics[width=.95\linewidth]{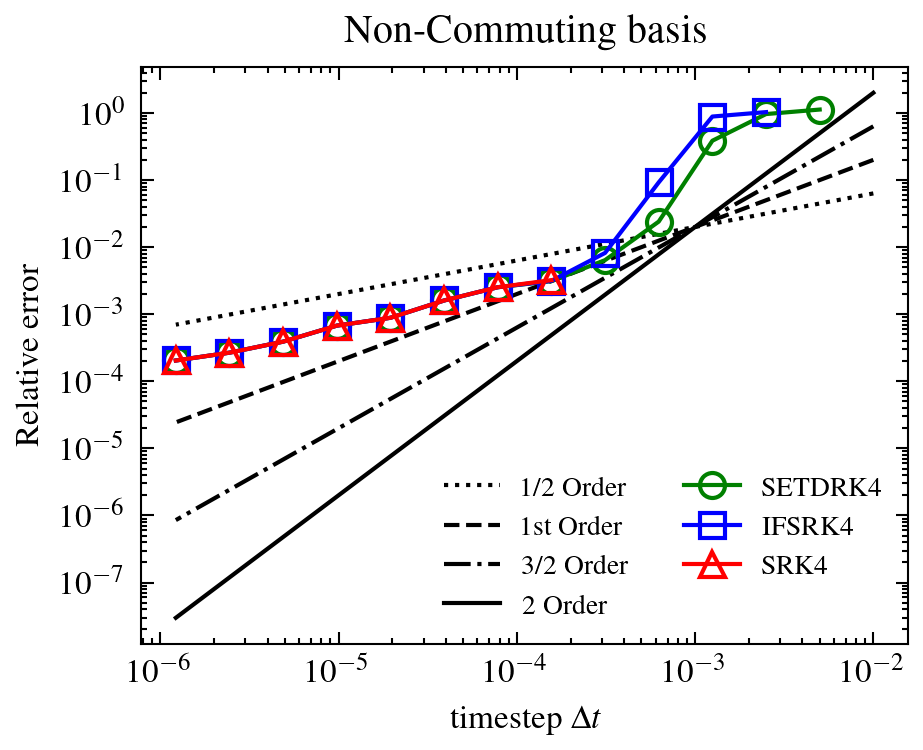}
\caption{Comparison of 4th order methods non-commutative basis.}
\label{fig:noncommutative fourth order}
\end{subfigure}
\begin{subfigure}{.495\textwidth}
\centering
\includegraphics[width=.95\linewidth]{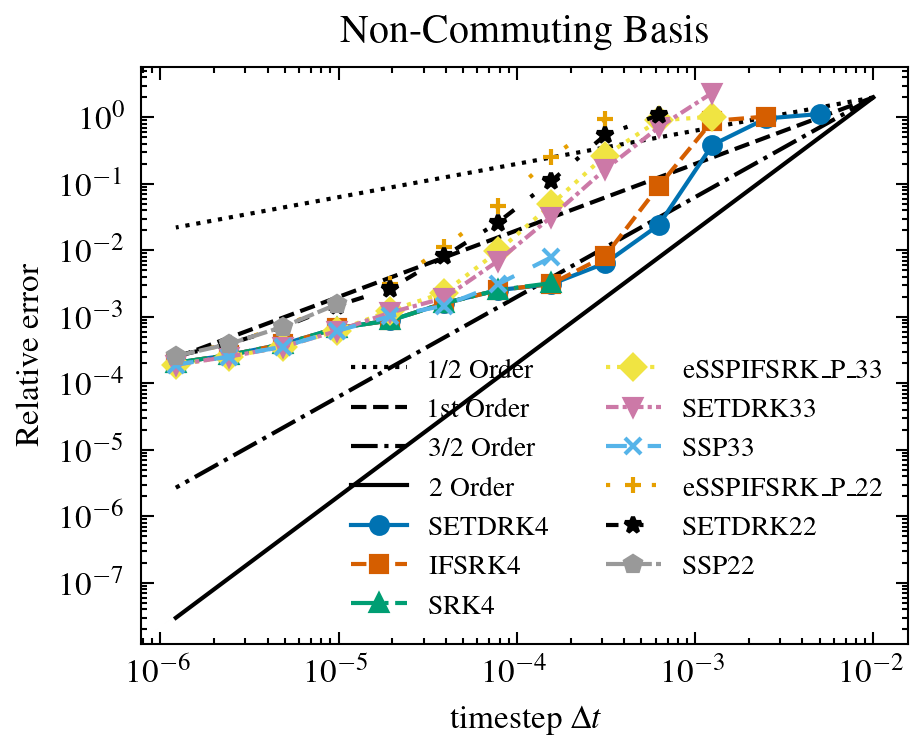}\caption{Comparison of all methods non-commutative basis.}
\label{fig:all order non-commutative basis}
\end{subfigure}
\caption{Temporal convergence for non-commutative noise. In all schemes we see that for non-commutative noise, we attain strong order $1/2$ in the limit of $\Delta t\rightarrow 0$ in \cref{fig:all order non-commutative basis}. SETDRK and SIFRK methods could be computed at timesteps significantly larger than RK methods. SETDRK had errors slightly lower than SIFRK methods. In \cref{fig:all order non-commutative basis} when using larger timesteps there was significant practical merit in using higher order schemes.  }
\label{fig:convergence non-commutative}
\end{figure}

As shown in \cref{fig:noncommutative second order} SETDRK22 and IFSRK22 can take timesteps 32 times larger than SSPRK22. Errors are slightly lower in SETDRK22 than IFSRK22. As shown in \cref{fig:noncommutative third order} SETDRK33 and SIFRK33 can take timesteps 8 times larger than SSP33. As shown in \cref{fig:noncommutative fourth order} SETDRK4 and IFRK4 can take timesteps 32, 16 times larger than RK4 respectively. SETDRK4 has smaller error than IFRK4 at large timesteps. \newline

In \cref{fig:all order non-commutative basis} all schemes are $\mathcal{O}(\Delta t^{1/2})$ in the limit of small $\Delta t$ in accordance with \cref{thm:onestep} \cref{item1}. Indicating decreasing benefit in using increasingly higher order Runge-Kutta, SIFRK and SETDRK approaches in the limit of small noise due to the dominant missing Lévy area approximation. However, as simultaneously shown in \cref{fig:all order non-commutative basis} when using larger and intermediate timesteps, there was significant practical merit in using a higher-order SRK, SETDRK or SIFRK approach in terms of absolute error. 

We run the same convergence experiment and at each temporal resolution we recompute the entire simulation 5 times and take an average for a CPU-time (per resolution per method). We then plot a relative error vs CPU time graphs in \cref{fig:non commuting time second order methods}, \cref{fig:non commuting time third order methods}, \cref{fig:non commuting time fourth order methods} and \cref{fig:non commuting time all order methods}. 
Figures \cref{fig:non commuting time second order methods}, \cref{fig:non commuting time third order methods}, \cref{fig:non commuting time fourth order methods} present CPU-time versus relative error using second, third, and fourth-order schemes, respectively. \Cref{fig:non commuting time all order methods} collates the performance of all methods. Ideally, a method has small error for a given computational cost,  if one scheme is below another scheme it is deemed more computationally efficient than another for a given CPU cost. This experiment takes into account the exponential operator cost and the cost associated with using more stages.

\begin{figure}[H]
\centering
\begin{subfigure}{.495\textwidth}
    \centering
\includegraphics[width=.95\linewidth]{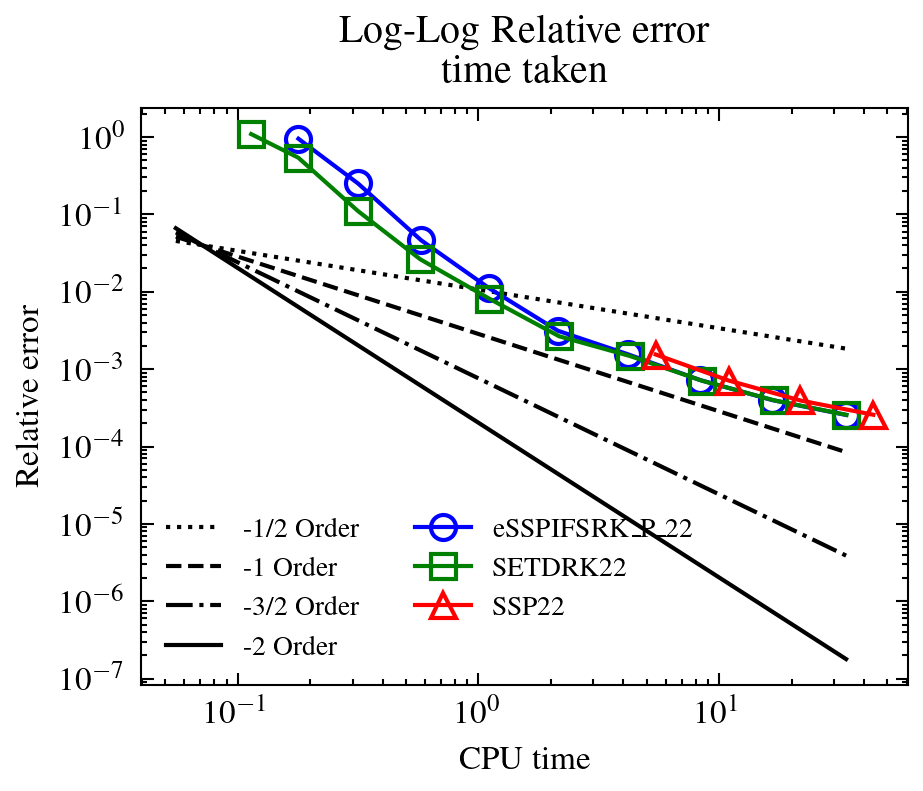}
\caption{2nd order methods.}
\label{fig:non commuting time second order methods}
\end{subfigure}
\begin{subfigure}{.495\textwidth}
\centering
\includegraphics[width=.95\linewidth]{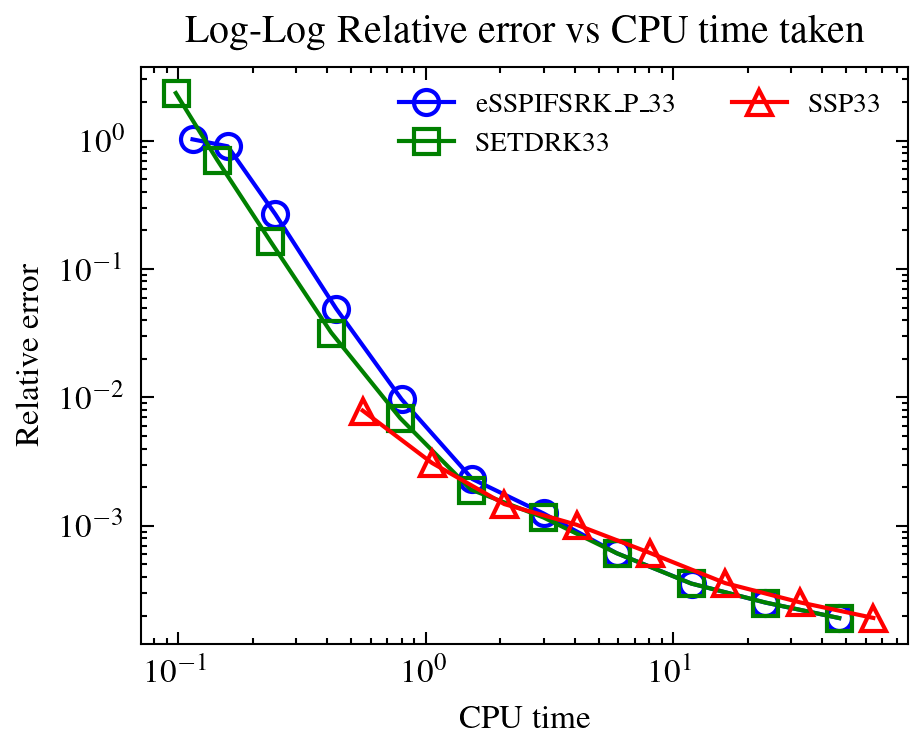}\caption{3rd order methods.}
\label{fig:non commuting time third order methods}
\end{subfigure}\\
\begin{subfigure}{.495\textwidth}
    \centering
\includegraphics[width=.95\linewidth]{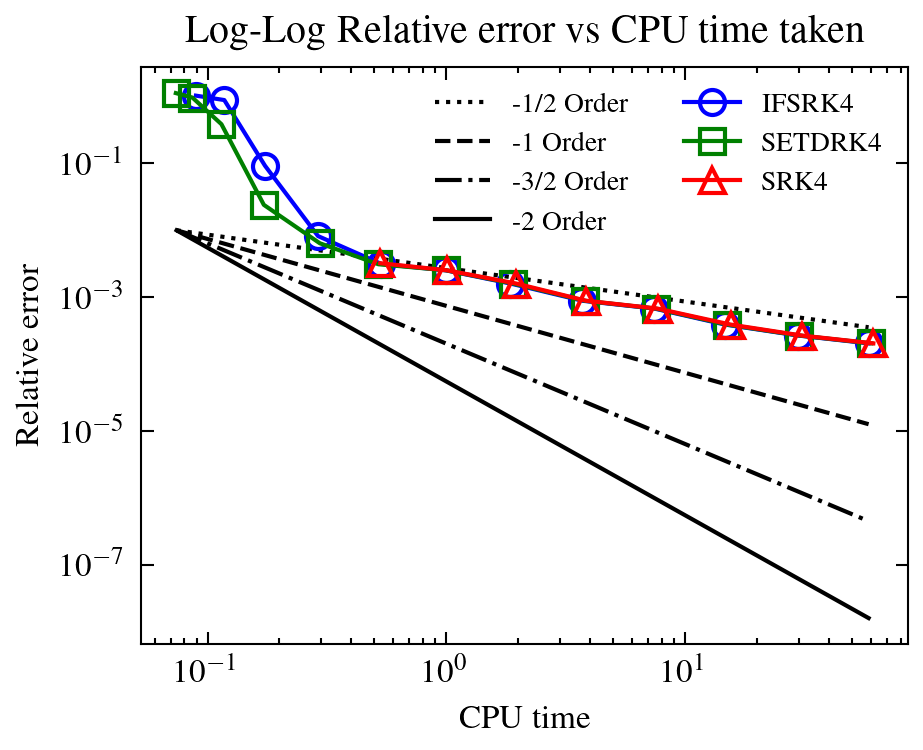}
\caption{4th order methods.}
\label{fig:non commuting time fourth order methods}
\end{subfigure}
\begin{subfigure}{.495\textwidth}
\centering
\includegraphics[width=.95\linewidth]{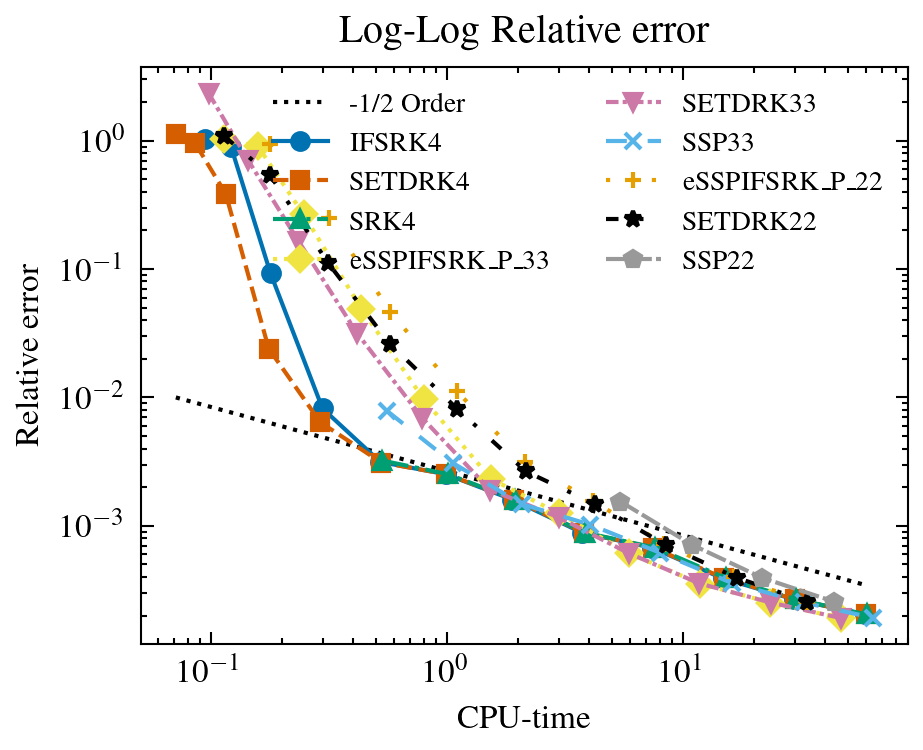}\caption{All methods.}
\label{fig:non commuting time all order methods}
\end{subfigure}
\caption{CPU-time vs Relative error, non-commuting basis. Bottom left, is best }
\label{fig:time all non commuting}
\end{figure}

\Cref{fig:non commuting time second order methods} demonstrates that among the three second-order schemes tested, SETDRK22 provides the best accuracy for a given CPU time. Its curve lies below that of SSP22 and eSSPIFSRK$^{+}22$ across the entire range. \Cref{fig:non commuting time third order methods}, demonstrates that amongst the three third-order schemes tested, SETDRK33 provides the best accuracy for a given CPU time. Its curve lies below that of eSSPIFSRK$^{+}33$ and SIFRK33 for the majority of CPU times; however SSP33 showed slight computational advantage at intermediate CPU times. \Cref{fig:non commuting time fourth order methods}, demonstrates that amongst the three fourth-order schemes tested (IFSRK4, SETDRK4, and SRK4) the CPU-ERROR plot collapse almost exactly onto a single curve (when using small timesteps). Indicating similar efficiency of all schemes when running at small timesteps.
In the same figure \cref{fig:non commuting time fourth order methods} SETDRK4 was more cost efficient than IFSRK4 and SRK4 in some small CPU time regimes (corresponding to using large timesteps). 
Overall \cref{fig:non commuting time second order methods,fig:non commuting time third order methods,fig:non commuting time fourth order methods} demonstrate that the computational overhead of exponential-type operators was justifiable, and the ETDRK schemes were overall more accurate per CPU time.

\Cref{fig:non commuting time all order methods} places all 9 methods on the same plot and scale. SETDRK4 is more accurate than SSP22 at all CPU times investigated. This takes into account the cost of using exponentials and more stages. In regimes corresponding to taking increasingly small timesteps (corresponding to the bottom right corner in \Cref{fig:non commuting time all order methods}) the advantage to using a higher order scheme is increasingly diminished for non-commutative noise when the stage cost is taken into account. Here eSSPIFSRK$^{+}$33 emerges as being more computationally efficient than SETDRK4. One could speculate that when using extremely small timesteps (smaller than $10^{-6}$) the trend observed in \cref{fig:non commuting time all order methods} would continue and second order methods such as SSP22 would eventually be more cost effective than ETDRK4 in the stochastic context since they have the same strong order but require fewer stages. However, in practice many applications require modest timesteps and for our examples the higher methods were overall the most cost effective, even when stage cost and exponentials were taken into account and the noise is noncommutative. In some regimes, fourth order ETDRK schemes deliver an order of magnitude smaller error than second-order and third order methods for the same CPU time. It is also clear that the CPU overhead of using exponential operators can be mitigated enough to be as efficient or more efficient than a SRK method, whilst giving the possibility of larger timesteps. CPU-time depends on implementation, software, hardware, we ran all experiments on a M2-mac using python with @jax.jit in JAX.

\subsection{Experiment 2: Commutative noise}\label{sec:ex2}

We consider the same IVP as \cref{eq:KdV multipliciative noise}, but define a $C^{\infty}(\mathbb{T})$ compactly supported basis of noise as follows: 
\begin{align}
\xi_j(x) = 
\mathbb{I}_{\{|x-c_j| < w/2\}} \;\exp\Bigg[-\frac{1}{1 - \big(\frac{2(x-c_j)}{w}\big)^2}\Bigg],\quad c_j = x_{\min} + j w,
\quad j = 1,\dots,M,
 \quad w = \frac{x_{\max}-x_{\min}}{M+1}, \quad M=3.
\end{align}
plotted in \cref{fig:Commutative basis}
\begin{figure}[H]
    \centering
\includegraphics[width=0.5\linewidth]{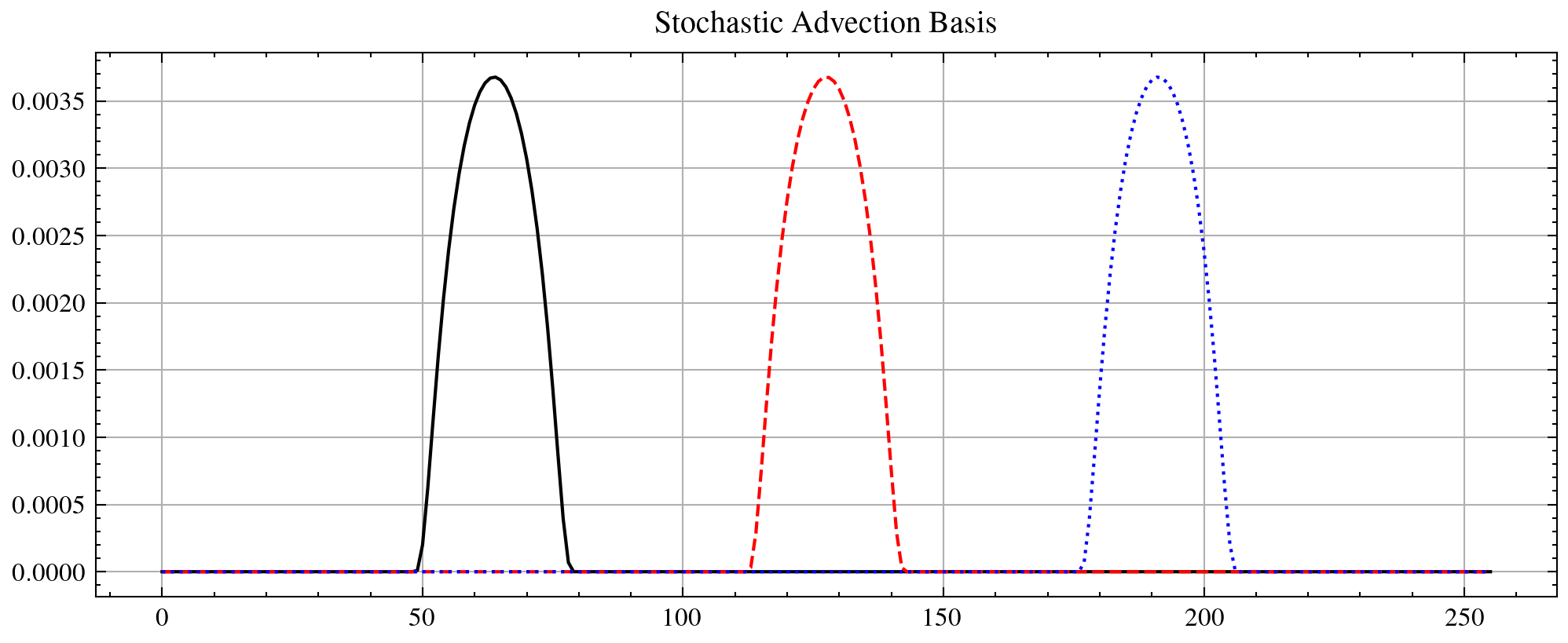}
    \caption{Commutative basis $\lbrace\xi_{m}\rbrace_{m=1,2,3}$.}
    \label{fig:Commutative basis}
\end{figure}
This setup does not admit an analytic path-wise solution. The noise commutes by distinct support in the Fréchet operator sense $[g_i,g_j]= [(\xi_i u)_x,(\xi_j u)_x] = 0$ when $i\neq j$ but does not drift commute $[f,g_j]\neq 0$, therefore \cref{thm:onestep}-\cref{item6} implies strong order 1 for all schemes. A high resolution solution $u_e$ is used as the reference solution. In figures \cref{fig:second order methods com,fig:third order methods com,fig: fourth order methods com}, we plot the convergence for the second order, third order, fourth order schemes, in \cref{fig:comparison of all com} all 9 methods relative error are plotted on the same graph.

\begin{figure}[H]
\centering
\begin{subfigure}{.495\textwidth}
    \centering
\includegraphics[width=.95\linewidth]{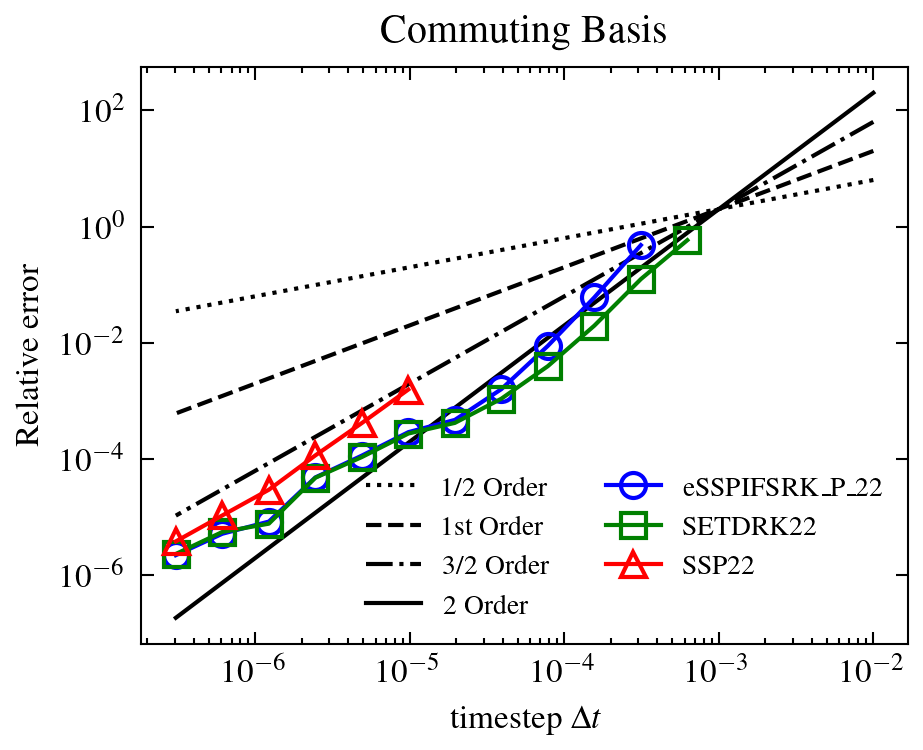}
\caption{Comparison of 2nd order methods commutative noise.}
\label{fig:second order methods com}
\end{subfigure}
\begin{subfigure}{.495\textwidth}
\centering
\includegraphics[width=.95\linewidth]{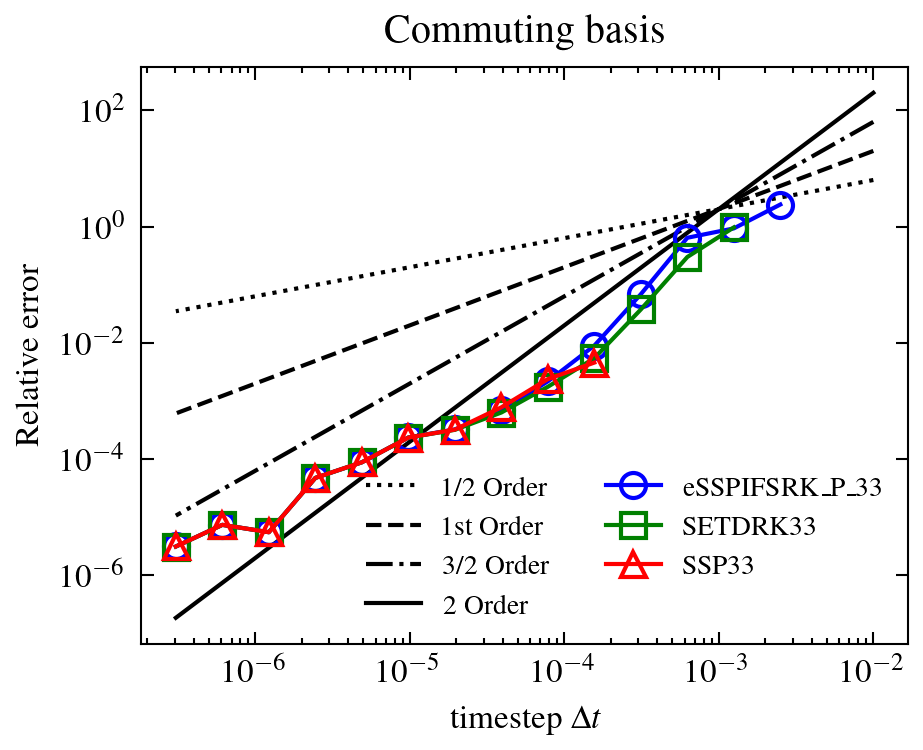}\caption{Comparison of 3rd order methods commutative noise.}
\label{fig:third order methods com}
\end{subfigure}\\
\begin{subfigure}{.495\textwidth}
    \centering
\includegraphics[width=.95\linewidth]{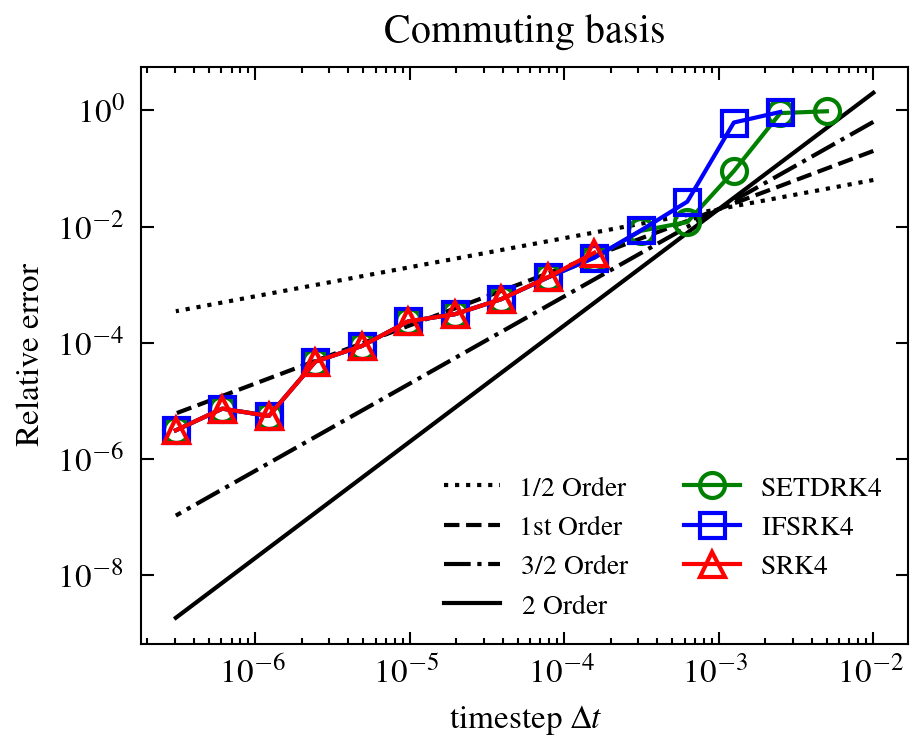}
\caption{Comparison of 4th order methods commutative noise.}
\label{fig: fourth order methods com}
\end{subfigure}
\begin{subfigure}{.495\textwidth}
\centering
\includegraphics[width=.95\linewidth]{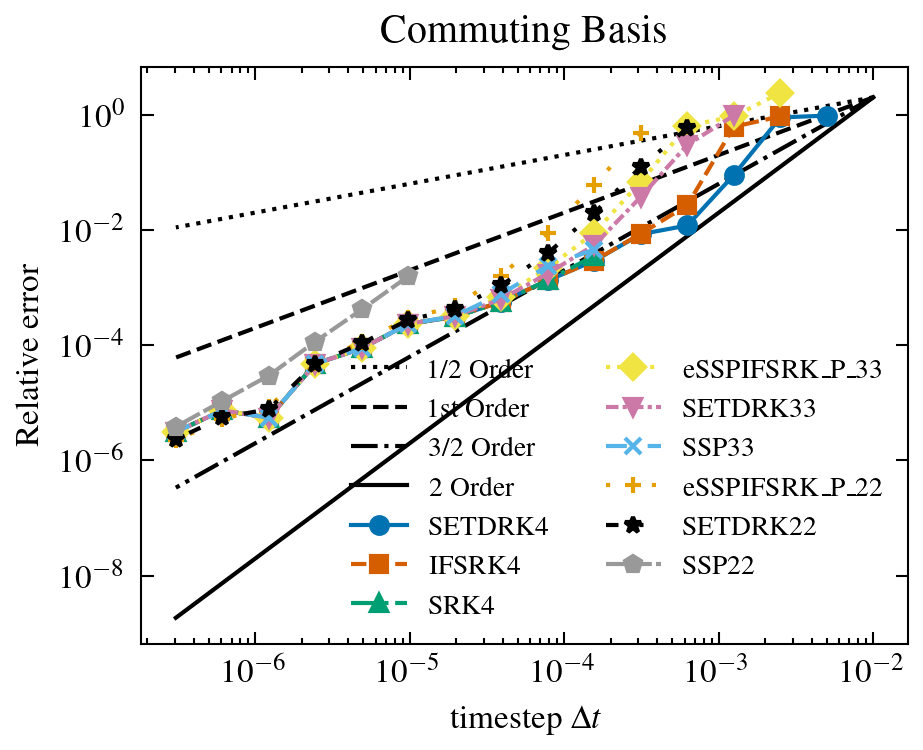}
\caption{Comparison of all methods commutative noise.}
\label{fig:comparison of all com}
\end{subfigure}
\caption{Commutative noise convergence order $1$ is verified \cref{thm:onestep}-\cref{item6}.}
\end{figure}
 Although \cref{fig:second order methods com,fig:third order methods com} appear to have better convergence than $\mathcal{O}(\Delta t^1)$ observed in \cref{fig: fourth order methods com}, this is because they have not reached the true order $\mathcal{O}(\Delta t^1)$ of convergence yet as demonstrated in \cref{fig:comparison of all com}.

ETDRK typically outperform IFRK methods and these outperform SRK methods in terms of not only accuracy but ability to run at larger timesteps, this was true independent of the order, but most apparent for the second order methods in \cref{fig:second order methods com}. In terms of accuracy, 4th outperforms 3rd, which outperforms 2nd, as is clear in \cref{fig:comparison of all com}, but the increase in accuracy becomes increasingly diminished in the limit of small timesteps. We verify \cref{thm:onestep}-\cref{item6}, $\mathcal{O}(\Delta t^1)$, but simultaneously that higher-order schemes are more accurate, as a consequence of \cref{item5,item2}

\subsection{Experiment 3: Drift Commutative noise.}\label{sec:ex3}

\Cref{thm:onestep} \cref{item3} and \cref{item4} indicates one may attain higher-order convergence for drift commutative noise, but due to the abstract form of the integration method, when $P$ is odd, the FTMSC was not sufficient to prove order $P/2$, but $P/2-1/2$ instead. The practical implications of such a theorem will be investigated numerically.

To test the theoretical order of the scheme, we use stochastic travelling wave solutions to the following stochastic KdV equation under constant advection noise. 
\begin{align}
d_t u  +uu_x + u_{xxx} + a u_x \circ dW_t= 0, \quad a\in \mathbb{R}.\label{eq:KdV constant noise}
\end{align}

By transforming into the stochastic moving frame $X(t,x):=x-aW(t)$. One (\cref{sec: stochastic travelling wave solutions to the KdV equation}) attains the analytic pathwise stochastic travelling wave solutions to the KdV equation
\begin{align}
u(x,t) = 3\beta \operatorname{sech}^2\left(\frac{\sqrt{\beta}}{2}(x-\beta t - aW(t))\right).\label{eq:stochastic travelling wave}
\end{align}

This particular setup not only admits analytic pathwise solutions, but also satisfies an operator version of the drift commutativity condition. Denote the drift and diffusion operators by 
\begin{align}
    F(u) = (u^2/2)_x + u_{xxx},\quad G(u) = a u_{x}.
\end{align}
The Fréchet derivatives with respect to $u$ in direction $v$ are
\begin{align}
DG(u)[v] = a v_x, \quad DF(u)[v] = (uv)_x + v_{xxx}.
\end{align}
These can be used to verify the infinite dimensional generalisation of the drift commutativity condition considered in \cref{thm:onestep}-\cref{item4,item3} as follows:
\begin{align}
[F,G] = DG(u)[F(u)] - DF(u)[G(u)]:= a((u^2/2)_{xx} + u_{xxxx}) - (uau_x)_x - a u_{xxxx}=0. 
\end{align}

In order to verify the strong convergence of the scheme
we compute a sequence of Brownian paths of decreasing resolution \cite{lord2014introduction}, and compare the relative L2 error of the numerical method in comparison to the analytic solution.  

We use the following parameters: 
$\Delta t = 10^{-4} (1/2^{i}),$ $i\in \lbrace 1,...,10 \rbrace$, 
$t_{max}=0.1$, $a=1$,
$n_{x}=256$,  
$u_0 = 3 \beta \operatorname{sech}^{2} ( \frac{\sqrt{\beta}}{2} x )$ where $\beta = 8^2$
and compare the relative spacetime $L^2$ error as between the numerical solution and the analytic solution generated using \cref{eq:stochastic travelling wave}. In \cref{fig:travelling wave solutions} the analytic and numerical solution are plotted at the highest resolution, in addition to the convergence plot on the right in the same figure. 
\begin{figure}[H]
\centering
\begin{subfigure}{.295\textwidth}
    \centering
\includegraphics[width=.95\linewidth]{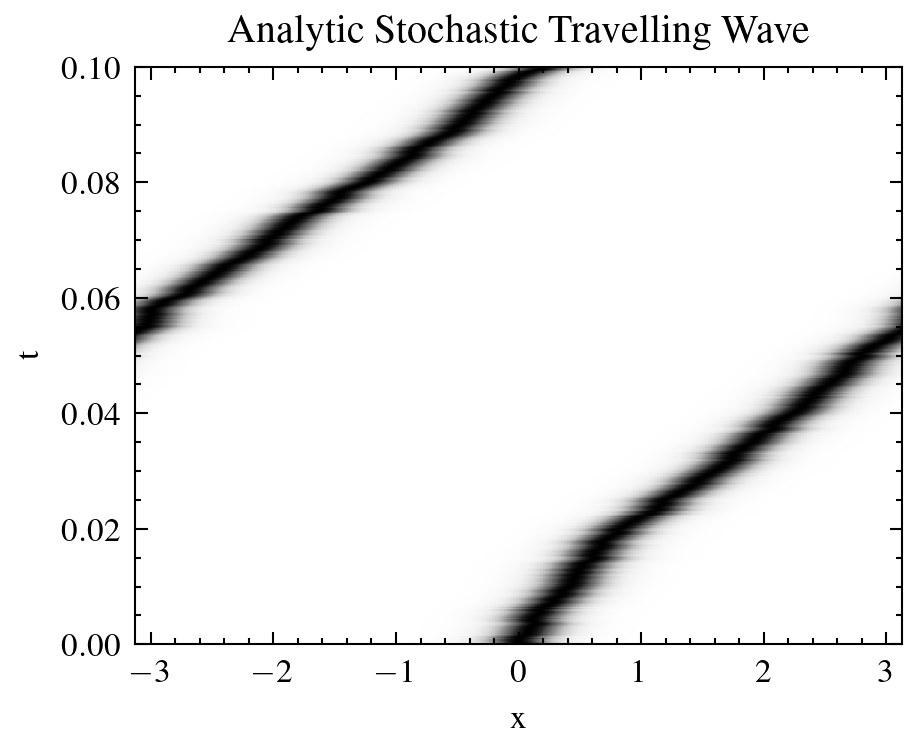}
\label{fig:Analytic travelling wave solution}
\end{subfigure}
\begin{subfigure}{.295\textwidth}
\centering
\includegraphics[width=.95\linewidth]{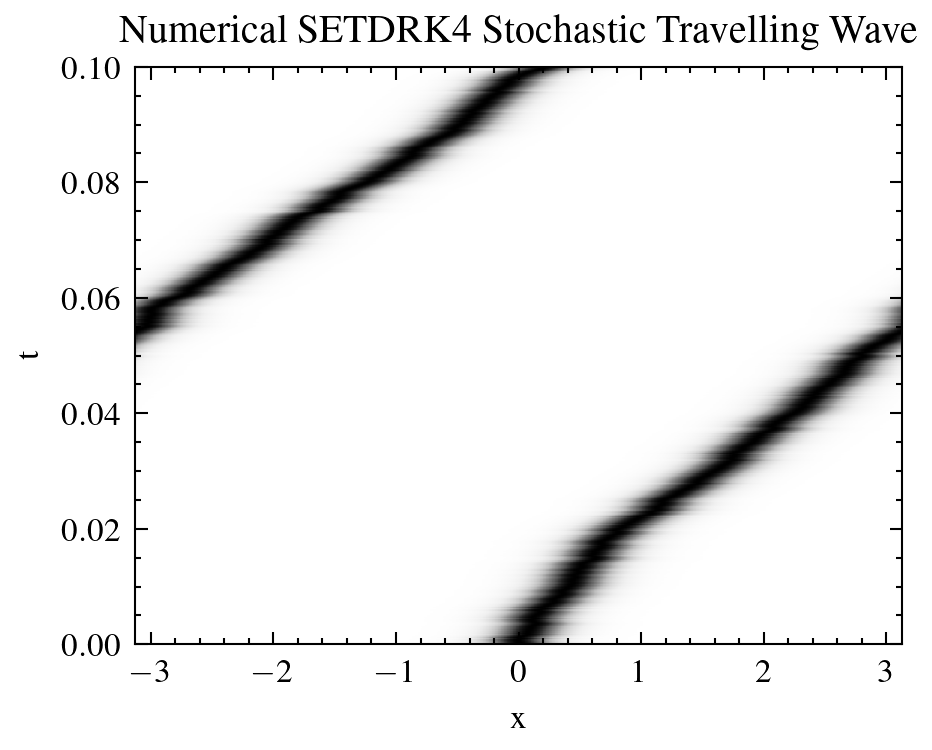}
\label{fig:Numerical travelling wave solution}
\end{subfigure}
\begin{subfigure}{.295\textwidth}
\centering
\includegraphics[width=.95\linewidth]{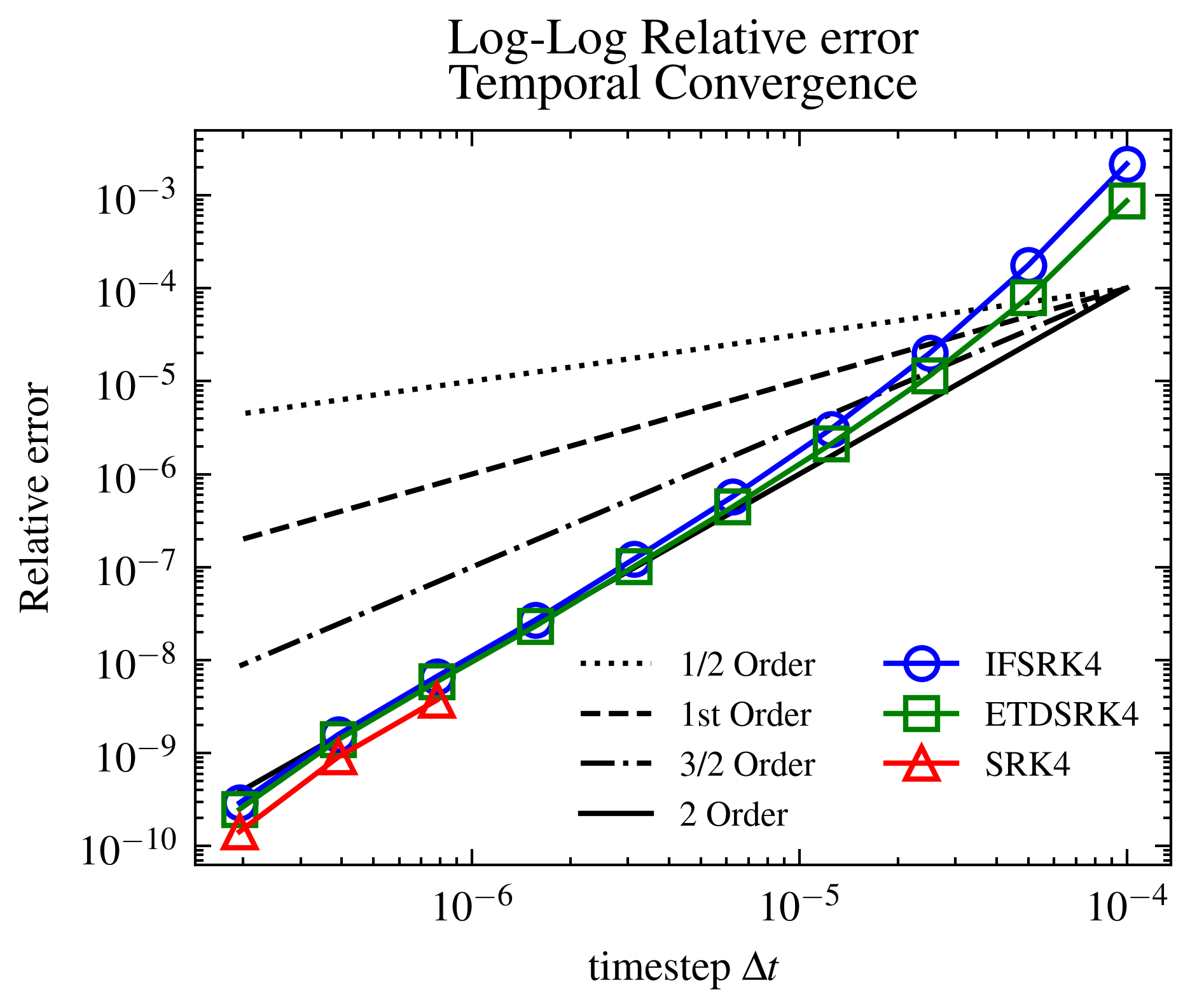}
\label{fig:convergence}
\end{subfigure}
\caption{Analytic stochastic travelling wave solution to the KdV equation (left) and the numerical solution when using dealiased spectral method in space, and the SETDRK4 scheme in time (middle). Temporal convergence (relative L2 spacetime error) of some 4-th order methods, are plotted  (right). }
\label{fig:travelling wave solutions}
\end{figure}

\Cref{fig:travelling wave solutions} verifies, all 4-th order schemes (SRK4, IFSRK4 and ETDRK4) attain the drift commuting mean square order $\mathcal{O}(\Delta t^{2})$ consistent with \cref{thm:onestep}-\cref{item3}. In \cref{fig:travelling wave solutions} the SETDRK4 scheme and the IFSRK4 scheme can run at timesteps at least $2^{7}$ times larger than the SRK4 scheme. In \cref{fig:travelling wave solutions} the SETDRK4 scheme is more accurate than the SIFRK4 scheme at all timesteps but less than SRK4 at small timesteps.

\subsection{Experiment 3a: Drift Commutative noise continued.}\label{sec:ex3a}

In this example we plot relative error for all 9 timestepping schemes over a wider range of possible timesteps. If a scheme is omitted from the plot at any timestep value, then the numerical scheme could not run stably at that temporal resolution. 
We use the following parameters: 
$\Delta t = 10^{-3} (1/2^{i}),$ $i\in \lbrace 1,...,14 \rbrace$, 
$t_{max}=0.01$, 
$n_{x}=256$, 
$\xi = 1$, 
$u_0 = 3 \beta \operatorname{sech}^{2} ( \frac{\sqrt{\beta}}{2} x )$ where $\beta = 9^2$, and use the final timestep to calculate the relative $L^2$ error from the analytic solution. In \cref{fig:second order methods} we plot the convergence of the second order methods. In \cref{fig:third order methods} we plot the convergence of the third order methods. In \cref{fig: fourth order methods} we plot the convergence of the fourth order methods. In \cref{fig:comparison of all} we overlay convergence results for all schemes onto one graph. 

\begin{figure}[H]
\centering
\begin{subfigure}{.495\textwidth}
    \centering
\includegraphics[width=.95\linewidth]{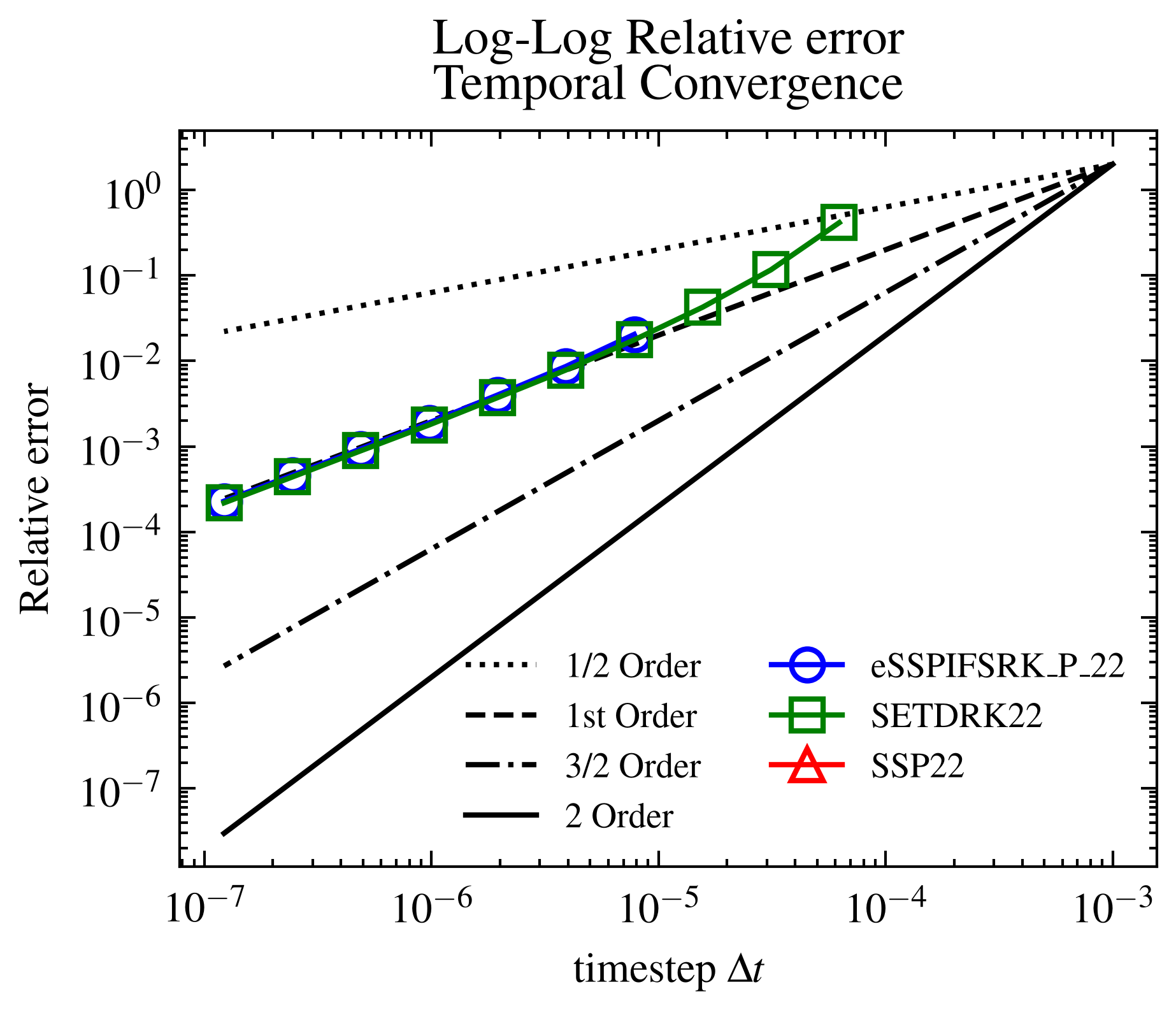}
\caption{Comparison of 2nd order methods, drift commutative noise.}
\label{fig:second order methods}
\end{subfigure}
\begin{subfigure}{.495\textwidth}
\centering
\includegraphics[width=.95\linewidth]{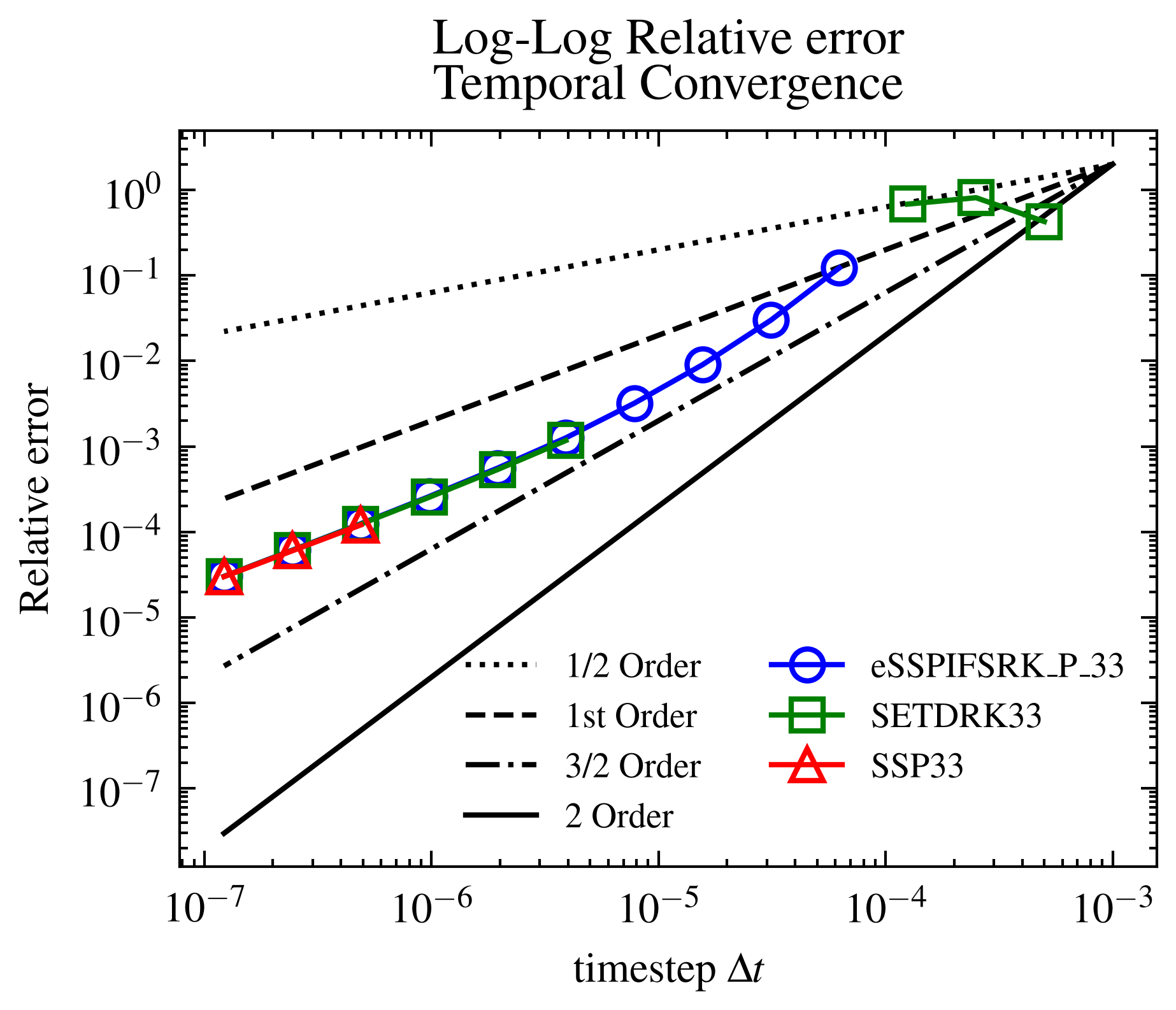}\caption{Comparison of 3rd order methods drift commutative noise.}
\label{fig:third order methods}
\end{subfigure}\\
\begin{subfigure}{.495\textwidth}
    \centering
\includegraphics[width=.95\linewidth]{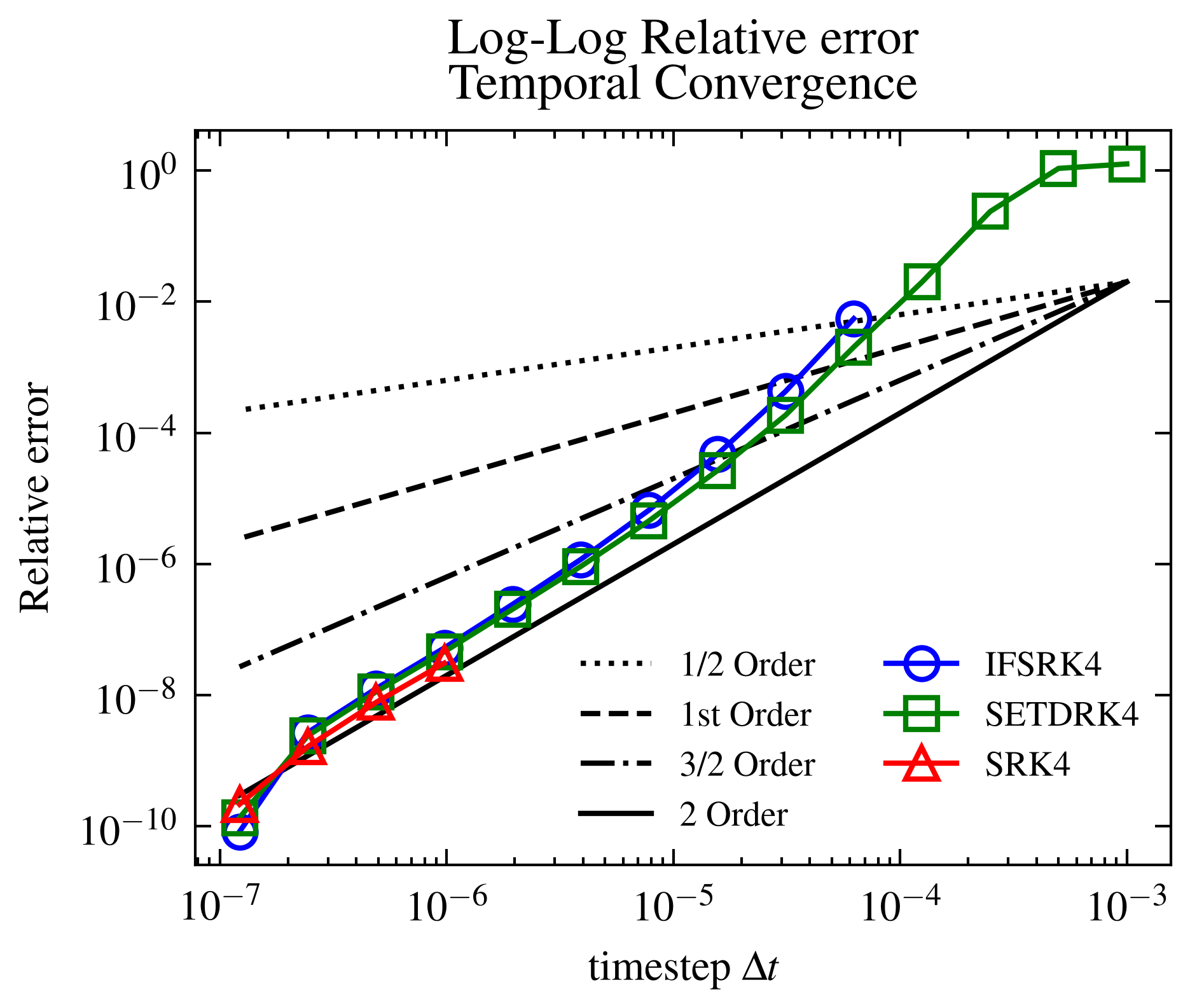}
\caption{Comparison of 4-th order methods drift commutative noise.}
\label{fig: fourth order methods}
\end{subfigure}
\begin{subfigure}{.495\textwidth}
\centering
\includegraphics[width=.95\linewidth]{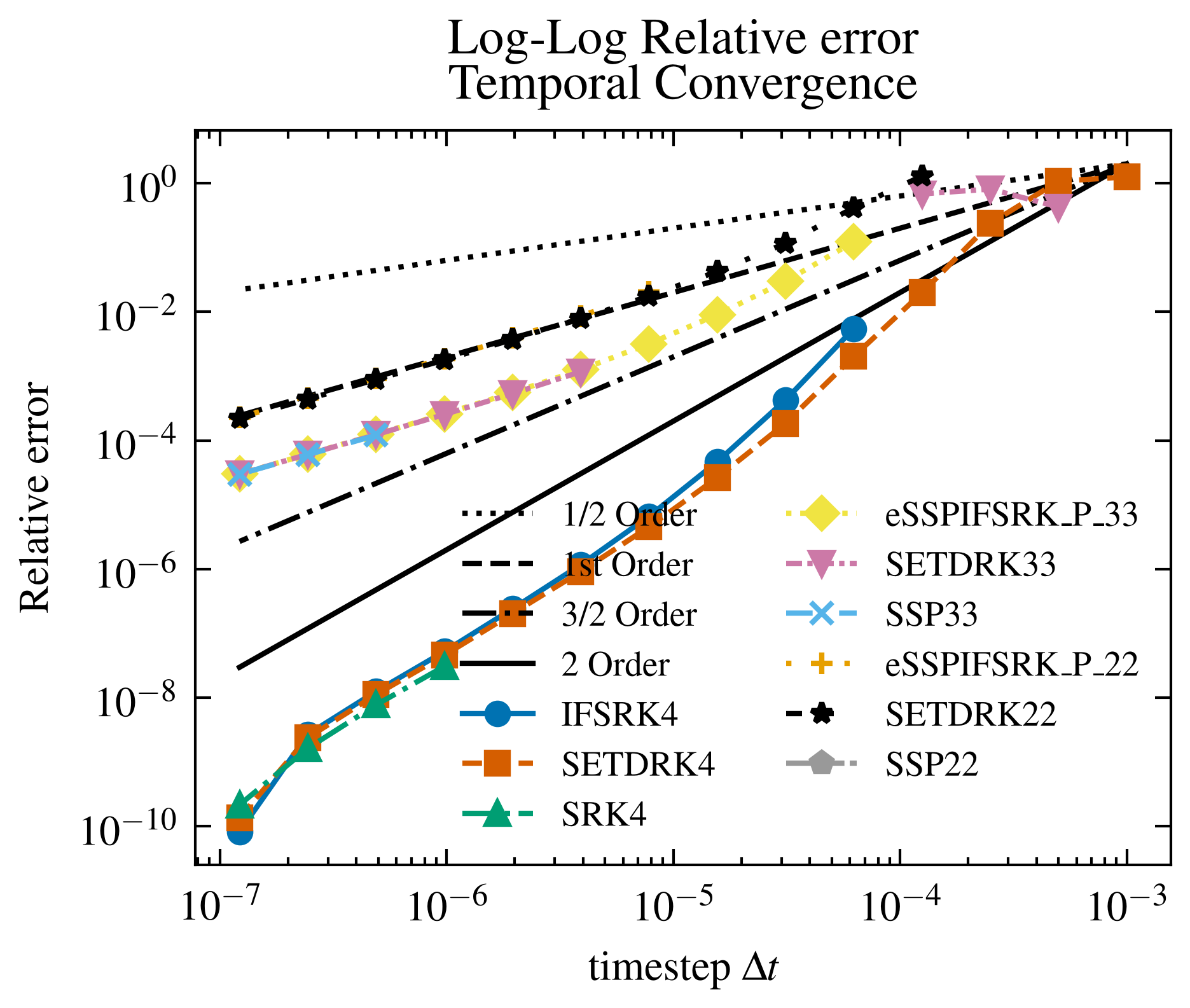}
\caption{Comparison of all methods drift commutative noise.}
\label{fig:comparison of all}
\end{subfigure}
\caption{}
\label{fig:second and third order methods}
\end{figure}

In \cref{fig:comparison of all}, we see that \Cref{thm:onestep}-\cref{item3} and \Cref{thm:onestep}-\cref{item4} are observed in practice for SETDRK, SIFRK and SRK time integration schemes. 
\Cref{thm:onestep}-\cref{item3} predicts that for the second order methods $\mathcal{O}(\Delta t^1)$ convergence and for the fourth order methods $\mathcal{O}(\Delta t^2)$ convergence. \Cref{thm:onestep}-\cref{item4} predicts that for the third order methods $\mathcal{O}(\Delta t^1)$ convergence.

In \cref{fig:second order methods} stochastic SSP22 did not converge for any timestep, eSSPIFSRK22 converged with strong order $\mathcal{O}(\Delta t^1)$, SETDRK22 also converged with strong order $\mathcal{O}(\Delta t^1)$ and was marginally more accurate than eSSPIFRK22 and ran at a timestep 8 times as large.

In \cref{fig:third order methods} SSP33, SETDRK33 and eSSPIFSRK33 converged with strong order $\mathcal{O}(\Delta t^1)$ as theoretically predicted in \cref{thm:onestep}-\cref{item4}. The SSP33 scheme was only stable for small timesteps. SETDRK and eSSPIFSRK methods were capable of larger timesteps than the SSP33 method. As shown in \cref{fig:third order methods} eSSPIFSRK33 was capable of running at timesteps SETDRK33 could not and vice versa.

In \cref{fig: fourth order methods} SRK4, SETDRK4 and SIFRK4 all converged with strong order $\mathcal{O}(\Delta t^2)$ for small timesteps, as theoretically predicted in \cref{item3} in \cref{thm:onestep}.  In \cref{fig: fourth order methods} at the few points SRK4 converged it was the most accurate, SETDRK4 was more accurate than the IFRK4 scheme at all timesteps. IFRK4 was capable of timesteps over $2^{6}$ larger than SRK4. SETDRK4 was capable of timesteps $2^{10}$ larger than the SRK4. In \cref{fig: fourth order methods} the SETDRK4 schemes was calculated as having slope order $3.6$ at large timesteps.

In \cref{fig:comparison of all} we overlay all the plots on the same figure. We see significant practical improvements in accuracy by using deterministic fourth order accurate schemes as opposed to deterministic second order ones or third order ones, for instance SETDRK4 attained similar error to eSSPIFSRK33 and SETDRK3 at a timestep $2^7$ larger. 

We run the convergence tests at each temporal instance 5 times and average for a the CPU-time. The CPU-time is plotted on the x-axis and the relative error is plotted on the y-axis. In \cref{fig:time all}, we observe the $\operatorname{int}(P/2)$ order behaviour even when taking into account the CPU-time. This indicates that the increase in accuracy is not lost when taking into account the stage cost and the exponential evaluation cost. 

\begin{figure}[H]
\centering
\begin{subfigure}{.495\textwidth}
    \centering
\includegraphics[width=.95\linewidth]{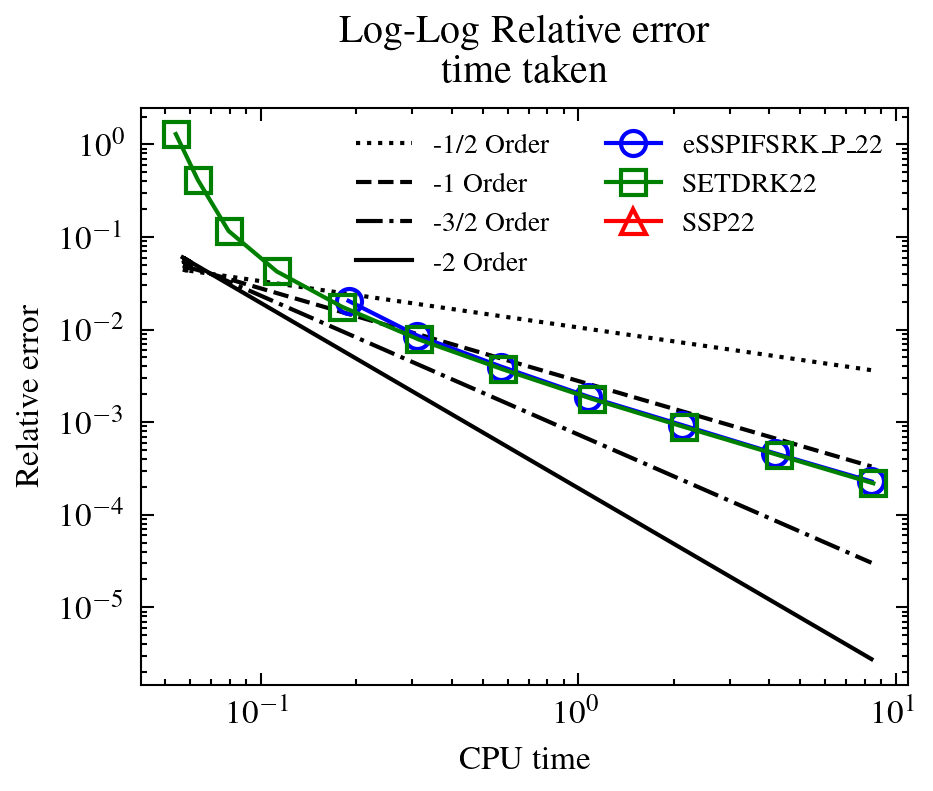}
\caption{Comparison of 2nd order methods CPU-time vs Relative error.}
\label{fig:time second order methods}
\end{subfigure}
\begin{subfigure}{.495\textwidth}
\centering
\includegraphics[width=.95\linewidth]{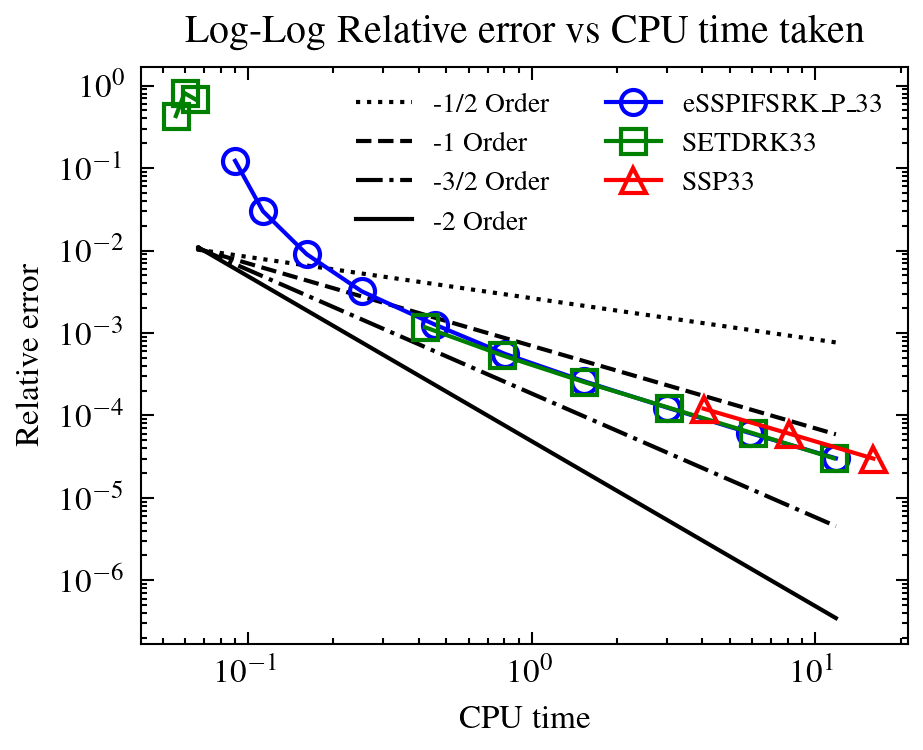}\caption{Comparison of 3rd order methods CPU-time vs Relative error.}
\label{fig:time third order methods}
\end{subfigure}\\
\begin{subfigure}{.495\textwidth}
    \centering
\includegraphics[width=.95\linewidth]{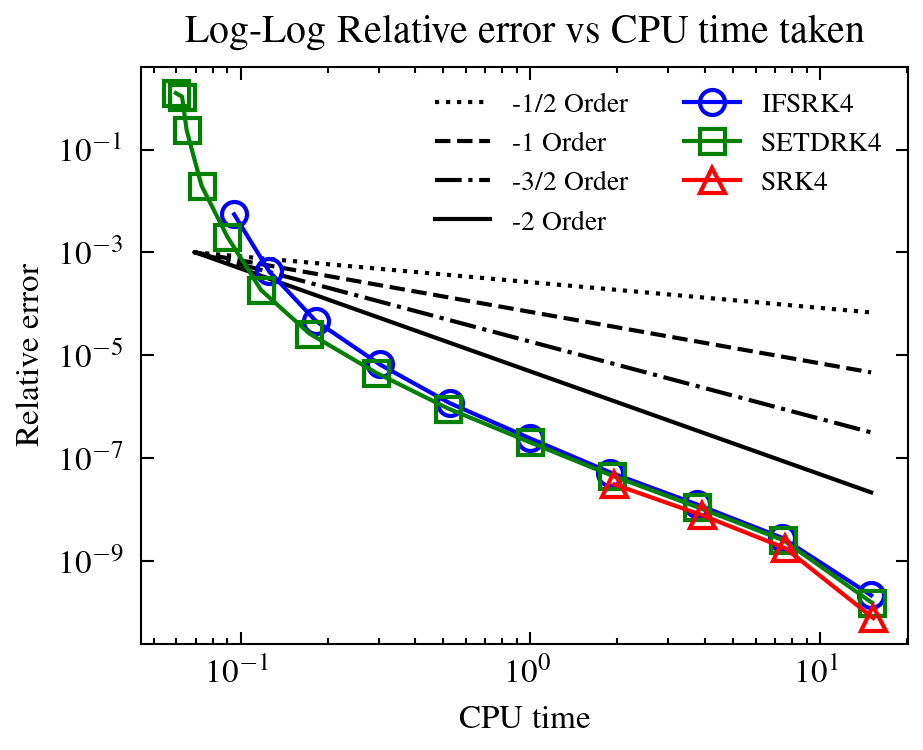}
\caption{Comparison of 4th order methods CPU-time vs Relative error.}
\label{fig:time fourth order methods}
\end{subfigure}
\begin{subfigure}{.495\textwidth}
\centering
\includegraphics[width=.95\linewidth]{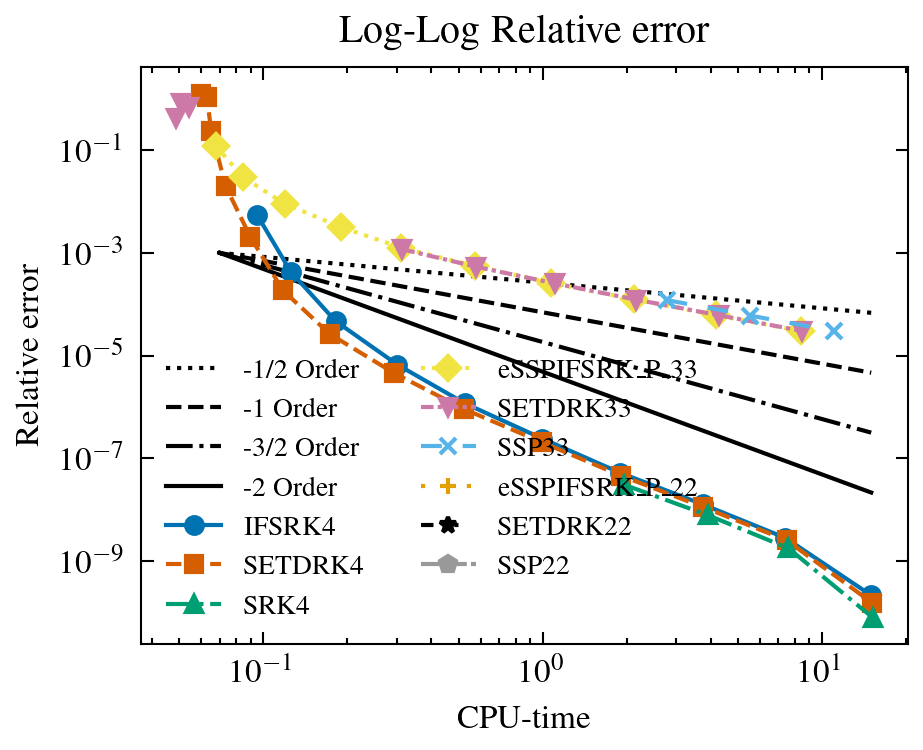}\caption{Comparison of all methods CPU-time vs Relative error.}
\label{fig:time all order methods}
\end{subfigure}
\caption{CPU-time vs Relative error, drift commutative basis. this experiment demonstrates that the ``increase" in computational cost (by using more stages or exponentials) does not offset the decreasing error. By plotting CPU time instead of temporal resolution on the x-axis, stage and exponential cost is taken into account. We observe theoretically predicted rates of relative error decrease, consistent with assuming that the CPU-cost of stages scales linearly and the cost of the exponential operator is done before runtime, i.e. we attain $\mathcal{O}(1)$, $\mathcal{O}(1)$, $\mathcal{O}(2)$, decreases in error as a function of runtime for deterministic order $2,3,4$ schemes predicted by \cref{thm:onestep}. }
\label{fig:time all}
\end{figure}

\subsection{Experiment 4: Decreasing noise magnitude and large timestep order.}\label{sec:ex4}

We solve the SETDRK4 scheme under the drift commutative setup but with decreasing noise magnitudes $a = \lbrace 1,1/2,1/4,1/8 \rbrace$, the results are plotted in \cref{fig:decreasing noise magnitude} where we observe deterministic order $4$, emerging in the limit $\xi = a \rightarrow 0$ as expected from \cref{item5}, in \cref{thm:onestep}. 

In \cref{fig:Large timesteps}, we plot a convergence with moderate magnitude noise. In \cref{fig:Large timesteps} at large timesteps the deterministic term in the SDE dominates the stochastic term (i.e. $f\Delta t >> g \Delta W$), such that when taking large timesteps deterministic $\mathcal{O}(\Delta t^4)$ convergence behaviour can play a role in the error. This is observed in \cref{fig:Large timesteps} when employing SETDRK or SIFRK schemes which can take larger timesteps. As shown in the previous convergence plots (\cref{fig:all order non-commutative basis}) high order deterministic convergence occurs at large timesteps independent of whether the noise is commutative or drift commutative further exemplifying the importance of capturing higher order deterministic dynamical terms (\cref{item5} from \cref{thm:onestep}).

\begin{figure}[H]
\centering
\begin{subfigure}{.495\textwidth}
    \centering
\includegraphics[width=.95\linewidth]{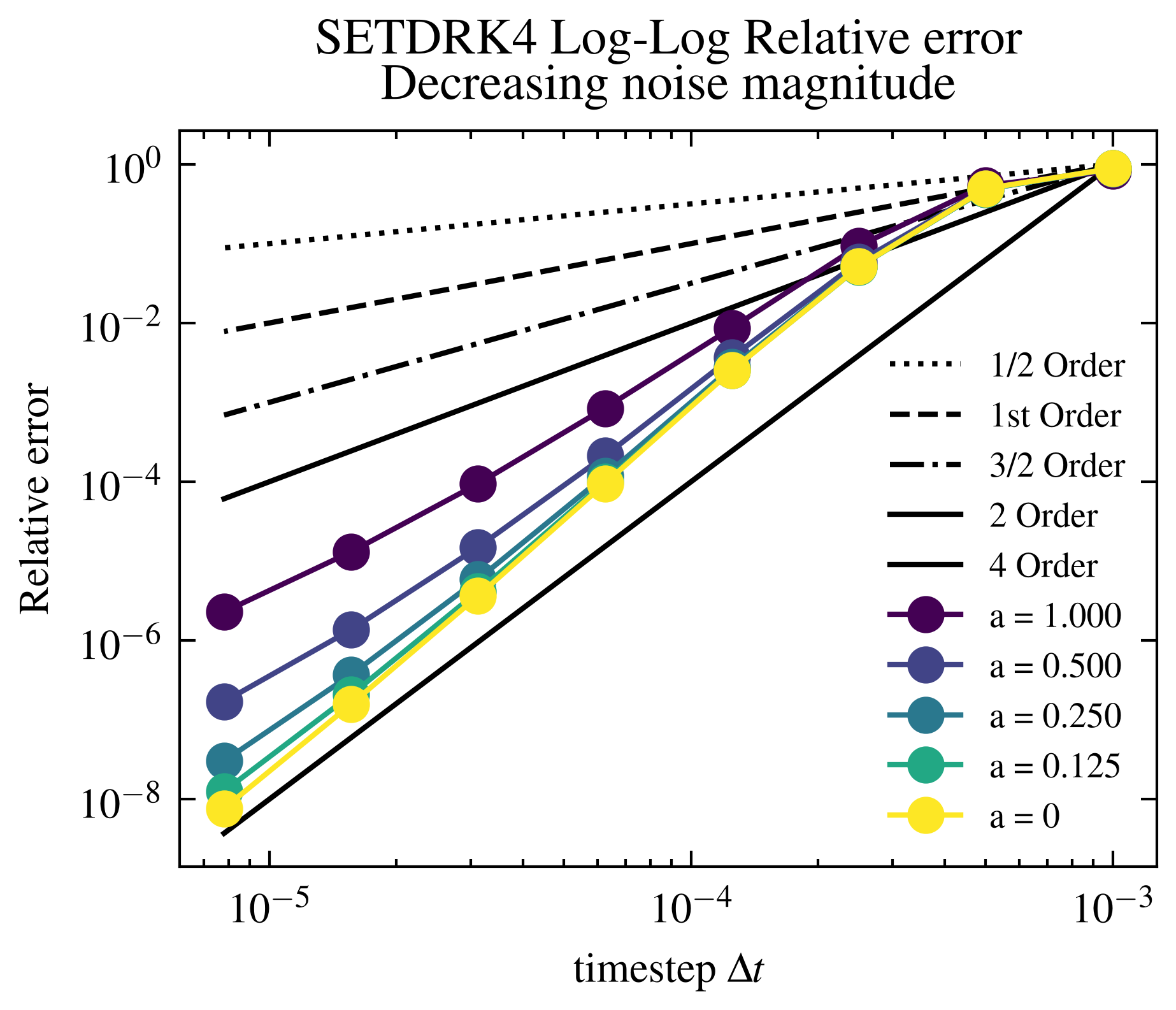}
\caption{Decreasing noise magnitude}
\label{fig:decreasing noise magnitude}
\end{subfigure}
\begin{subfigure}{.495\textwidth}
\centering
\includegraphics[width=.95\linewidth]{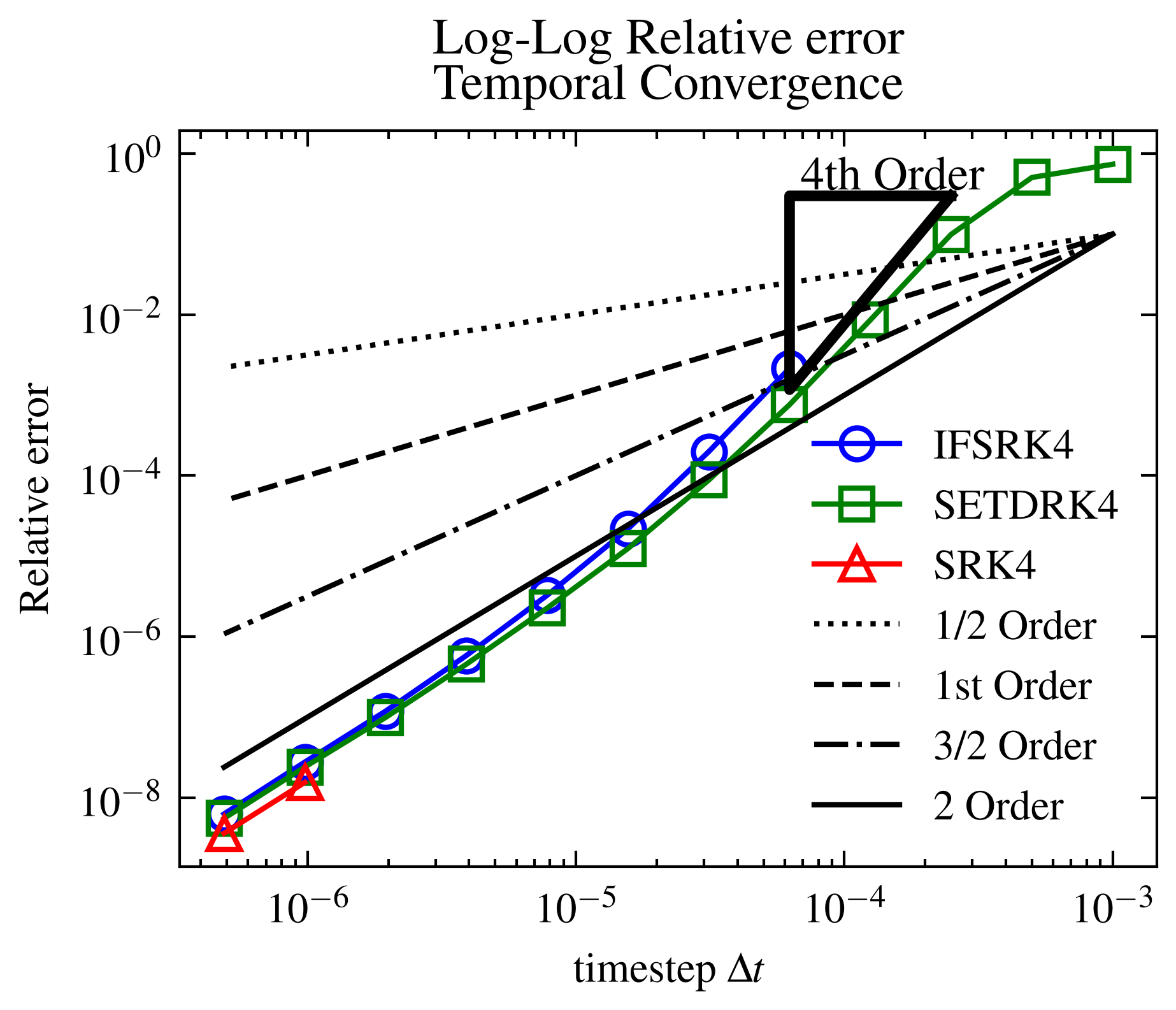}
\caption{Large timesteps}
\label{fig:Large timesteps}
\end{subfigure}
\caption{Higher order convergence associated with small noise or large timesteps. }
\label{fig:small_noise_and_large_timesteps}
\end{figure}

\subsection{Experiment 5: 2D Incompressible Navier Stokes Equation.}\label{sec:ex5}

This example simply indicates SETDRK methods are adoptable to higher-dimensional fluid-dynamics applications, and may be helpful to avoid timestep restrictions associated with linear stiff operators. The 2D incompressible Navier-Stokes equation in vorticity formulation under a stochastic perturbation to the momentum in the Lie-Poisson bracket takes the following form:
\begin{align}
d q + \underbrace{J\left( \psi , q dt + \sum_{m=1}^{M}\xi_m(x,y) \circ dW^m\right)}_{N(q)dt+g_m(q) \circ dW^m} = \underbrace{\mu \Delta q - \nu \Delta^2 q}_{-Lu}, \quad J(a,b):= a_x b_y - b_x a_y, \quad  \psi = -\Delta^{-1}q , \quad u = -\nabla^{\perp}\psi. \label{eq:NS}
\end{align}
In the absence of viscosity and hyperviscosity $(\mu,\nu)=(0,0)$ this type of stochastic perturbation (originally appearing more generally in \cite{holm2021stochastic} under the acronym SFLT) preserves energy but also the first integral of vorticity since $J(a,b) = \lbrace a,b\rbrace = \operatorname{div}(-[\nabla^{\perp}a]b)$ and
\begin{align}
d_t \left(\frac{1}{2}\int_{\Omega} u^2 d^2 x\right) &= d_t \left(\frac{1}{2}\int_{\Omega} -\psi q d^2 x\right) =  - \frac{1}{2}\int J(\psi^2,q dt+ \sum_{m=1}^{M}\xi_m \circ dW ) = 0,\quad d_t \left(\int_{\Omega} q d^2 x\right) =0.
\end{align}
 
Therefore, with diffusion and hyperdiffusion the above stochastic perturbations do not change the usual global 2D deterministic energy balance of the NS equation and the stochastic system obeys:
\begin{align}
    d_t \frac{1}{2}||u||_{2}^2 = - \mu||\nabla u||_{2}^2-  \nu||\Delta u||_{2}^2\leq 0.
\end{align}
But allows spread per ensemble member. Viscosity dissipates energy, and hyper-viscosity controls unresolved small scales; however, unless $\mu,\nu$ are adjusted carefully to be small and proportional to the grid, both diffusion and hyper-diffusion can result in additional CFL restrictions for explicit schemes. Under a flux-form de-aliased Fourier spatial discretisation of \cref{eq:NS}, one identifies non-linear contributions $N,g_{m}$ and linear contributions $Lu$ which can be treated with SETDRK4 avoiding additional CFL restrictions by treating diffusion and hyperdiffusion exponentially. We employ the Kassam-Trefethen \cite{kassam2005fourth} contour integration in the complex plane for $E_{i}$ and the SETDRK4 \cref{method:SETDRK4}, equispaced snapshots of the solution for $t\in[0,80]$ are plotted in \cref{fig:NS-SFLT} where the stochastic perturbation destabilises the deterministic symmetric flow pattern. 

\begin{figure}[H]
\centering
\includegraphics[width=1\linewidth]{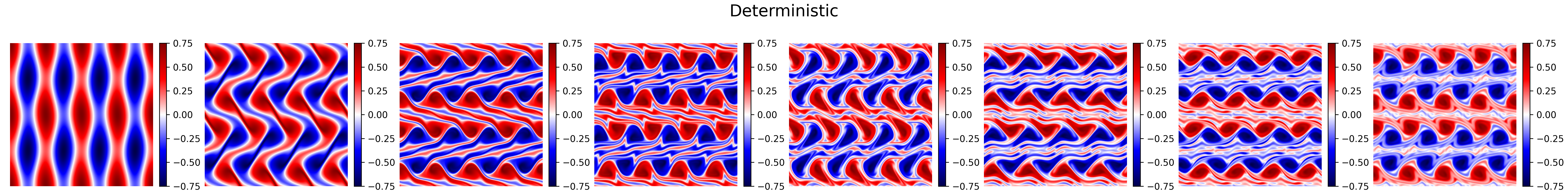}\\
\includegraphics[width=1\linewidth]{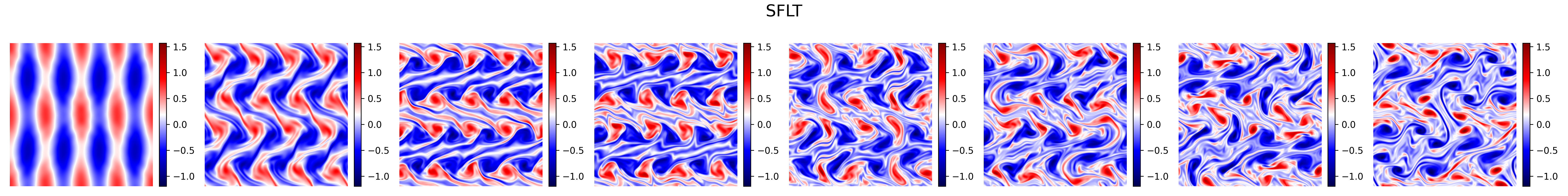}
\caption{Comparison between Deterministic NS and one realisation of a SFLT perturbed NS ensemble.}
\label{fig:NS-SFLT}
\end{figure}

\section{Conclusion}\label{sec:conclusion}

Theoretically we prove and numerically demonstrate that deterministic IFRK–ETDRK–RK integrators can be modified to converge to a Stratonovich SDE. For general non-commutative noise, the methods achieve mean-square order $\mathcal{O}(\Delta t^{1/2})$ (\cref{thm:onestep}-\cref{item1}). In the commutative noise case, the order improves to $\mathcal{O}(\Delta t^{1})$ (\cref{thm:onestep}-\cref{item6}), while in the presence of drift commutativity the order further increases to $\mathcal{O}(\Delta t^{\operatorname{int}(p/2)})$ (\cref{thm:onestep}-\cref{item3}–\cref{item4}). Finally, in the limit of small noise, large timestep $\Delta t$, or absence of noise, deterministic order $\mathcal{O}(\Delta t^{p})$ is recovered (\cref{thm:onestep}-\cref{item5}). 

In comparing SETDRK, SIFRK, and SRK schemes of the same deterministic order, we found that SETDRK and SIFRK methods were able to take timesteps typically orders of magnitude larger than those permitted by SRK methods. Among these approaches, SETDRK methods were generally more accurate than SIFRK methods. Furthermore, the additional overhead associated with exponential operator evaluations in both SETDRK and SIFRK methods was negligible, since these terms were precomputed prior to runtime, making their cost justifiable when compared to SRK methods of the same order.

In comparing SRK, SIFRK, and SETDRK methods of different deterministic orders, our experiments demonstrate that higher-order methods provide clear benefits at large timesteps, for small noise amplitudes, and under commutativity conditions. For non-commutative noise where second-order methods could be theoretically expected to be more cost-effective in the limit $\Delta t \to 0$ due to the dominant missing Lévy area term and reduced stage cost, we numerically observed higher-order methods achieving better accuracy per unit CPU cost. These results indicate that higher-order SRK, SIFRK, and SETDRK schemes remain competitive and effective in practical applications where timesteps cannot be taken arbitrarily small, or where the noise is small. 

Several open directions follow from this work. The formulation is not intrinsically restricted to classical B-series integrators, aspects of \cref{thm:onestep} can accommodate one-step methods whose algebraic descriptions may be different, including splitting methods, nonlinear projection methods, and non-B-series discrete-gradient schemes. The development of generalised additive SGARK, SIFGARK, and SETDGARK schemes, allowing separate treatment of drift and diffusion within the generalised additive Runge–Kutta framework \cite{sandu2015generalized,ruemelin1982numerical}, is a promising avenue for improving efficiency. Particularly, detailed studies using splitting methods could be used to investigate the importance of capturing the symmetric terms in the Stratonovich-Taylor expansion as opposed to high order deterministic terms. Extensions of strong stability preserving (SSP) theory to stochastic integrating factor methods (SSP-SIFRK) and to stochastic exponential time-differencing Runge–Kutta methods (SSP-SETDRK) remain to be established following \cite{isherwood2018strong,qin2025energy,woodfield2024strong}. Stochastic analogues of A-stability \cite{kloeden2011numerical} have not yet been developed for some of the methods introduced including the SETDRK4 scheme. The consequences of stage order reduction, well understood for deterministic ETDRK methods \cite{strehmel1987b,krogstad2005generalized,hochbruck2010exponential}, requires investigation in the stochastic context.

While our convergence analysis focused on the stochastic Korteweg–de Vries equation, the methodology extends naturally to other stiff SPDEs, including the Kuramoto–Sivashinsky, Heat, Burgers, Navier–Stokes, Complex Ginzburg–Landau, and Quasi-geostrophic equations under stochastic transport and forcing. Providing an opportunity for easily implementable stochastic integration for ensemble forecasting and uncertainty quantification methodology.

\section{Data availability}
The figures, data, and code used in the creation of this document can be found in the following library \url{https://github.com/jameswoodfield/Particle_Filter}.

\section*{Acknowledgements}

Grateful to R. Hu, for valuable comments leading to the improvement of this manuscript, and enjoyable discussions with D. Crisan, D. Holm, O. Street, M. Singh and E. Fausti. Both J. Woodfield and A. Lobbe have been supported during the present work by the European Research Council (ERC) Synergy grant ``Stochastic Transport in Upper Ocean Dynamics" (STUOD) -- DLV-856408.

\bibliography{References.bib}

@phdthesis{shmatkov2006rate,
  author = {Shmatkov, Anton},
  title = {Rate of Convergence of Wong--Zakai Approximations for SDEs and SPDEs},
  school = {University of Edinburgh},
  year = {2006}
}

@article{sengul2021performance,
  title={The performance of Wong-Zakai approximations for the investigation of stochastic differential equation models with nonlinear multiplicative noise},
  author={Sengul, Suleyman and Bekiryazici, Zafer and Merdan, Mehmet},
  journal={Acta Mathematica Universitatis Comenianae},
  volume={90},
  number={2},
  pages={231--243},
  year={2021}
}

@article{londono2016numerical,
  title={Numerical performance of some Wong-Zakai type approximations for stochastic differential equations},
  author={Londono, Jaime A and Villegas, Andr{\'e}s M},
  journal={International Journal of Pure and Applied Mathematics},
  volume={107},
  number={2},
  pages={301--315},
  year={2016}
}

@article{wong1965convergence,
  title={On the convergence of ordinary integrals to stochastic integrals},
  author={Wong, Eugene and Zakai, Moshe},
  journal={The Annals of Mathematical Statistics},
  volume={36},
  number={5},
  pages={1560--1564},
  year={1965},
  publisher={JSTOR}
}

@article{eugene1965relation,
  title={On the relation between ordinary and stochastic differential equations},
  author={Eugene, Wong and Moshe, Zakai},
  journal={International Journal of Engineering Science},
  volume={3},
  number={2},
  pages={213--229},
  year={1965},
  publisher={Elsevier}
}

@article{ikeda1977comparison,
  author  = {Ikeda, Nobuyuki and Watanabe, Shinzo},
  title   = {{A Comparison Theorem for Solutions of Stochastic Differential Equations and Its Applications}},
  journal = {Osaka Journal of Mathematics},
  volume  = {14},
  number  = {3},
  pages   = {619--633},
  year    = {1977},
  doi     = {10.18910/7664}
}

@book{butcher2016numerical,
  title={Numerical methods for ordinary differential equations},
  author={Butcher, John Charles},
  year={2016},
  publisher={John Wiley \& Sons}
}

@book{hairer2006geometric,
  author    = {Hairer, Ernst and Lubich, Christian and Wanner, Gerhard},
  title     = {{Geometric Numerical Integration: Structure-Preserving Algorithms for Ordinary Differential Equations}},
  edition   = {2},
  series    = {Springer Series in Computational Mathematics},
  volume    = {31},
  publisher = {Springer},
  address   = {Berlin, Heidelberg},
  year      = {2006},
  doi       = {10.1007/3-540-30666-8}
}

@book{gottlieb2011strong,
  title={Strong stability preserving Runge-Kutta and multistep time discretizations},
  author={Gottlieb, Sigal and Ketcheson, David I and Shu, Chi-Wang},
  year={2011},
  publisher={World Scientific}
}

@article{jardak2010comparison,
  title={{Comparison of sequential data assimilation methods for the Kuramoto--Sivashinsky equation}},
  author={Jardak, M and Navon, IM and Zupanski, M},
  journal={International journal for numerical methods in fluids},
  volume={62},
  number={4},
  pages={374--402},
  year={2010},
  publisher={Wiley Online Library}
}

@article{dai2023exponential,
  title={{Exponential time differencing-Pad'e finite element method for nonlinear convection-diffusion-reaction equations with time constant delay}},
  author={Dai, Haishen and Wang, C and Huang, Q},
  journal={Journal of computational mathematics},
  volume={41},
  number={3},
  year={2023}
}

@article{lord2013stochastic,
  title={{Stochastic exponential integrators for the finite element discretization of SPDEs for multiplicative and additive noise}},
  author={Lord, Gabriel J and Tambue, Antoine},
  journal={IMA Journal of Numerical Analysis},
  volume={33},
  number={2},
  pages={515--543},
  year={2013},
  publisher={OUP}
}

@article{burrage1996high,
  title={High strong order explicit Runge-Kutta methods for stochastic ordinary differential equations},
  author={Burrage, Kevin and Burrage, Pamela Marion},
  journal={Applied Numerical Mathematics},
  volume={22},
  number={1-3},
  pages={81--101},
  year={1996},
  publisher={Elsevier}
}

@article{burrage1998general,
  title={General order conditions for stochastic Runge-Kutta methods for both commuting and non-commuting stochastic ordinary differential equation systems},
  author={Burrage, Kevin and Burrage, PM},
  journal={Applied Numerical Mathematics},
  volume={28},
  number={2-4},
  pages={161--177},
  year={1998},
  publisher={Elsevier}
}

@book{higham2002accuracy,
  title={{Accuracy and stability of numerical algorithms}},
  author={Higham, Nicholas J},
  year={2002},
  publisher={SIAM}
}

@article{sandu2015generalized,
  title={A generalized-structure approach to additive Runge--Kutta methods},
  author={Sandu, Adrian and Günther, Michael},
  journal={SIAM Journal on Numerical Analysis},
  volume={53},
  number={1},
  pages={17--42},
  year={2015},
  publisher={SIAM}
}

@book{iserles2009first,
  title={{A first course in the numerical analysis of differential equations}},
  author={Iserles, Arieh},
  year={2009},
  publisher={Cambridge university press}
}

@article{holm2021stochastic,
  title={Stochastic effects of waves on currents in the ocean mixed layer},
  author={Holm, Darryl D and Hu, Ruiao},
  journal={Journal of Mathematical Physics},
  volume={62},
  number={7},
  year={2021},
  publisher={AIP Publishing}
}

@article{cox2002exponential,
  title={{Exponential time differencing for stiff systems}},
  author={Cox, Steven M and Matthews, Paul C},
  journal={Journal of Computational Physics},
  volume={176},
  number={2},
  pages={430--455},
  year={2002},
  publisher={Elsevier}
}

@article{kassam2005fourth,
  title={{Fourth-order time-stepping for stiff PDEs}},
  author={Kassam, Aly-Khan and Trefethen, Lloyd N},
  journal={SIAM Journal on Scientific Computing},
  volume={26},
  number={4},
  pages={1214--1233},
  year={2005},
  publisher={SIAM}
}

@article{jentzen2009overcoming,
  title={{Overcoming the order barrier in the numerical approximation of stochastic partial differential equations with additive space--time noise}},
  author={Jentzen, Arnulf and Kloeden, Peter E},
  journal={Proceedings of the Royal Society A: Mathematical, Physical and Engineering Sciences},
  volume={465},
  number={2102},
  pages={649--667},
  year={2009},
  publisher={The Royal Society London}
}

@article{geiger2012exponential,
  title={{Exponential time integrators for stochastic partial differential equations in 3D reservoir simulation}},
  author={Geiger, Sebastian and Lord, Gabriel and Tambue, Antoine},
  journal={Computational Geosciences},
  volume={16},
  pages={323--334},
  year={2012},
  publisher={Springer}
}

@article{strehmel1987b,
  title={{B-convergence results for linearly implicit one step methods}},
  author={Strehmel, K and Weiner, R},
  journal={BIT Numerical Mathematics},
  volume={27},
  number={2},
  pages={264--281},
  year={1987},
  publisher={Springer}
}

@article{burrage2000order,
  title={{Order conditions of stochastic Runge--Kutta methods by B-series}},
  author={Burrage, Kevin and Burrage, Pamela M},
  journal={SIAM Journal on Numerical Analysis},
  volume={38},
  number={5},
  pages={1626--1646},
  year={2000},
  publisher={SIAM}
}

@article{yang2022new,
  title={{A new class of structure-preserving stochastic exponential Runge-Kutta integrators for stochastic differential equations}},
  author={Yang, Guoguo and Burrage, Kevin and Komori, Yoshio and Ding, Xiaohua},
  journal={BIT Numerical Mathematics},
  volume={62},
  number={4},
  pages={1591--1623},
  year={2022},
  publisher={Springer}
}

@article{besse2017high,
  title={{High Order Exponential Integrators for Nonlinear Schrödinger Equations with Application to Rotating Bose--Einstein Condensates}},
  author={Besse, Christophe and Dujardin, Genevi{\`e}ve and Lacroix-Violet, Ingrid},
  journal={SIAM Journal on Numerical Analysis},
  volume={55},
  number={3},
  pages={1387--1411},
  year={2017},
  publisher={SIAM}
}

@article{anton2018exponential,
  title={{Exponential integrators for stochastic Schr{\"o}dinger equations driven by It{\^o} noise}},
  author={Anton, Rikard and Cohen, David},
  journal={Journal of Computational Mathematics},
  pages={276--309},
  year={2018},
  publisher={JSTOR}
}

@article{komori2017weak,
  title={{Weak second order explicit exponential Runge--Kutta methods for stochastic differential equations}},
  author={Komori, Yoshio and Cohen, David and Burrage, Kevin},
  journal={SIAM Journal on Scientific Computing},
  volume={39},
  number={6},
  pages={A2857--A2878},
  year={2017},
  publisher={SIAM}
}

@article{debrabant2021runge,
  title={{Runge--Kutta Lawson schemes for stochastic differential equations}},
  author={Debrabant, Kristian and Kv{\ae}rn{\o}, Anne and Mattsson, Nicky Cordua},
  journal={BIT Numerical Mathematics},
  volume={61},
  number={2},
  pages={381--409},
  year={2021},
  publisher={Springer}
}

@article{bhatt2017structure,
  title={{Structure-preserving exponential Runge--Kutta methods}},
  author={Bhatt, Ashish and Moore, Brian E},
  journal={SIAM Journal on Scientific Computing},
  volume={39},
  number={2},
  pages={A593--A612},
  year={2017},
  publisher={SIAM}
}

@article{lawson1967generalized,
  title={{Generalized Runge-Kutta processes for stable systems with large Lipschitz constants}},
  author={Lawson, J Douglas},
  journal={SIAM Journal on Numerical Analysis},
  volume={4},
  number={3},
  pages={372--380},
  year={1967},
  publisher={SIAM}
}

@article{niegemann2007higher,
  title={{Higher-order time-domain simulations of Maxwell's equations using Krylov-subspace methods}},
  author={Niegemann, Jens and Tkeshelashvili, Lasha and Busch, Kurt},
  journal={Journal of Computational and Theoretical Nanoscience},
  volume={4},
  number={3},
  pages={627--634},
  year={2007},
  publisher={American Scientific Publishers}
}

@article{krogstad2005generalized,
  title={{Generalized integrating factor methods for stiff PDEs}},
  author={Krogstad, Stein},
  journal={Journal of Computational Physics},
  volume={203},
  number={1},
  pages={72--88},
  year={2005},
  publisher={Elsevier}
}

@article{ruemelin1982numerical,
  title={{Numerical treatment of stochastic differential equations}},
  author={R{\"u}emelin, W},
  journal={SIAM Journal on Numerical Analysis},
  volume={19},
  number={3},
  pages={604--613},
  year={1982},
  publisher={SIAM}
}

@article{bhatt2016compact,
  title={{A compact fourth-order L-stable scheme for reaction--diffusion systems with nonsmooth data}},
  author={Bhatt, Harish P and Khaliq, AQM},
  journal={Journal of Computational and Applied Mathematics},
  volume={299},
  pages={176--193},
  year={2016},
  publisher={Elsevier}
}

@article{yousuf2013efficient,
  title={{An efficient ETD method for pricing American options under stochastic volatility with nonsmooth payoffs}},
  author={Yousuf, M and Khaliq, AQM},
  journal={Numerical Methods for Partial Differential Equations},
  volume={29},
  number={6},
  pages={1864--1880},
  year={2013},
  publisher={Wiley Online Library}
}

@article{qin2025energy,
  title={Energy Stability and Maximum Principle of Skeletal Finite Element Method for Allen--Cahn Equation with Exponential Time Differencing Schemes},
  author={Qin, Mingze and Zhai, Qilong and Zhang, Ran},
  journal={Journal of Scientific Computing},
  volume={105},
  number={1},
  pages={1--29},
  year={2025},
  publisher={Springer}
}

@article{beylkin1998new,
  title={{A new class of time discretization schemes for the solution of nonlinear PDEs}},
  author={Beylkin, Gregory and Keiser, James M and Vozovoi, Lev},
  journal={Journal of computational physics},
  volume={147},
  number={2},
  pages={362--387},
  year={1998},
  publisher={Elsevier}
}

@article{shu1988efficient,
  title={Efficient implementation of essentially non-oscillatory shock-capturing schemes},
  author={Shu, Chi-Wang and Osher, Stanley},
  journal={Journal of computational physics},
  volume={77},
  number={2},
  pages={439--471},
  year={1988},
  publisher={Elsevier}
}

@article{ahmat2021compact,
  title={{Compact ETDRK scheme for nonlinear dispersive wave equations}},
  author={Ahmat, Muyassar and Qiu, Jianxian},
  journal={Computational and Applied Mathematics},
  volume={40},
  number={8},
  pages={286},
  year={2021},
  publisher={Springer}
}

@misc{zhangcharacterizing,
  author = {Zhang, Hong and Wang, Haifeng},
  title = {{Characterizing the Stabilization Size of a Third-Order One-Parameter ETDRK Scheme for the Swift--Hohenberg Equation}},
  year = {2024},
  doi = {10.13140/RG.2.2.13828.33927}
}

@article{fu2024higher,
  title={{Higher-order energy-decreasing exponential time differencing Runge-Kutta methods for gradient flows}},
  author={Fu, Zhaohui and Shen, Jie and Yang, Jiang},
  journal={Science China Mathematics},
  pages={1--20},
  year={2024},
  publisher={Springer}
}

@book{milstein2004stochastic,
  title={{Stochastic numerics for mathematical physics}},
  author={Milstein, Grigori N and Tretyakov, Michael V},
  volume={39},
  year={2004},
  publisher={Springer}
}

@article{wang2025novel,
  title={{A novel up to fourth-order equilibria-preserving and energy-stable exponential Runge-Kutta framework for gradient flows}},
  author={Wang, Haifeng and Sun, Jingwei and Zhang, Hong and Qian, Xu},
  journal={CSIAM Trans. Appl. Math.},
  volume={6},
  number={1},
  pages={106--147},
  year={2025}
}

@article{sun2023family,
  title={{A family of structure-preserving exponential time differencing Runge--Kutta schemes for the viscous Cahn--Hilliard equation}},
  author={Sun, Jingwei and Zhang, Hong and Qian, Xu and Song, Songhe},
  journal={Journal of Computational Physics},
  volume={492},
  pages={112414},
  year={2023},
  publisher={Elsevier}
}

@article{chen2025parallelization,
  title={{Parallelization of the Exponential Time Differencing Method for Solving Diffuse-Interface Models}},
  author={Chen, Meng-Huo and Wu, Yuanqing and Feng, Xiaoyu and Sun, Shuyu},
  journal={Journal of Scientific Computing},
  volume={104},
  number={3},
  pages={1--26},
  year={2025},
  publisher={Springer}
}

@book{kloeden2011numerical,
  title={{Numerical Solution of Stochastic Differential Equations}},
  author={Kloeden, P.E. and Platen, E.},
  isbn={9783540540625},
  lccn={92015916},
  series={Stochastic Modelling and Applied Probability},
  year={1991},
  publisher={Springer Berlin Heidelberg}
}

@article{nie2006efficient,
  title={{Efficient semi-implicit schemes for stiff systems}},
  author={Nie, Qing and Zhang, Yong-Tao and Zhao, Rui},
  journal={Journal of Computational Physics},
  volume={214},
  number={2},
  pages={521--537},
  year={2006},
  publisher={Elsevier}
}

@article{ta2015integration,
  title={{An integration factor method for stochastic and stiff reaction--diffusion systems}},
  author={Ta, Catherine and Wang, Dongyong and Nie, Qing},
  journal={Journal of computational physics},
  volume={295},
  pages={505--522},
  year={2015},
  publisher={Elsevier}
}

@article{komori2014stochastic,
  title={{A stochastic exponential Euler scheme for simulation of stiff biochemical reaction systems}},
  author={Komori, Yoshio and Burrage, Kevin},
  journal={BIT Numerical Mathematics},
  volume={54},
  pages={1067--1085},
  year={2014},
  publisher={Springer}
}

@article{woodfield2024strong,
  title={{Strong stability preservation for stochastic partial differential equations}},
  author={Woodfield, James},
  journal={arXiv preprint arXiv:2411.11172},
  year={2024}
}

@article{isherwood2018strong,
  title={{Strong stability preserving integrating factor Runge--Kutta methods}},
  author={Isherwood, Leah and Grant, Zachary J and Gottlieb, Sigal},
  journal={SIAM Journal on Numerical Analysis},
  volume={56},
  number={6},
  pages={3276--3307},
  year={2018},
  publisher={SIAM}
}

@book{lord2014introduction,
  title={{An introduction to computational stochastic PDEs}},
  author={Lord, Gabriel J and Powell, Catherine E and Shardlow, Tony},
  volume={50},
  year={2014},
  publisher={Cambridge University Press}
}

@article{hochbruck2010exponential,
  title={{Exponential integrators}},
  author={Hochbruck, Marlis and Ostermann, Alexander},
  journal={Acta Numerica},
  volume={19},
  pages={209--286},
  year={2010},
  publisher={Cambridge University Press}
}

@article{becker2016exponential,
  title={{An exponential Wagner--Platen type scheme for SPDEs}},
  author={Becker, Sebastian and Jentzen, Arnulf and Kloeden, Peter E},
  journal={SIAM Journal on Numerical Analysis},
  volume={54},
  number={4},
  pages={2389--2426},
  year={2016},
  publisher={SIAM}
}

@article{von2024exponential,
  title={{An exponential stochastic Runge--Kutta type method of order up to 1.5 for SPDEs of Nemytskii-type}},
  author={von Hallern, Claudine and Missfeldt, Ricarda and R{\"o}ssler, Andreas},
  journal={IMA Journal of Numerical Analysis},
  pages={drae064},
  year={2024},
  publisher={Oxford University Press}
}

@article{erdougan2019new,
  title={{A new class of exponential integrators for SDEs with multiplicative noise}},
  author={Erdo{\u{g}}an, Utku and Lord, Gabriel J},
  journal={IMA Journal of Numerical Analysis},
  volume={39},
  number={2},
  pages={820--846},
  year={2019},
  publisher={Oxford University Press}
}

@article{clancy2013use,
  title={{On the use of exponential time integration methods in atmospheric models}},
  author={Clancy, Colm and Pudykiewicz, Janusz A},
  journal={Tellus A: Dynamic Meteorology and Oceanography},
  volume={65},
  number={1},
  pages={20898},
  year={2013},
  publisher={Taylor \& Francis}
}
\bibliographystyle{abbrv}
\appendix

\subsection{Symmetric part of nested Stratonovich integration }\label{sec:sym_J}

Assume the following inductive hypothesis
\begin{align}
\operatorname{Sym}\left(J_{\mathrm{\b j}}\right) = \frac{1}{n!} \sum_{\sigma \in S_n} J_{j_{\sigma(1)}, \ldots, j_{\sigma(n)}} = \frac{1}{n!}J_{j_1}J_{j_2}...J_{j_n}, \quad n =  1,2,...N.\label{eq:inductive hypothesis}
\end{align}

We shall recall the shuffle product formula for nested integration:
\begin{align}
J_{j_{1},...j_{m}}J_{k_{1}...k_{n}} = \sum_{\sigma \in \text{Sh}(m,n)}J_{\sigma(j_{1},...,j_{m};k_{1},...,k_{n})} 
\end{align}
where $\text{Sh}(m,n)$ denotes the set of shuffles length $m+n$, preserving the order of $j_{1},...j_{m}$ and $k_{1},...,k_{n}$. A subcase of this formula gives the following identity:
\begin{align}
J_{j_{1},...j_{n}}J_{j_{n+1}} = \sum_{\sigma \in \text{Sh}(n,1)}J_{\sigma(j_{1},...,j_{n};j_{n+1})} = \sum_{l=0}^{n} J_{j_{1},...,j_{l},j_{n+1},j_{l+1},...,j_{n}}.\label{eq:key identity}
\end{align}
$n=1$ is true trivially, and
\cref{eq:key identity} can be used with $n=2$ to give $J_{j_1}J_{j_2} = J_{j_1,j_2}+J_{j_2,j_1}$ then in the case $n=3$ one can use \cref{eq:key identity} recursively,
\begin{align}
J_{j_1}J_{j_2}J_{j_3} =J_{j_1,j_2}J_{j_3}+J_{j_2,j_1}J_{j_3}=\sum_{\sigma\in S_3}J_{j_{\sigma(1)},j_{\sigma(2)},j_{\sigma(3)}}.
\end{align}

Assume \cref{eq:inductive hypothesis} up to finite $N$ as an inductive hypothesis, using \cref{eq:key identity} we have
\begin{align}
J_{j_1}....J_{j_{N}}J_{j_{N+1}} &=  \left(\sum_{\sigma \in S_N} J_{j_{\sigma(1)}, \ldots, j_{\sigma(N)}} \right)J_{j_{N+1}}=\sum_{\sigma \in S_N} \left( J_{j_{N+1},j_{\sigma(1)}, \ldots, j_{\sigma(N)}} +...+ J_{j_{\sigma(1)}, \ldots, j_{\sigma(N)},j_{N+1}} \right)\\
&=  \sum_{\sigma\in S_{N+1}}J_{j_{\sigma(1)}, \ldots, j_{\sigma(N+1)}} 
\end{align}
upon dividing by $(N+1)!$, the inductive hypothesis is proven for $n+1$.

\subsection{Transfer of symmetrisation}\label{app:Transfer of symmetrisation}
\begin{lemma}\label{lem:Transfer of symmetrisation}
\begin{align}
\sum_{0\leq j_1,...,j_p\leq M} L^{j_1}...L^{j_p} \pi^k(q) \operatorname{Sym}(J_{j_1...j_p}) = \sum_{0\leq j_1,...,j_p\leq M} \operatorname{Sym}(L^{j_1}...L^{j_p} )\pi^k(q) J_{j_1...,j_p}
\end{align} 
\end{lemma}
\begin{proof}
Use the definition of $\operatorname{Sym}$ as the average over all permutations in the symetric group, then for a fixed permutation $\sigma$, reindex $(i_1,...,i_p) = (j_{\sigma(1)},...,j_{\sigma(p)})$ such that $(i_{\sigma^{-1}(1)},...,i_{\sigma^{-1}(p)}) = (j_1,...,j_p)$, and noting summation over all $\sigma \in S_p$ is the same as summing over all $\sigma^{-1} \in S_p$ one can verify
\begin{align}
   \sum_{0\leq j_1,...,j_p\leq M} L^{j_1}...L^{j_p} \pi^k(q) \operatorname{Sym}(J_{j_1...j_p}) &=  \frac{1}{p!}\sum_{0\leq j_1,...,j_p\leq M} L^{j_1}...L^{j_p}\pi^k(q) \sum_{\sigma\in S_p}J_{j_{\sigma(1)}...j_{\sigma(p)}}\\
   & =  \frac{1}{p!}\sum_{\sigma\in S_p}\sum_{0\leq i_1,...,i_p\leq M} L^{i_{\sigma^{-1}(1)}}...L^{i_{\sigma^{-1}(p)}}\pi^k(q) J_{i_{1}...i_{p}}\\
   &= \sum_{0\leq i_1,...,i_p\leq M} \operatorname{Sym}(L^{i_1}...L^{i_p})\pi^k(q) J_{i_{1}...i_{p}}.
\end{align}
\end{proof}

\subsection{Derivation of stochastic travelling wave solutions to the KdV equation}\label{sec: stochastic travelling wave solutions to the KdV equation}

By the Chain rule, one can compute 
\begin{align}
d_t u = d_t u(t,X(x,t)) = \partial_t u + \partial_{X}u \circ d_t X = \partial_t u -a u_{x}\circ dW.
\end{align}
This allows \cref{eq:KdV constant noise} to be transformed into the deterministic KdV equation
\begin{align}
\partial_t u + uu_x+u_{xxx} = 0.\label{deterministic:kvd}
\end{align} 
Seeking a travelling wave solution of the form 
\begin{align}
u = f(\xi) \quad \text{where}\quad \xi = x - \beta t.
\end{align}
We can deduce that the stochastic travelling wave solution of \cref{deterministic:kvd} satisfies the ODE
\begin{align}
-\beta f'(\xi) + f(\xi)f'(\xi) + f^{'''}(\xi) = 0.\label{eq:ode_travellingwave}
\end{align}
Assuming vanishing contribution of $u,u_x,u_{xx}$ at positive and negative infinity, we can derive the travelling wave solutions of \cref{eq:ode_travellingwave} by
integrating w.r.t. $\xi$ and then multiplying by $f'(\xi)$ and then integrating again w.r.t. $\xi$. using vanishing contributions at the endpoints so constants of integration disappear. Then using the separation of variables one attains
\begin{align}
\int_0^{F} \frac{1}{f\sqrt{\beta-\frac{f}{3}}} df = \int_0^{\xi} d\xi
\end{align}
upon letting $f = 3\beta \operatorname{sech}^2(g)$, $df = -6\beta\tanh(g) \operatorname{sech}^2(g)dg$ one can derive the travelling wave solution \cref{eq:stochastic travelling wave} to the stochastic KdV equation \cref{eq:KdV constant noise}.

\subsection{SETDRK methods for Itô systems, introduction and literature}\label{sec:SETDRK methods for Ito systems}

We review Itô-SETDRK schemes, in the simplified setting of one-dimensional Itô-SDEs, 
\begin{align}
d_t u  = \left[ N(u,t) + L u \right] dt +  \left[ g(u,t)\right] dW_t.\label{eq:ito-sde}
\end{align}

Where $u$ is a one dimensional stochastic process, $N(u,t)$ is a nonlinear function, $L$ is a linear function, and $g(u,t)$ is a nonlinear function, $dW_t$ is a one dimensional Brownian motion. By multiplying \cref{eq:ito-sde} through by the integrating factor $e^{-Lt}$, and integrating between $t^n$ and $t^{n+1}$ one may write \cref{eq:ito-sde} in integral form as follows 
\begin{align}
 u^{n+1}  = e^{L \Delta t} u^{n}  + e^{L \Delta t}\int_{0}^{\Delta t} e^{-L s}  N(u(t_n + s),t_n + s) ds +  e^{L (\Delta t+t^{n})}\int_{t^{n}}^{t^{n+1}} e^{-L s}  g(u(s), s)  dW_s.\label{eq:integral-ito}
\end{align}
where $\Delta t = t^{n+1}-t^{n}$. Stochastic exponential time differencing methodology arises when one considers approximating the above integral representation of the SDE. 
\begin{method}[SIFEM- Stochastic integrating factor Euler Maruyama]\label{method:SIFEM}
Approximating both the integrands in \cref{eq:integral-ito} by evaluating the integrands at the leftmost temporal point $t_n$, leads to the following scheme 
\begin{align}
u^{n+1}  =  e^{L \Delta t} (u^{n}  +  \Delta t N(u_n,t_n) +  g(u_n,t_n) \Delta W ). 
\end{align}
\end{method}
 this converges to the Itô SDE and treats the linear operator with an exponential operator. This scheme is referred to as SETDM0 in \cite{lord2013stochastic}.
Whilst this scheme does arise from approximating the integrals in \cref{eq:integral-ito} one can also naturally derive this scheme as a stochastic integrating factor Euler Maruyama scheme  SIFEM.

\begin{method}[SETDM$10$- Drift Exponential Time Differencing]\label{method:DETD}
Another approach following \cite{cox2002exponential}, approximates $N$ in \cref{eq:integral-ito} by a constant $N_n = N(u(t_{n}),t_n)$ and integrates $e^{-Ls}ds$ exactly to give the following scheme
\begin{align}
 u^{n+1}  =  e^{L \Delta t} u^{n}  + N_n( e^{L\Delta t} - 1)L^{-1} + e^{L\Delta t} g_n \Delta W.
\end{align}  
\end{method}
This can be found under the name SETDM1 scheme in \cite{lord2013stochastic}. 
\begin{method}[SETDM01-Drift-Diffusion Stochastic exponential time differencing]\label{method: Drift-Diffusion Stochastic exponential time differencing}
In an analogous manner to the deterministic setting, if 
one instead approximates $g = g_n$ by a constant, one cannot directly compute the stochastic integral but instead it's distribution can be attained through Itô's Isometry
$$\int_{t^n}^{t^{n+1}} e^{-L s} d W \sim \mathcal{N}\left(0, \int_{t^n}^{t^{n+1}}\left(e^{-L s}\right)^2 d s\right).$$

Sampling from this distribution leads to the following first order Stochastic Exponential time differencing scheme:
\begin{align}
 u^{n+1}  = e^{L \Delta t} u^{n}  + N_n (e^{L\Delta t } - 1) L^{-1} +  g_n \sqrt{\frac{1}{2 }\left(e^{2 L \Delta t} -1 \right)L^{-1}} \Delta Z, \quad \Delta Z \sim \mathcal{N}(0,1)\label{eq:SETDM01}
\end{align}
\end{method}
Such schemes can be found for example in \cite{jardak2010comparison,jentzen2009overcoming}.


Consider the Taylor expansion of the exponential operators in \cref{method:SIFEM}, \cref{method:DETD}, \cref{method: Drift-Diffusion Stochastic exponential time differencing}, these are respectively
\begin{align}
    e^{L\Delta t} &= (1+ L\Delta t + 1/2(L\Delta t)^2 + ...)  \\
(e^{L\Delta t} -1)L^{-1} &= \Delta t +  1/2 L\Delta t^2 - ... \\
\sqrt{(2 L)^{-1}\left(e^{2 L \Delta t} -1 \right)} &= \sqrt{\Delta t + L\Delta t^2+...} 
\end{align}
and allows all schemes to be convergent to the Itô system and at leading order, equal to the Euler-Maruyama scheme. Therefore, with this as motivation, if one finds another perhaps more convenient exponential operator, agreeing at leading order to the desired SDE, one can derive a new method. 
Consider modifying \cref{method: Drift-Diffusion Stochastic exponential time differencing}, with the following transformation
\begin{align}
    \sqrt{\frac{1}{2 L }\left(e^{2 L (\Delta t)} -1 \right)}  \mapsto \frac{e^{L\Delta t}-1}{L\sqrt{\Delta t}},
\end{align}
these expressions agree at leading order and this results in the following \cref{method:CSETDRK1}.
\begin{method}[Convenient stochastic exponential time differencing Itô]\label{method:CSETDRK1}
\begin{align}
 u^{n+1}  =  e^{L \Delta t} u^{n}  + N_n\frac{e^{L\Delta t}-1}{L} + g_n \frac{ e^{L\Delta t}-1}{L\Delta t} \Delta W,
\end{align} 
\end{method}
The motivation is that this can be rearranged into
\begin{align}
 u^{n+1}  =  e^{L \Delta t} u^{n}  + \frac{e^{L\Delta t}-1}{L}\left( N_n + g_n\frac{\Delta W}{\Delta t} \right), 
\end{align} 
requires little change to any deterministic ETDRK1 codes, whilst converges to the Itô SPDE with strong order $1/2$, whilst benefiting from the exponential treatment of the linear operator. This scheme is found in \cite{komori2014stochastic} as the stochastic exponential Euler (SEE) scheme, derived under different considerations. This observation also leads to the ability to generate SETDRK methods such as those in this paper.

\subsection{CPU-Cost for a given tolerance}\label{sec:CPU-Cost for a given tolerance}

SETDRK and SIFRK methods can be run at larger timesteps than SRK methods, so can be significantly more cost effective at achieving a given user tolerance when working with stiff SDE/SPDE's.

In \cref{fig:non-commuting} we consider the non-commuting basis of noise, and sequentially refine the timestep until a relative tolerance of $10^{-4}$ is reached. In \cref{fig:non-commuting} the SETDRK4 method is over 4000x more efficient than SSP22. In \cref{fig:non-commuting} SETDRK22, and eSSPIFRK22 are twice as cost effective as SETDRK4 due to stage cost increase with same convergence order. Indicating 2nd order exponential integrators are the most efficient in the non-commuting stiff case. 

In \cref{fig:drift-commuting}, we perform the same experiment but with drift-commuting basis, SETDRK4 is the most cost effective owing to its increased stability and increased accuracy/convergence in the drift commuting setting, which outweighs the increased stage cost.

\begin{figure}[h]
\centering
\begin{subfigure}{0.47\textwidth}
\centering
    \includegraphics[width=0.95\linewidth]{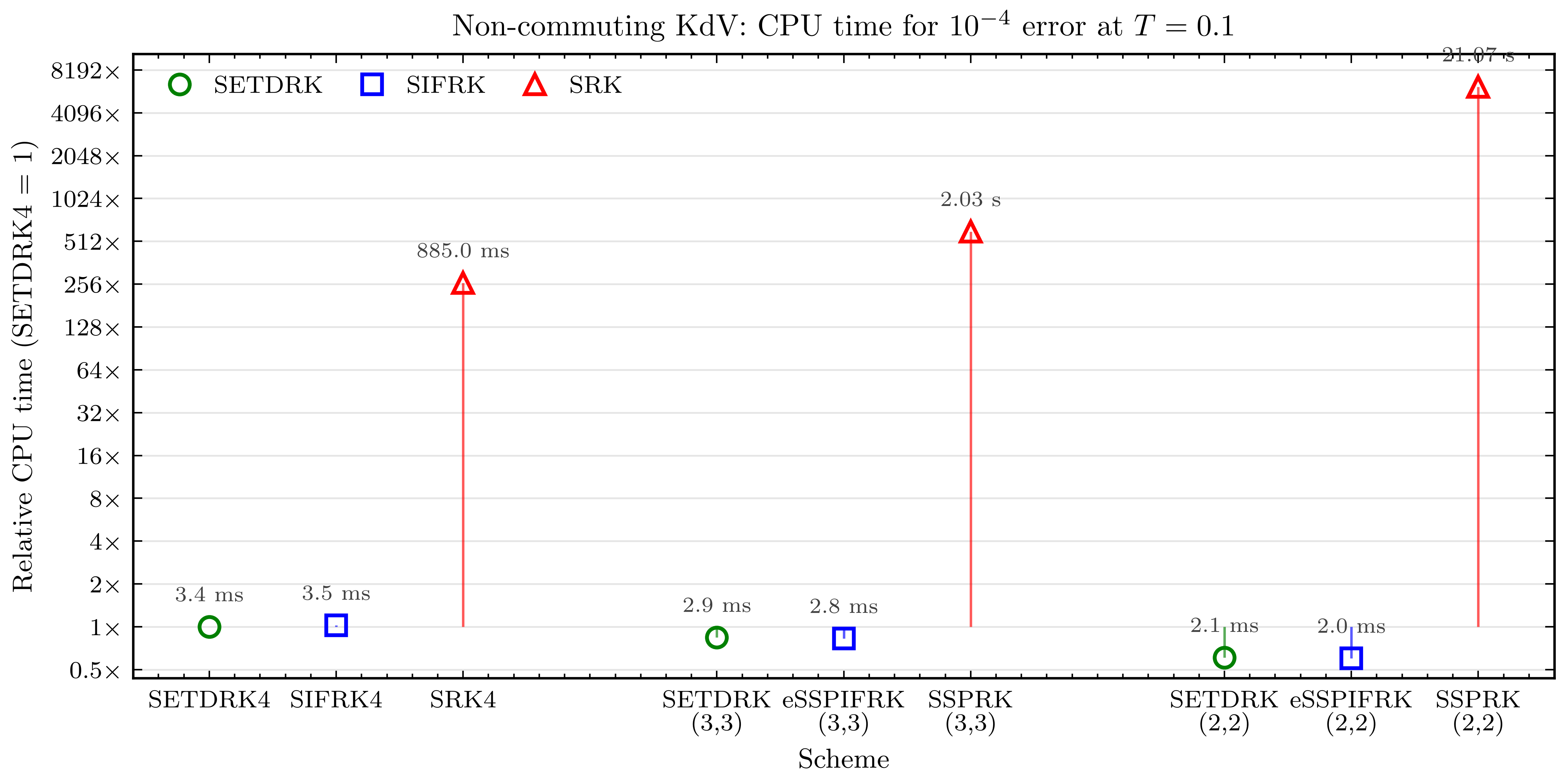}
    \caption{Non-commuting basis}
    \label{fig:non-commuting}
\end{subfigure}
\begin{subfigure}{0.47\textwidth}
\centering
    \includegraphics[width=0.95\linewidth]{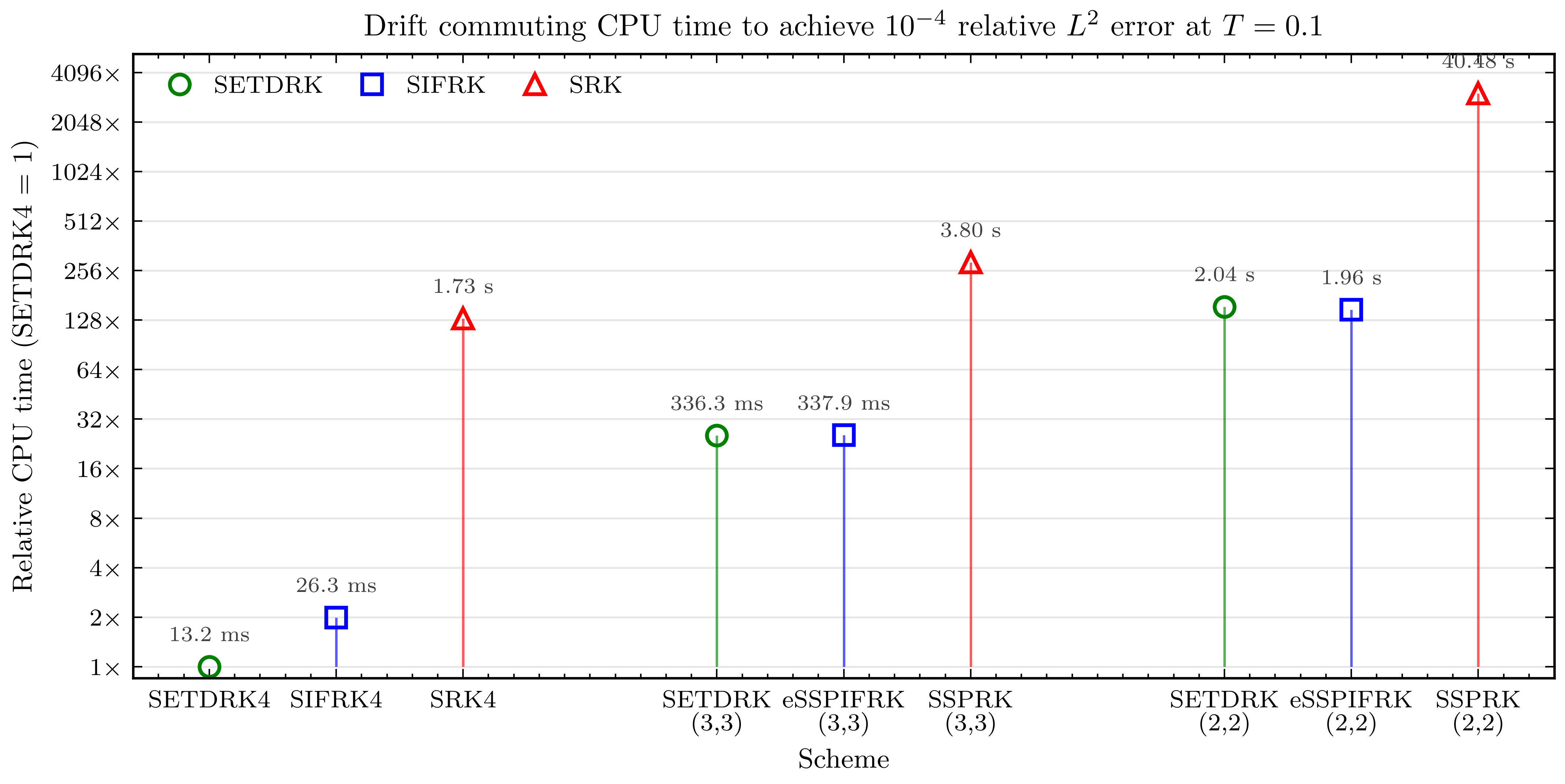}
    \caption{Drift commuting basis}
    \label{fig:drift-commuting}
\end{subfigure}
\label{fig:cpu-timing}
\caption{Relative CPU-time for a given tolerance. Y-axis demonstrates the cost multiplier on a logarithmic scale where $1$x is the same cpu cost as SETDRK4. In both \cref{fig:drift-commuting,fig:non-commuting} SETDRK and SIFRK methods are orders of magnitude more effective than corresponding SRK schemes. In the non-commuting case \cref{fig:non-commuting}, SETDRK22 and eSSPIFRK22 are the cheapest at nearly $0.5$ the cost of SETDRK4, in the drift commuting case \cref{fig:drift-commuting}, we see SETDRK4 being 128x times more efficient as SETDRK22 and eSSPIFRK22. }
\end{figure}
\end{document}